\documentclass[11pt,oneside]{amsart}
\usepackage[top=25mm, bottom=25mm, left=20mm, right=20mm]{geometry}
\usepackage{amsfonts,amssymb,amsmath,amsthm,amsxtra}
\usepackage[T1]{fontenc}
\usepackage[utf8]{inputenc}
\usepackage{bm}
\usepackage{mathtools}
\usepackage{dsfont}
\usepackage{mathrsfs}
\usepackage{enumitem}
\allowdisplaybreaks
\usepackage{amsmath}

\usepackage{graphics}
\usepackage{hyperref}

\numberwithin{equation}{section}
\newtheorem{theorem}[equation]{Theorem}

\newtheorem{lemma}[equation]{Lemma}
\newtheorem{proposition}[equation]{Proposition}
\theoremstyle{definition}

\newcommand{\Mod}[1]{\ (\mathrm{mod}\ #1)}
\DeclareMathOperator\supp{supp}
\DeclareMathOperator\id{id}
\DeclareMathOperator\poly{poly}
\DeclareMathOperator\Lip{Lip}
\DeclareMathOperator\F{F}
\DeclareMathOperator\G{G}

\begin{document}

\title{Pointwise convergence of ergodic averages along quadratic bracket polynomials}

\author{Leonidas Daskalakis} 
\address[Leonidas Daskalakis]{Institute of Mathematics,
Polish Academy of Sciences,
\'Sniadeckich 8,
00-656 Warsaw, Poland \& Institute of Mathematics, University of Wroc\l aw, Plac Grunwaldzki 2, 50-384 Wroc\l aw, Poland}
\email{ldaskalakis@impan.pl}
\maketitle\begin{abstract}
We establish a pointwise convergence result for ergodic averages modeled along orbits of the form $(n\lfloor n\sqrt{k}\rfloor)_{n\in\mathbb{N}}$, where $k$ is an arbitrary positive rational number with $\sqrt{k}\not\in\mathbb{Q}$. Namely, we prove that for every such $k$, every measure-preserving system $(X,\mathcal{B},\mu,T)$ and every $f\in L^{\infty}_{\mu}(X)$, we have that 
\[
\lim_{N\to\infty}\frac{1}{N}\sum_{n=1}^Nf(T^{n\lfloor n\sqrt{k}\rfloor}x)\quad\text{exists for $\mu$-a.e. $x\in X$.}
\]
Notably, our analysis involves a curious implementation of the circle method developed for analyzing exponential sums with phases $(\xi n \lfloor n\sqrt{k}\rfloor)_{1\le n\le N}$ exhibiting arithmetical obstructions beyond rationals with small denominators, and is based on the Green and Tao's result on the quantitative behaviour of polynomial orbits on nilmanifolds \cite{BG}. For the case $k=2$ such a circle method was firstly employed for addressing the corresponding Waring-type problem by Neale \cite{NL}, and their work constitutes the departure point of our considerations.
\end{abstract}
\section{Introduction}
The main result of the present work is establishing pointwise convergence for ergodic averages modeled along $(n\lfloor n\sqrt{k}\rfloor)_{n\in\mathbb{N}}$ for $k\in\mathbb{Q}_{>0}\footnote{where $\mathbb{Q}_{>0}\coloneqq \mathbb{Q}\cap(0,\infty)$.}$ such that $\sqrt{k}\not\in\mathbb{Q}$, see Theorem~$\ref{Main}$. Before making precise formulations, let us make some brief prefatory remarks.
\subsection{Introductory remarks}
Establishing pointwise convergence for ergodic averages along a sparse sequence of natural numbers $\mathfrak{a}=(a_n)_{n\in\mathbb{N}}$, which in its simplest formulation amounts to proving that for every probability space $(X,\mathcal{B},\mu)$ equipped with an invertible measure-preserving transformation $T\colon X\to X$ and every $f\in L^{\infty}_{\mu}(X)$ the following limit exists $\mu$-a.e.
\begin{equation}\label{anaverage}
\lim_{N\to\infty}\frac{1}{N}\sum_{n=1}^Nf(T^{a_n}x)\text{,}
\end{equation}
has been the object of intense study in recent decades. The first results in this direction appeared in the late 1980s, when Bourgain in a series of seminal works \cite{bg2,bg3,Bourgain} established pointwise convergence for the averages \eqref{anaverage} along polynomial and prime orbits, introducing numerous new ideas and laying the foundation for furthering our understanding of pointwise convergence phenomena in ergodic theory.    The now-standard approach, employed in the aforementioned works, involves establishing certain quantitative estimates (for example of the form \eqref{L2osc}) which imply the pointwise convergence of such averages. This has the advantage of allowing one to work on the integer shift system, namely $(\mathbb{Z},\mathcal{P}(\mathbb{Z}),|\cdot|,x\to x-1)$, where $|\cdot|$ denotes the counting measure, since through Calder\'on's transference principle, once one establishes such estimates on the integer shift system setting, they can immediately obtain the corresponding ones for arbitrary $\sigma$-finite dynamical systems. 

Since then, substantial progress has been made in our understanding of the pointwise limiting behaviour of ergodic averages of the form \eqref{anaverage} for a variety of sparse sequences $\mathfrak{a}=(a_n)_{n\in\mathbb{N}}$. In this context,  loosely speaking, all pointwise convergence results appearing in the literature seem to be naturally split into two categories which we describe below.
 
The operators \eqref{anaverage} considered for the integer shift system amount to the following convolution operator
\begin{equation}\label{anavop}
A^{(\mathfrak{a})}_{N}f(x)\coloneqq\frac{1}{N}\sum_{n=1}^Nf(x-a_n)\text{,}
\end{equation}
which in the frequency side becomes $\mathcal{F}_{\mathbb{Z}}[A^{(\mathfrak{a})}_{N}f](\xi)=m^{(\mathfrak{a})}_{N}(\xi)\mathcal{F}_{\mathbb{Z}}(\xi)$, where
\begin{equation}\label{genmultiplier}
m^{(\mathfrak{a})}_{N}(\xi)\coloneqq\frac{1}{N}\sum_{n=1}^Ne(\xi a_n)\text{.}
\end{equation} 
The behaviour of the above exponential sums plays a crucial role in the proof of pointwise convergence of ergodic averages modeled along $\mathfrak{a}=(a_n)_{n\in\mathbb{N}}$ and the sequences considered in the context of such problems in the literature can be naturally divided into the following two categories depending on which frequencies $\xi\in\mathbb{T}$ do not introduce sufficient cancellation in \eqref{genmultiplier} to make the exponential sum adequately small, which we will refer to as problematic frequencies from now on.

 The first category includes sequences $\mathfrak{a}$ for which the only obstruction is the trivial one, namely, the only problematic frequency is $\xi\equiv 0\Mod{1}$. Usually, in such a case, the exponential sum \eqref{genmultiplier} can be directly compared with a smoothly weighted variant of the complete exponential sum, see \cite{ncfull,mirek1,LDWT11,frprimes}.

The second category includes the sequences for which the problematic frequencies are rationals with a small denominator (or appropriately dilated variants, see for example section~8 in \cite{Bourgain}). In this case the situation is more delicate and the Hardy--Littlewood circle method is required \cite{bg2,bg3,Bourgain,primeevaluatedpolyn,JF}, which again, in the context of pointwise convergence of ergodic averages was first used by Bourgain.

In the present work we establish pointwise convergence for ergodic averages modeled along quadratic bracket polynomials of a certain form, which to the best of the author's knowledge is the first result falling into neither category.
\subsection{Statement of results}
The main result is the following.
\begin{theorem}[Pointwise ergodic theorem along $(n\lfloor n\sqrt{k}\rfloor)_{n\in\mathbb{N}}$]\label{Main}Assume $k\in\mathbb{Q}_{>0}$ is such that $\sqrt{k}\not\in\mathbb{Q}$. Let $(X,\mathcal{B},\mu)$ be a probability space and $T\colon X\to X$ an invertible $\mu$-invariant transformation. Then for every $f\in L^{\infty}_{\mu}(X)$ we have that
\begin{equation}\label{mainconv}
\lim_{N\to\infty}\frac{1}{N}\sum_{n=1}^Nf(T^{n\lfloor n\sqrt{k}\rfloor}x)\quad\text{exists for $\mu$-a.e. $x\in X$.}
\end{equation} 
\end{theorem}
We note that if $\sqrt{k}\in\mathbb{Q}$, then the pointwise convergence of the ergodic averages \eqref{mainconv} can be easily deduced by passing to arithmetic progressions and appealing to Bourgain's pointwise result along the squares \cite{Bourgain}. We deduce the above theorem by establishing certain $2$-oscillation estimates, see Theorem~$\ref{osc}$, and in order to make precise formulations, let us introduce the relevant notation. For every $Y\subseteq X\subseteq \mathbb{R}$ with $|X|>2$ and $J\in\mathbb{N}$,  let $\mathfrak{S}_J(X)\coloneqq\{\{I_0<\dots<I_J\}\subseteq X\}$, i.e.: $\mathfrak{S}_J(X)$ contains all increasing sequences of length $J+1$ taking values in $X$. For every family of complex-valued functions $(a_t(x):\,t\in X)$ let
\begin{equation}\label{oscdef}
O^2_{I,J}(a_t(x):\,t\in Y)\coloneqq\Big(\sum_{j=0}^{J-1}\sup_{t\in [I_j,I_{j+1})\cap Y}|a_t(x)-a_{I_j}(x)|^2 \Big)^{1/2}\text{.}
\end{equation}
We refer the reader to subsections~2.6 and 2.7 in \cite{MOE} for the basic properties of oscillations. 
\begin{theorem}[Oscillation estimates along lacunary scales on $L^2_{\mu}(X)$]\label{osc}
Assume $k\in\mathbb{Q}_{>0}$ is such that $\sqrt{k}\not\in\mathbb{Q}$ and let $(X,\mathcal{B},\mu)$ be a $\sigma$-finite measure space and $T\colon X\to X$ an invertible $\mu$-invariant transformation. For every $t\in[1,\infty)$ and $f\colon X\to\mathbb{C}$ let
\begin{equation}
\mathbf{A}_{t;k}f(x)=\frac{1}{\lfloor t\rfloor}\sum_{n\le t}f(T^{n\lfloor n\sqrt{k}\rfloor}x)\text{.}
\end{equation}
Then for every $\lambda\in(1,2]$ there exists a positive constant $C=C(k,\lambda)$ such that for every $f\in L^2_{\mu}(X)$ we have
\begin{equation}\label{L2osc}
\sup_{J\in\mathbb{N}}\sup_{I\in\mathfrak{S}_J(\mathbb{N}_0)}\|O^2_{I,J}(\mathbf{A}_{\lambda^n;k}f:n\in\mathbb{N}_0)\|_{L_{\mu}^2(X)}\le C\|f\|_{L_{\mu}^2(X)}\text{.}
\end{equation}
\end{theorem}
An immediate consequence is the following theorem, which clearly implies Theorem~$\ref{Main}$.
\begin{theorem}\label{stronger}
Assume $k\in\mathbb{Q}_{>0}$ is such that $\sqrt{k}\not\in\mathbb{Q}$ and let $(X,\mathcal{B},\mu)$ be a $\sigma$-finite measure space and $T\colon X\to X$ an invertible $\mu$-invariant transformation. For every $p\in[2,\infty)$ and $f\in L^p_{\mu}(X)$ we have that  
\begin{equation}\label{mainconv1}
\lim_{N\to\infty}\mathbf{A}_{N;k}f(x)\quad\text{exists for $\mu$-a.e. $x\in X$.}
\end{equation} 
\end{theorem}
Theorem~$\ref{osc}$ implies Theorem~$\ref{stronger}$ straightforwardly, and for the sake of completeness, we provide a brief explanation here. Firstly, the estimate \eqref{L2osc} immediately yields that for every $\lambda\in(1,2]$ and every $f\in L_{\mu}^{2}(X)$ we have that
\begin{equation}\label{mainconv1}
\lim_{n\to\infty}\mathbf{A}_{\lambda^n;k}f(x)\quad\text{exists for $\mu$-a.e. $x\in X$,}
\end{equation} 
see for example Proposition~2.8 in \cite{MOE}. By applying this for $\lambda=2^{\frac{1}{2^d}}$, $d\in\mathbb{N}_0$ and by taking into account the positivity of the operator $\mathbf{A}_{N;k}$ one can easily deduce that for every $f\in L^2_{\mu}(X)$ we have that
  \begin{equation}\label{mainconv2}
\lim_{N\to\infty}\mathbf{A}_{N;k}f(x)\quad\text{exists for $\mu$-a.e. $x\in X$.}
\end{equation}
The estimate \eqref{L2osc} also implies that the maximal function associated with the operators $\mathbf{A}_{N;k}$ is bounded on $L^2_{\mu}(X)$, since we have
\begin{multline}\label{suptoosc}
\big\|\sup_{N\in\mathbb{N}}|\mathbf{A}_{N;k}f|\big\|_{L^2_{\mu}(X)}\lesssim \big\|\sup_{n\in\mathbb{N}_0}|\mathbf{A}_{2^n;k}f|\big\|_{L^2_{\mu}(X)}
\\
\lesssim \sup_{n\in\mathbb{N}_0}\|\mathbf{A}_{2^n;k}f\|_{L^2_{\mu}(X)}+\sup_{J\in\mathbb{N}}\sup_{I\in\mathfrak{S}_J(\mathbb{N}_0)}\|O^2_{I,J}(\mathbf{A}_{2^n;k}f:n\in\mathbb{N}_0)\|_{L^2_{\mu}(X)}\lesssim_k\|f\|_{L^2_{\mu}(X)}\text{,}
\end{multline}
where for the second estimate one may use for example Proposition~2.6 in \cite{MOE}. Secondly, we trivially have $\|\sup_{N\in\mathbb{N}}|\mathbf{A}_{N;k}f|\|_{L^{\infty}_{\mu}(X)}\le \|f\|_{L^{\infty}_{\mu}(X)}$ and by Marcinkiewicz interpolation theorem we obtain the estimate $\|\sup_{N\in\mathbb{N}}|\mathbf{A}_{N;k}f|\|_{L^{p}_{\mu}(X)}\lesssim_{p,k} \|f\|_{L^{p}_{\mu}(X)}$ for $p\in[2,\infty]$. Finally, we note that such an estimate implies that the set of functions in $L^p_{\mu}(X)$, $p\in[2,\infty)$, for which the limit \eqref{mainconv2} exists is closed in $L^p_{\mu}(X)$ and since we have already established pointwise convergence on $L^2_{\mu}(X)$ and $L^2_{\mu}(X)\cap L^p_{\mu}(X)$ is dense in $L^p_{\mu}(X)$, we immediately obtain Theorem~$\ref{stronger}$.
\subsection{Strategy}
As explained above, our main theorems can be derived by establishing Theorem~$\ref{osc}$ and, by Calder\'on's transference principle, it suffices to establish the estimate \eqref{L2osc} for the integer shift system. The averaging operator becomes
\begin{equation}\label{avop}
A_{t;k}f(x)\coloneqq\frac{1}{\lfloor t\rfloor}\sum_{n\le t}f(x-n\lfloor n\sqrt{k}\rfloor)\text{,}
\end{equation}
and in the frequency side we have $\mathcal{F}_{\mathbb{Z}}[A_{t;k}f](\xi)=m_{t;k}(\xi)\mathcal{F}_{\mathbb{Z}}(\xi)$, where $m_{t;k}(\xi)\coloneqq\frac{1}{\lfloor t\rfloor}\sum_{n\le t}e(\xi n\lfloor n\sqrt{k}\rfloor)$. The first part towards establishing the estimate \eqref{L2osc} is performing a suitably adapted circle method for the above exponential sums, culminating in the proof of a key intermediate approximation result, see Proposition~$\ref{cmethod}$. This part utilizes and extends ideas appearing in \cite{NL}, where a circle method for the case $k=2$ is developed, and the reader is encouraged to compare our sections~$\ref{majsection}$ and $\ref{minsection}$ with sections 4.3.3 and 4.2.1 and from \cite{NL} respectively. The second part amounts to establishing the desired 2-oscillation estimates for the approximant. We briefly elaborate on these steps below.
\subsubsection{Minor arc estimates}\label{minarcstrat}
We argue using the $3$-dimensional Heisenberg nilmanifold, namely, we work on $G/\Gamma$ where
\[
G=\Big\{
\left(\begin{smallmatrix}1&x&z
\\
0&1&y\\
0&0&1
\end{smallmatrix}\right)
,\,\,x,y,z\in\mathbb{R}\Big\}\quad\text{and}\quad\Gamma=\Big\{\left(\begin{smallmatrix}1&x&z
\\
0&1&y\\
0&0&1
\end{smallmatrix}\right),\,\,x,y,z\in\mathbb{Z}\Big\}\text{.}
\]
Note that $\mathcal{F}=\Big\{
\left(\begin{smallmatrix}1&x&z
\\
0&1&y\\
0&0&1
\end{smallmatrix}\right)
,\,\,x,y,z\in[0,1)\Big\}$ is a fundamental domain. We will appeal to Green and Tao's quantitative Leibman theorem \cite{BG}, see Theorem~$\ref{QLT}$ for a precise formulation, by viewing the average $\mathbb{E}_{n\in[N]}e(n\lfloor n\sqrt{k}\rfloor)$ as the average value of an appropriate function $F\colon G/\Gamma\to\mathbb{C}$ sampled over a polynomial sequence in $G/\Gamma$.  More precisely, for 
\begin{equation*}\label{definitionofg1}
g_{\xi;k}(n)\coloneqq\begin{pmatrix}
1 & -\xi n & 0\\
0 & 1 & \sqrt{k} n\\
0 & 0 & 1 
\end{pmatrix}
\end{equation*}
and for the unique function $F\colon G/\Gamma\to\mathbb{C}$ such that $F\Big(\Big(\begin{smallmatrix}
1 & x & z\\
0 & 1 & y\\
0 & 0 & 1 
\end{smallmatrix}\Big)\Gamma\Big)=e(z)$ for all $x,y,z\in[0,1)$, we see that 
\begin{multline}
\mathbb{E}_{n\in[N]}F(g_{\xi;k}(n)\Gamma)=\mathbb{E}_{n\in[N]}F\bigg(\bigg(\begin{smallmatrix}
1 & -\xi n & 0\\
0 & 1 & \sqrt{k} n\\
0 & 0 & 1 
\end{smallmatrix}\bigg)\Gamma\bigg)
\\
=\mathbb{E}_{n\in[N]}F\bigg(\bigg(\begin{smallmatrix}
1 & \{-\xi n\} & \{\xi n\lfloor n\sqrt{k}\rfloor\}\\
0 & 1 & \{\sqrt{k} n\}\\
0 & 0 & 1 
\end{smallmatrix}\bigg)\Gamma\bigg)=\mathbb{E}_{n\in[N]}e(\xi n\lfloor n\sqrt{k}\rfloor)\text{,}
\end{multline}
where we used the following identity $
\Big(\begin{smallmatrix}
1 & x & z\\
0 & 1 & y\\
0 & 0 & 1
\end{smallmatrix}\Big)\Gamma=
\Big(\begin{smallmatrix}
1 & \{x\} & \{z-x\lfloor y\rfloor\}\\
0 & 1 & \{y\}\\
0 & 0 & 1
\end{smallmatrix}\Big)\Gamma$. Although $F$ is not continuous, after a suitable approximation, we will appeal to the previously mentioned theorem which will ultimately allow us show that if $|m_{N;k}(\xi)|$ is not appropriately small, then $\xi$ is of the form $\frac{a+b\sqrt{k}}{q}+t$ for $q,|b|\lesssim N^{\varepsilon}$ and $|t|\le N^{-1+\varepsilon}$ for some adequately small $\varepsilon>0$, see Lemma~$\ref{minLemma1'}$ or Lemma~$\ref{minLemma1}$. Since the sequence $(n\lfloor n\sqrt{k}\rfloor)_{n\in\mathbb{N}}$ is of quadratic growth, the major arc length must naturally be $\lesssim N^{-2+\varepsilon'}$ for some small $\varepsilon'>0$, and thus, to conclude we have to separately treat the frequencies $\xi$ which are of the form $\frac{a+b\sqrt{k}}{q}+t$, where $N^{-2+\varepsilon'}\lesssim|t|\lesssim N^{-1+\varepsilon}$. We bound $m_{N;k}$ for these frequencies by using an appropriate factorization of $g_{\xi;k}$ for such $\xi$'s in the spirit of Theorem~1.19 in \cite{BG}, see \eqref{gxifactorization}, and by carefully appealing once again to the quantitative Leibman theorem for a suitable two dimensional abelian nilmanifold. We made some effort to apply as straightforwardly as possible the quantitative Leibman theorem and although the notation and terminology in the precise formulation of this theorem is somewhat involved, he hope that the reader will be able to interpret it appropriately for the two concrete applications necessitated by the outlined approach. 
\subsubsection{Major arcs estimates}
The important first step for establishing the major arc estimates is proving a suitable approximation result for the exponential sum of $(\frac{a+b\sqrt{k}}{q})n\lfloor n\sqrt{k}\rfloor$ along short intervals in $n$, see Proposition~$\ref{t=0}$, since then one may conclude relatively straightforwardly for small perturbations, that is, for exponential sums with phases $\big((\frac{a+b\sqrt{k}}{q}+t)n\lfloor n\sqrt{k}\rfloor\big)_{n\in[N]}$ with small $|t|$, see Proposition~$\ref{MainMajPro}$. In contrast to the situation for the squares where for $\xi=a/q$ one obtains a quadratic Gauss sum with the following simple calculation
\begin{multline}
\frac{1}{N}\sum_{n=1}^{N}e(\xi n^2)=\frac{1}{N}\sum_{n=1}^{N}e(a n^2/q)=\frac{1}{N}\sum_{r=1}^q\sum_{\substack{1\le n\le N:\\n\equiv r\Mod{q}}}e(a n^2/q)=\frac{1}{q}\sum_{r=1}^qe(a r^2/q)+O(q/N)\text{,}
\end{multline}
here the problematic frequencies do not appear as effortlessly. Ultimately, we prove that for suitably small $t$, the exponential sum $m_{N;k}\big(\frac{a+b\sqrt{k}}{q}+t\big)$ can be approximated with a polynomially decaying error in $N$ by the following product of an ``arithmetic part'' and the continuous counterpart of $m_{N;k}$
\begin{equation}\label{majarcintroduction}
\bigg(\underbrace{\mathbb{E}_{r,s\in[2qk_2]}e\Big(\frac{a}{q}rs+\frac{b}{2qk_2}(k_1r^2+k_2s^2)\Big)}_{\coloneqq \G_{k}(a,b,q)\text{, where }k=k_1/k_2\text{.}}\underbrace{\int_0^1e(-(b/2q) t^2)dt}_{\coloneqq \F(b/(2q))}\bigg)\cdot\bigg(\underbrace{\frac{1}{N}\int_{0}^Ne(t\sqrt{k}x^2)dx}_{\coloneqq V_{N;k}(t)}\bigg)\text{.}
\end{equation}
To establish the ``$t=0$''-case of the estimate, see Proposition~$\ref{t=0}$, we use an effective (simultaneous) equidistribution lemma for $(n,\lfloor n\sqrt{k}\rfloor,\{n\sqrt{k}\})$ in $\big(\mathbb{Z}/q\mathbb{Z}\big)\times\big(\mathbb{Z}/q\mathbb{Z}\big)\times\mathbb{T}$, see Lemma~$\ref{easylemma2}$, which utilizes bounds for the discrepancy of the sequence $\{n\sqrt{k}\}$, while to derive the precise form of $G_k$ we do require the full assumption $k\in\mathbb{Q}$.

The final step pertaining to the major arc analysis is establishing estimates for the arithmetic part of the operator. By van der Corput-type estimates we immediately get $\big|\F\big(b/(2q)\big)\big|\lesssim \min(1,|b/q|^{-1/2})$ and with certain number-theoretic considerations we show that $|G_{k}(a,b,q)|\lesssim_k q^{-1/2}$ for $\gcd(a,b,q)=1$. Our work here deviates from \cite{NL} where the fact that $k=2$ is used in an essential manner for the proof of the bounds for $G_2$. We give a quick proof accommodating for every $k\in\mathbb{Q}$ which establishes an estimate sufficient for our needs.
\subsubsection{Establishing oscillation estimates}We collect the estimates from the circle method in section~$\ref{circleinputsection}$. The approximation result of Proposition~$\ref{cmethod}$, the bounds for the ``arithmetic part'' and continuous parts of the operator, see \eqref{majarcintroduction}, and  a separation condition for the problematic frequencies are the only inputs used for establishing the oscillation estimates \eqref{L2osc} and is carried out in section~$\ref{oscsection}$. Our proof proceeds via a number of reductions in the spirit of \cite{Bourgain}, ultimately reducing our problem to establishing a maximal estimate on $L^2(\mathbb{R})$ for a multifrequency variant of the standard maximal operator corresponding to dilated convolutions with a smooth function, and to conclude we employ Bourgain's logarithmic lemma, see Lemma~4.13 in \cite{Bourgain}. Although the proof is complicated, we made an effort to organize the various steps in subsections. Additionally, let us mention that we hope that despite the technicalities involved, our arguments will highlight the importance of projections in pointwise ergodic theory.   
\\

The paper is organized as follows. In section~$\ref{setupsection}$ we introduce the major and minor arcs and prove some preparatory results. Sections~$\ref{majsection}$ and $\ref{minsection}$ are devoted to the major and minor arc analysis, respectively, and in section~$\ref{circleinputsection}$ we combine everything to establish the approximation result in Proposition~$\ref{cmethod}$. In the final section we prove Theorem~$\ref{osc}$, which, as discussed earlier, implies Theorem~$\ref{stronger}$, which in turn, gives Theorem~$\ref{Main}$.

\subsection{Further directions}
We end our introduction with some comments on three natural questions that arise.
\subsubsection{Generalized (or bracket) polynomials.}
It is rather natural to ask whether Theorem~$\ref{Main}$ can be extended to orbits $(n\lfloor n\alpha\rfloor)_{n\in\mathbb{N}}$ for arbitrary $\alpha$. Let us note that our arguments with straightforward modifications cover orbits of the form $(n\lfloor -n\sqrt{k}\rfloor)_{n\in\mathbb{N}}$ for $k\in\mathbb{Q}_{>0}$ and thus the pointwise result along $(n\lfloor n\alpha\rfloor)_{n\in]\mathbb{N}}$ holds when $\alpha^2\in\mathbb{Q}$. The most restricting part of our approach, not allowing us to cover more general $\alpha$'s, is deriving an approximation of the form \eqref{majarcintroduction} for $t=0$, see \eqref{notfreetarc}, where we rely on the identity $n\alpha\lfloor n\alpha\rfloor=\frac{(\alpha n)^2+\lfloor \alpha n\rfloor^2-\{n\alpha\}^2}{2}$, and it is crucial that $\alpha^2\in\mathbb{Q}$ so that we can ``factorize'' the exponential sum, see the calculation in \eqref{firstmanuever}. For a discussion regarding possible implementations of the circle method for general $\alpha$'s, we refer the reader to Chapter~6 in \cite{NL}. 

In general, it would be interesting to know whether pointwise convergence holds for ergodic averages along any generalized (or bracket) polynomial, for example $\big\lfloor\lfloor n^2e\rfloor\cdot\lfloor n^3\sqrt{2}\rfloor+n\big\rfloor\cdot\sqrt{5}n+n^3$, see section~2.3 in \cite{NL} as well as \cite{Lb1,Lb2,Lb3}, for precise definitions. The aforementioned works, together with the main result in \cite{BG} present a good starting point for developing the corresponding circle method which combined with the ideas presented here could have the potential to progress our understanding of the pointwise behaviour of ergodic averages along such orbits. 

To the best of the author's knowledge, the only (genuinely) generalized bracket polynomials along which pointwise convergence has been established in this context prior to the present work are of the form $\lfloor p(n)\rfloor$, where $p(n)\in\mathbb{R}[n]$, see section~8 in \cite{Bourgain}.

\subsubsection{The $L^p$-theory.}
Finding the sharpest rage of $p\in[1,\infty]$ for which pointwise convergence holds for the corresponding ergodic averages along an integer sequence $(a_n)_{n\in\mathbb{N}}$ is an intensely studied problem. The $L^p$-theory developed in the context of such problems relies heavily on the arithmetic structure of the problematic frequencies. Both the Ionescu--Wainger multiplier theory \cite{IWT} as well as the more ad hoc method proposed already in \cite{Bourgain} for establishing $L^p$-estimates involve arguments relying on a certain kind of periodicity. We give an argument for $L^2$ in the spirit of \cite{Bourgain} which completely forsakes the specific nature of the problematic frequencies, apart from their spacing, and we note that it is unlikely that a successful treatment of the complete $L^p$-theory can afford to rely solely on the spacing of the obstructions. In fact, for general well-spaced frequencies, certain useful intermediate estimates can fail, see \cite{CD}. How one can take into account the specific form of these frequencies $\frac{a+b\sqrt{k}}{q}$ and develop the appropriate Ionescu--Waigner theory counterpart or even manage to proceed in a more ad hoc manner similar section~7 from \cite{Bourgain} seems to be a rather interesting and difficult task. Clearly, such questions become substantially more difficult for arbitrary generalized polynomials, where the problematic frequencies may be of an even more convoluted form.
\subsubsection{State-of-the-art quantitative ergodic theorems} It would be interesting to see whether one can establish the state-of-the-art quantitative ergodic theorem along such orbits, namely, the corresponding uniform variational, jump and oscillation ergodic theorem, see Theorem~1.20 (iv),(v) and (vi) in \cite{MOE}, for the natural range $p\in(1,\infty)$. Although we do not pursue this here, we believe that an elaboration of the methods of the present work, together with standard techniques of the fields can yield such estimates on $L^2_{\mu}(X)$, and certain interpolative approaches make a result near $L^2$ to not seem out of reach, but ultimately, to obtain the complete natural range $p\in(1,\infty)$ for such estimates, the first step would be to produce a robust $L^p$-theory suitable for the maximal function.

\section*{Acknowledgments}
The author would like to thank Nikos Frantzikinakis and Borys Kuca for their valuable feedback. I would also like to thank Mariusz Mirek for multiple enlightening discussions and for his constant support and encouragement.

\section{Notation}
We use the standard notation
\[
\lfloor x\rfloor=\max\{n\in\mathbb{Z}:\,n\le x\}\text{,}\quad\{x\}=x-\lfloor x\rfloor\text{,} \quad\|x\|=\min\{|x-n|:\,n\in\mathbb{Z}\}\text{.}
\]
and for every $N\in[1,\infty)$ we let 
\[
[N]\coloneqq[1,N]\cap\mathbb{Z}\text{,}\quad[\pm N]\coloneqq[-N,N]\cap\mathbb{Z}\text{.}
\]
A function $f\colon A\to \mathbb{C}$ is called $1$-bounded if $|f|\le 1$. For every function $f\colon A\to\mathbb{C}$ and every nonempty subset $B\subseteq A$ we denote the average value of $f$ over $B$ as usual by
\[
\mathbb{E}_{b\in B}f(b)\coloneqq\frac{1}{|B|}\sum_{b\in B}f(b)\text{.}
\]
For two nonnegative quantities $A,B$, we write $A\lesssim B$ or $B \gtrsim A$  to denote that there exists a positive constant $C$, possibly depending on a fixed choice of parameters such that $A\le C B$. Whenever we want to highlight the dependence of the implicit constant $C$ on a set of parameters $\tau_1,\dotsc,\tau_{n}$, we write $A\lesssim_{\tau_1,\dotsc,\tau_{n}} B$. 

In the sequel, we denote $e^{2\pi i x}$ by $e(x)$. For every finitely supported $f\colon\mathbb{Z}\to\mathbb{C}$, we define the Fourier transform of $f$ as the function $\hat{f}\colon \mathbb{T}\to\mathbb{C}$ such that 
\[
\hat{f}(\xi)=\sum_{k\in\mathbb{Z}}f(k)e(k\xi)\text{,}
\]
and we sometimes opt for the notation $\mathcal{F}_{\mathbb{Z}}[f](\xi)=\hat{f}(\xi)$. For every $g\colon L^1(\mathbb{T})\to\mathbb{C}$, we define the inverse Fourier transform $\mathcal{F}^{-1}_{\mathbb{Z}}[g]\colon \mathbb{Z}\to\mathbb{C}$ as
\[
\mathcal{F}^{-1}_{\mathbb{Z}}[g](x)=\int_{\mathbb{T}}g(\xi)e(-x\xi)d\xi\text{.}
\]
For every finitely supported function $f\colon\mathbb{Z}\to\mathbb{C}$ and every trigonometric polynomial $g\colon\mathbb{T}\to\mathbb{C}$ we have
\[
\mathcal{F}^{-1}_{\mathbb{Z}}\big[\mathcal{F}_{\mathbb{Z}}[f]\big]=f\quad\text{and}\quad
\mathcal{F}_{\mathbb{Z}}\big[\mathcal{F}^{-1}_{\mathbb{Z}}[g]\big]=g\text{.}
\]
Both operators extend boundedly to the entirety of $\ell^2(\mathbb{Z})$ and $L^2(\mathbb{T})$, respectively, and the inversion formulas above still hold. For every function $h\in\mathcal{S}(\mathbb{R},\mathbb{C})$ we define the Fourier and inverse Fourier transform as follows
\[
\mathcal{F}_{\mathbb{R}}[h](\xi)=\int_{\mathbb{R}}h(x)e(\xi x)dx\quad\text{and}\quad\mathcal{F}^{-1}_{\mathbb{R}}[h](x)=\int_{\mathbb{R}}h(\xi)e(-x\xi)dx\text{,}
\]
and we have 
\[
\mathcal{F}^{-1}_{\mathbb{R}}\big[\mathcal{F}_{\mathbb{R}}[h]\big]=h\quad\text{and}\quad
\mathcal{F}_{\mathbb{R}}\big[\mathcal{F}^{-1}_{\mathbb{R}}[h]\big]=g\text{.}
\]
Similarly, the operators extend boundedly to $L^2(\mathbb{R})$ and the same two identities hold. For any $m\in L^{\infty}(\mathbb{T})$ and finitely supported $f\colon \mathbb{Z}\to\mathbb{C}$, we define $T_{\mathbb{Z}}[m]f\colon \mathbb{Z}\to \mathbb{C}$ by 
\[
T_{\mathbb{Z}}[m]f(n)=\mathcal{F}^{-1}_{\mathbb{Z}}\big[m\mathcal{F}_{\mathbb{Z}}[f]\big](n)=\int_{\mathbb{T}}m(\xi)\mathcal{F}_{\mathbb{Z}}[f](\xi)e(-\xi n)d\xi\text{.}
\]
For any $m\in L^{\infty}(\mathbb{R})$ and $h\in\mathcal{S}(\mathbb{R},\mathbb{C})$, we define $T_{\mathbb{R}}[m]h\colon \mathbb{R}\to \mathbb{C}$ by 
\[
T_{\mathbb{R}}[m]h(x)=\mathcal{F}^{-1}_{\mathbb{R}}\big[m\mathcal{F}_{\mathbb{R}}[h]\big](x)=\int_{\mathbb{R}}m(\xi)\mathcal{F}_{\mathbb{R}}[h](\xi)e(-\xi x)d\xi\text{.}
\]
Note that with this notation the averaging operator \eqref{avop} becomes $A_{t;k}f=T_{\mathbb{Z}}[m_{t;k}]f$.
\section{Setting up the circle method}\label{setupsection} In this section we introduce some notation and prove certain preparatory results for the circle method suitably adapted to the analysis of the following exponential sums
\[
m_{N;k}(\xi)\coloneqq\frac{1}{N}\sum_{n\in[N]}e(\xi n\lfloor n\sqrt{k}\rfloor)\text{.}
\]
We fix $k\in\mathbb{Q}_{>0}$ with $\sqrt{k}\not\in\mathbb{Q}$ as well as two parameters $\gamma,\gamma'\in(0,1/10)$ to be determined. For every $q\in\mathbb{N}$, $a,b\in\mathbb{Z}$ and $N\in[1,\infty)$, let 
\begin{equation}\label{mabq}
\widetilde{\mathfrak{M}}_{a,b,q}=\widetilde{\mathfrak{M}}_{a,b,q}(N;k,\gamma)\coloneqq\Big\{\xi\in[-1/2,1/2):\,\Big\|\xi-\frac{a+b\sqrt{k}}{q}\Big\|\le N^{-1+\gamma}\Big\}\text{,}
\end{equation}
\begin{equation}\label{Mabq}
\mathfrak{M}_{a,b,q}=\mathfrak{M}_{a,b,q}(N;k,\gamma')\coloneqq\Big\{\xi\in[-1/2,1/2):\,\Big\|\xi-\frac{a+b\sqrt{k}}{q}\Big\|\le N^{-2+\gamma'}\Big\}\text{,}
\end{equation}
and note that $\mathfrak{M}_{a,b,q}\subseteq \widetilde{\mathfrak{M}}_{a,b,q}$, since $-1+\gamma>-2+\gamma'\iff 1>\gamma'-\gamma\quad\text{and}\quad\gamma,\gamma'\in(0,1/10)$. We show that for $N\gtrsim_{k,\gamma,\gamma'}1$ we have that
\begin{equation}\label{ExtMaj}
\Big\{\widetilde{\mathfrak{M}}_{a,b,q}:\,q\in[N^{\gamma}],\,b\in[\pm N^{\gamma}],\,a\in[q],\,\gcd(a,b,q)=1\Big\}
\end{equation}
contains mutually disjoint sets. This will be an immediate consequence of the following three lemmas, the first of which is elementary and we omit its proof.
\begin{lemma}[Unique representation for the major arc centers]\label{Faithful}
Assume $\alpha$ is an irrational number, $q,q'\in \mathbb{N}$, $b,b'\in\mathbb{Z}$ and $a\in[q],a'\in[q']$ are such that $\gcd(a,b,q)=\gcd(a',b',q')=1$. Then 
\begin{equation}\label{niceequiv}
\frac{a+b\alpha}{q}\equiv\frac{a'+b'\alpha}{q'}\Mod{1}\iff (a,b,q)=(a',b',q')\text{.}
\end{equation}
\end{lemma}
\begin{lemma}[General separation condition]\label{seperationverygeneral}
Assume $\alpha$ is a real algebraic number of degree $d\ge  2$ and let $X,Y\ge 1$. Then there a positive constant $c=c(\alpha)$ such that for all $(a,b,q)\neq(a',b',q')$ with $q,q'\in[X]$, $b,b'\in[\pm Y]$, $a\in[q]$, $a'\in[q']$ and $\gcd(a,b,q)=\gcd(a',b',q')=1$ we have
\[
\Big\|\frac{a+b\alpha}{q}-\frac{a'+b'\alpha}{q'}\Big\|>c X^{-d-1}Y^{-d+1}\text{.}
\]
\end{lemma}
\begin{proof}
Let $(a,b,q)\neq(a',b',q')$ with $q,q'\in[X]$, $b,b'\in[\pm Y]$, $a\in[q]$, $a'\in[q']$ and $\gcd(a,b,q)=\gcd(a',b',q')=1$. If $b/q=b'/q'$, then by Lemma~$\ref{Faithful}$, since $(a,b,q)\neq(a',b',q')$ we must have $a/q\neq a'/q'\Mod{1}$, and thus for all $m\in\mathbb{Z}$ we get that $aq'-a'q-qq'm\neq 0$. Then there exists $m_0\in\mathbb{Z}$ such that
\[
\Big\|\frac{a+b\alpha}{q}-\frac{a'+b'\alpha}{q'}\Big\|=\Big\|\frac{aq'-a'q}{qq'}\Big\|=\Big|\frac{aq'-a'q-m_0qq'}{qq'}\Big|\ge \frac{1}{qq'}\ge X^{-2}\ge X^{-d-1}Y^{-d+1}\text{.}
\] 
On the other hand, if $b/q\neq b'/q'$, then there exists $m_0\in\mathbb{Z}$ such that
\begin{multline}
\Big\|\frac{a+b\alpha}{q}-\frac{a'+b'\alpha}{q'}\Big\|=\Big|\frac{a+b\alpha}{q}-\frac{a'+b'\alpha}{q'}-m_0\Big|=\Big|\frac{aq'-a'q}{qq'}+\frac{(bq'-b'q)\alpha}{qq'}-m_0\Big|
\\
=\frac{|bq'-b'q|}{qq'}\bigg|\frac{aq'-a'q}{bq'-b'q}-\frac{m_0qq'}{bq'-b'q}+\alpha\bigg|\gtrsim_{\alpha} \frac{|bq'-b'q|}{qq'}\frac{1}{|bq'-b'q|^d}=\frac{1}{qq'|bq'-b'q|^{d-1}}
\\
\ge \frac{1}{X^2(|bq'|+|b'q|)^{d-1}}\ge\frac{1}{X^2(2XY)^{d-1}}=\frac{1}{2^{d-1}}X^{-d-1}Y^{-d+1}\text{,}
\end{multline}
where for the estimate in the second line we used Liouville's theorem \cite{LV}, namely the fact that for every real algebraic number $\alpha$ of degree $d\ge 2$, there exists a positive constant $c(\alpha)$ such that for all $p\in\mathbb{Z}$ and $q\in\mathbb{N}$ we have
\begin{equation}\label{Liouvillemanouver}
\Big|\alpha-\frac{p}{q}\Big|\ge\frac{c(\alpha)}{q^d} \text{.}
\end{equation}
\end{proof}
\begin{lemma}[Separation condition for the major arc centers]\label{seperationgeneral}
Assume $\alpha$ is a real algebraic number of degree $d\ge  2$ and $\gamma_1,\gamma_2,\gamma_3\in(0,1)$ satisfy 
\begin{equation}\label{restrgamma}
(d+1)\gamma_1+(d-1)\gamma_2+\gamma_3<1\text{.}
\end{equation}
Then there exists a positive constant $C=C(\alpha,\gamma_1,\gamma_2,\gamma_3)$ such that for all $N\ge C$ we have
\[
\Big\|\frac{a+b\alpha}{q}-\frac{a'+b'\alpha}{q'}\Big\|>2N^{-1+\gamma_3}
\]
for all $(a,b,q)\neq(a',b',q')$ with $q,q'\in[N^{\gamma_1}]$, $b,b'\in[\pm N^{\gamma_2}]$, $a\in[q]$, $a'\in[q']$ and $\gcd(a,b,q)=\gcd(a',b',q')=1$.
\end{lemma}
\begin{proof}Assume for the sake of a contradiction that there exists such $(a,b,q)\neq(a',b',q')$. By the previous lemma we have that
\[
2N^{-1+\gamma_3}\ge\Big\|\frac{a+b\alpha}{q}-\frac{a'+b'\alpha}{q'}\Big\|\gtrsim_{\alpha} N^{\gamma_1(-d-1)}N^{\gamma_2(-d+1)}\text{,}
\]
but then $N^{-1+(d+1)\gamma_1+(d-1)\gamma_2+\gamma_3}\gtrsim_{\alpha}1$, which is a contradiction for $N\gtrsim_{\alpha,\gamma_1,\gamma_2,\gamma_3}1$, since $-1+(d+1)\gamma_1+(d-1)\gamma_2+\gamma_3<0$, and the proof is complete.
\end{proof}
We see that for every fixed choice of $k\in\mathbb{Q}_{>0}$ with $\sqrt{k}\not\in\mathbb{Q}$ and $\gamma,\gamma'\in(0,1/10)$ the previous lemma is applicable for $\alpha=\sqrt{k}$, and thus the family \eqref{ExtMaj} comprises of mutually disjoint sets for $N\gtrsim_{k,\gamma,\gamma'}1$. Note that, since $\mathfrak{M}_{a,b,q}\subseteq \widetilde{\mathfrak{M}}_{a,b,q}$, the same holds for the analogous family of $\mathfrak{M}_{a,b,q}$'s. 

We identify the torus $\mathbb{T}$ with $[-1/2,1/2)$, and partition it into major and minor arcs, denoted by $\mathfrak{M}=\mathfrak{M}(N;k,\gamma,\gamma')$ and $\mathfrak{m}=\mathfrak{m}(N;k,\gamma,\gamma')$, respectively, as follows
\begin{equation}\label{Mandm}
\mathfrak{M}\coloneqq\bigcup_{q\in[N^{\gamma}]}\bigcup_{b\in[\pm N^{\gamma}]}\bigcup _{\substack{a\in[q]:\\\gcd(a,b,q)=1}}\mathfrak{M}_{a,b,q}\quad\text{and}\quad\mathfrak{m}\coloneqq
\mathbb{T}\setminus \mathfrak{M}\text{.}
\end{equation}
We will require a further partition of the minor arcs, see subsection~$\ref{minarcstrat}$, which we introduce below 
\begin{equation*}
\mathfrak{m}_2=\mathfrak{m}_2(N;k,\gamma,\gamma')\coloneqq\bigg(\bigcup_{q\in[ N^{\gamma}]}\bigcup_{b\in[\pm N^{\gamma}]}\bigcup _{\substack{a\in[q]:\\\gcd(a,b,q)=1}}\widetilde{\mathfrak{M}}_{a,b,q}\bigg)\setminus\mathfrak{M}\quad\text{and}\quad \mathfrak{m}_1=\mathfrak{m}_1(N;k,\gamma,\gamma')\coloneqq \mathfrak{m}\setminus\mathfrak{m}_2\text{.}
\end{equation*}
Taking into account the fact that the union above is taken over mutually disjoint sets, as well as the fact that $\mathfrak{M}_{a,b,q;k}\subseteq \widetilde{\mathfrak{M}}_{a,b,q}$, we get that
\[
\mathfrak{m}_2=\bigcup_{q\in[ N^{\gamma}]}\bigcup_{b\in[\pm N^{\gamma}]}\bigcup _{\substack{a\in[q]:\\\gcd(a,b,q)=1}}\Big\{\xi\in[-1/2,1/2):\,N^{-2+\gamma'}<\Big\|\xi-\frac{a+b\sqrt{k}}{q}\Big\|\le N^{-1+\gamma}\Big\}
\]
and 
\begin{equation}\label{m1disjoint}
\mathfrak{m}_1=\mathbb{T}\setminus \bigg(\bigcup_{q\in[ N^{\gamma}]}\bigcup_{b\in[\pm N^{\gamma}]}\bigcup _{\substack{a\in[q]:\\\gcd(a,b,q)=1}}\widetilde{\mathfrak{M}}_{a,b,q}\bigg)\quad\text{for $N\gtrsim_{k,\gamma,\gamma'}1$.}
\end{equation}
This concludes our preparation for the development of the circle method in the sequel, and we note that we will treat $k,\gamma,\gamma'$ as fixed parameters, making the major and minor arcs, as well as the further partitioning of the minor arcs, depend only on $N\in[1,\infty)$. 
\section{Major arc estimates}\label{majsection}
Here we perform the major arc analysis which in the context of our problem naturally splits into two tasks. Firstly, we establish a major arc approximation result, which should be thought of as a tool to appropriately factorize the multiplier 
\[
m_{N;k}(\xi)=\frac{1}{N}\sum_{n=1}^Ne(\xi n\lfloor n\sqrt{k}\rfloor)\quad\text{for $\xi\in\mathfrak{M}(N;k,\gamma,\gamma')$.}
\] 
More precisely, one of the main results of the present section is the following.
\begin{proposition}[Major arc approximation]\label{MainMajPro}
For every $k\in\mathbb{Q}_{>0}$ with $\sqrt{k}\not\in\mathbb{Q}$ and every $\kappa\in(0,1)$ there exists a positive constant $C=C(k,\kappa)$ such that for every $N\in [1,\infty)$, $q\in\mathbb{N}$ and $a,b\in\mathbb{Z}$, we have
\begin{multline}\label{freetarc}
\bigg|\bigg(\mathbb{E}_{n\in [N]}e\Big(\Big(\frac{a+b\sqrt{k}}{q}+t\Big)n\lfloor n\sqrt{k}\rfloor\Big)\bigg)-\G_k(a,b,q)\F\bigg(\frac{b}{2q}\bigg)V_{N;k}(t)\bigg|
\\
\le C\Big(N^{\kappa-1}+|t|N^{1+\kappa}+q^2(1+|b|)(1+\log N )N^{-\kappa/2}\Big)\text{,}
\end{multline}
where, if $k=\frac{k_1}{k_2}$ with $k_1,k_2\in\mathbb{N}$ and $\gcd(k_1,k_2)=1$, we have let
\begin{equation}\label{GFdef}
\G_{k}(a,b,q)\coloneqq \mathbb{E}_{r,s\in[2qk_2]}e\Big(\frac{a}{q}rs+\frac{b}{2qk_2}\big(k_1r^2+k_2s^2\big)\Big)\text{,}\quad\F(\xi)\coloneqq \int_0^1e(-\xi x^2)dx
\end{equation}
and
\begin{equation}\label{Vdef}
V_{N;k}(t)\coloneqq \frac{1}{N}\int_{0}^Ne(t\sqrt{k}x^2)dx\text{.}
\end{equation}
\end{proposition}
The second task is to appropriately bound the three factors appearing in Proposition~$\ref{MainMajPro}$, namely $\G_k,\F,V_{N;k}$. This is achieved straightforwardly using van der Corput estimates for the last two, see Proposition~$\ref{frVXbounds}$, while number-theoretic considerations are naturally required for handling the double Gauss sum variant $\G_k$, see Proposition~$\ref{GaussSumstype}$. These two tasks are completed in the following two subsections.
\subsection{Major arc approximation}\label{majaproxsection}In this subsection we prove Proposition~$\ref{MainMajPro}$. A key intermediate step is to establish a variant of the result with $t=0$, see Proposition~$\ref{t=0}$, in the proof of which the expressions $\G_k$ and $\F$ naturally appear. We begin by collecting some useful lemmas.
\begin{lemma}\label{easylemma1}
Let $x\in\mathbb{R}$, $q\in\mathbb{N}$ and $r\in\{0,\dotsc,q-1\}$. Then
\[
\lfloor x \rfloor\equiv r\Mod{q}\iff \bigg\{\frac{x}{q}\bigg\}\in\bigg[\frac{r}{q},\frac{r+1}{q}\bigg)\text{.}
\]
\end{lemma}
\begin{proof}
The left-hand side implies that there exists $m\in\mathbb{Z}$ such that $mq=\lfloor x\rfloor -r$ and thus we have
\[
x-\{x\}-mq=r\Rightarrow \frac{x}{q}-m=\frac{r+\{x\}}{q}\in\bigg[\frac{r}{q},\frac{r+1}{q}\bigg)\subseteq[0,1)\Rightarrow \bigg\{\frac{x}{q}\bigg\}\in\bigg[\frac{r}{q},\frac{r+1}{q}\bigg)\text{.}
\]
Conversely, the right-hand side implies that there exists $m\in\mathbb{Z}$ such that
\[
\frac{r}{q}\le \frac{x}{q}-m<\frac{r+1}{q}\Rightarrow r\le x-mq<r+1\Rightarrow mq+r\le x<mq+r+1\Rightarrow \lfloor x\rfloor =mq+r\text{,}
\]
and thus $\lfloor x\rfloor\equiv r\Mod{q}$.
\end{proof}
\begin{lemma}\label{easylemma3}For every $n\in\mathbb{N}$ and $\alpha\in\mathbb{R}$ we have
\[
n\alpha\lfloor n\alpha\rfloor=\frac{(\alpha n)^2+\lfloor \alpha n\rfloor^2-\{n\alpha\}^2}{2}\text{.}
\]
\end{lemma}
\begin{proof}
The assertion is equivalent to $\{n\alpha\}^2=(\alpha n)^2+\lfloor \alpha n\rfloor^2-2n\alpha\lfloor n\alpha\rfloor$,
which is clearly true since $\{n\alpha\}=n\alpha-\lfloor n\alpha\rfloor$.
\end{proof}
\begin{lemma}[Effective equidistribution $(n,\lfloor n\sqrt{k}\rfloor,\{n\sqrt{k}\})$ in $\big(\mathbb{Z}/q\mathbb{Z}\big)\times\big(\mathbb{Z}/q\mathbb{Z}\big)\times\mathbb{T}$]\label{easylemma2}Assume $k\in\mathbb{Q}_{>0}$ is such that $\sqrt{k}\not\in\mathbb{Q}$. Assume that $I\subseteq\mathbb{Z}$ is a nonempty interval and $D,q,r,s,d$ are natural numbers with $d\in\{0,\dotsc, D-1\}$. Then we have 
\[
\frac{1}{|I|}\bigg|\bigg\{n\in I:\,n\equiv r\Mod{q},\,\lfloor n\sqrt{k}\rfloor\equiv s\Mod{q},\,\{n\sqrt{k}\}\in\bigg[\frac{d}{D},\frac{d+1}{D}\bigg)\bigg\}\bigg|=\frac{1}{q^2D}+O_k\bigg(\frac{1+\log |I|}{|I|}\bigg)\text{.}
\]
\end{lemma}
\begin{proof}We fix $q,D$ and without loss of generality, we let $r,s\in\{0,\dotsc,q-1\}$, $d\in\{0,\dotsc,D-1\}$. Let $I_{r,s,d}\subseteq I$ be the set defined in the left-hand side above. Firstly, note that if $|I|\le 2q$, then $|I_{r,s,d}|\le 2$, and thus
\[
\bigg|\frac{|I_{r,s,d}|}{|I|}-\frac{1}{q^2D}\bigg|\le \frac{2}{|I|}+\frac{1}{q^2D}\le \frac{4}{|I|}\text{,}
\]
so we assume from now on that $|I|>2q$.
Let $a,b\in\mathbb{Z}$ be such that $I=\{a,\dotsc,b\}$. We define
\[
I'=\{l\in\mathbb{Z}:\,a\le ql+r\le b\}=\Big\{l\in\mathbb{Z}:\,\frac{a-r}{q}\le l\le\frac{b-r}{q}\Big\}\text{.}
\]
Note that $I'$ has at least one element, and also that
\[
|I_{r,s,d}|=\bigg|\bigg\{l\in I':\,\lfloor (ql+r)\sqrt{k}\rfloor\equiv s\Mod{q},\,\{(ql+r)\sqrt{k}\}\in\bigg[\frac{d}{D},\frac{d+1}{D}\bigg)\bigg\}\bigg|\text{.}
\]
By Lemma~$\ref{easylemma1}$ we get $\lfloor (ql+r)\sqrt{k}\rfloor\equiv s\Mod{q}\iff \{(ql+r)\sqrt{k}/q\}\in[s/q,(s+1)/q)$ which, in turn, is equivalent to
\[
\bigg\{l\sqrt{k}+\frac{r\sqrt{k}}{q}\bigg\}\in\bigg[\frac{s}{q},\frac{s+1}{q}\bigg)\iff\bigg\{l\sqrt{k}+\frac{r\sqrt{k}-s}{q}\bigg\}\in\bigg[0,\frac{1}{q}\bigg)\text{.}
\]
We claim that for all $l\in\mathbb{Z}$ the following equivalence holds
\[
\bigg\{l\sqrt{k}+\frac{r\sqrt{k}-s}{q}\bigg\}\in\bigg[0,\frac{1}{q}\bigg)\text{ and }\{(ql+r)\sqrt{k}\}\in\bigg[\frac{d}{D},\frac{d+1}{D}\bigg)\iff\bigg\{l\sqrt{k}+\frac{r\sqrt{k}-s}{q}\bigg\}\in\bigg[\frac{d}{qD},\frac{d+1}{qD}\bigg)\text{.}
\]
To see this, note that if $l$ satisfies the left-hand side, then 
\begin{multline}
\{(ql+r)\sqrt{k}\}=\Big\{q\Big(l\sqrt{k}+\frac{r\sqrt{k}-s}{q}\Big)\Big\}=q\Big(l\sqrt{k}+\frac{r\sqrt{k}-s}{q}\Big)-\Big\lfloor q\Big(l\sqrt{k}+\frac{r\sqrt{k}-s}{q}\Big)\Big\rfloor
\\
=q\Big\{l\sqrt{k}+\frac{r\sqrt{k}-s}{q}\Big\}+q\Big\lfloor l\sqrt{k}+\frac{r\sqrt{k}-s}{q}\Big\rfloor-\Big\lfloor q\Big(l\sqrt{k}+\frac{r\sqrt{k}-s}{q}\Big)\Big\rfloor\text{,}
\end{multline}
and we have $q\Big\{l\sqrt{k}+\frac{r\sqrt{k}-s}{q}\Big\}\in[0,1)$ since the first condition holds. By taking fractional parts we obtain
\[
\bigg[\frac{d}{D},\frac{d+1}{D}\bigg)\ni\{(ql+r)\sqrt{k}\}=\{\{(ql+r)\sqrt{k}\}\}=q\Big\{l\sqrt{k}+\frac{r\sqrt{k}-s}{q}\Big\}\text{,}
\] 
and thus $\big\{l\sqrt{k}+\frac{r\sqrt{k}-s}{q}\big\}\in \big[\frac{d}{qD},\frac{d+1}{qD}\big)$ as desired. Conversely, if the right-hand side holds, then clearly the first condition of the left-hand side holds. But then, similarly to before, we get
\[
\{(ql+r)\sqrt{k}\}=q\Big\{l\sqrt{k}+\frac{r\sqrt{k}-s}{q}\Big\}\in\bigg[\frac{d}{D},\frac{d+1}{D}\bigg)\text{,}
\]
and the proof of the equivalence is complete. Therefore we get 
\begin{multline}
|I_{r,s,d}|=\bigg|\bigg\{l\in I':\,\Big\{l\sqrt{k}+\frac{r\sqrt{k}-s}{q}\Big\}\in\bigg[\frac{d}{qD},\frac{d+1}{qD}\bigg)\bigg\}\bigg|\\
=\bigg|\bigg\{l\in I':\,l\sqrt{k}+\frac{r\sqrt{k}-s}{q}\in\bigg[\frac{d}{qD},\frac{d+1}{qD}\bigg)+\mathbb{Z}\bigg\}\bigg|
\\
=\bigg|\bigg\{l\in I':\,l\sqrt{k}\in\bigg[\frac{d}{qD}-\frac{r\sqrt{k}-s}{q},\frac{d+1}{qD}-\frac{r\sqrt{k}-s}{q}\bigg)+\mathbb{Z}\bigg\}\bigg|\text{.}
\end{multline}
Since $\sqrt{k}$ is a quadratic irrational, the digits of its continued fraction expansion are bounded since the expansion is eventually periodic and thus, by Corollary~1.65 in \cite{refequid} for example, we get
\begin{equation}\label{discrepancy}
\Bigg|\bigg|\bigg\{l\in I':\,l\sqrt{k}\in\bigg[\frac{d}{qD}-\frac{r\sqrt{k}-s}{q},\frac{d+1}{qD}-\frac{r\sqrt{k}-s}{q}\bigg)+\mathbb{Z}\bigg\}\bigg|-\frac{|I'|}{qD}\Bigg|\lesssim_k 1+\log |I'|\le1+\log |I|\text{.}
\end{equation}
Finally, we get
\[
\bigg|\frac{|I_{r,s,d}|}{|I|}-\frac{1}{q^2D}\bigg|\le\frac{1}{|I|}\bigg||I_{r,s,d}|-\frac{|I'|}{qD}\bigg|+\frac{1}{|I|}\bigg|\frac{|I'|}{qD}-\frac{|I|}{q^2D}\bigg| \lesssim \frac{1+\log|I|}{|I|}\text{,}
\]
where for the last estimate we have used \eqref{discrepancy} and the following estimate
\[
\bigg|\frac{|I'|}{qD}-\frac{|I|}{q^2D}\bigg|=\frac{1}{q^2D}\big|q|I'|-|I|\big|\le \frac{2}{qD}\lesssim 1 \text{,}
\] 
where the penultimate estimate can be justified as follows
\[
b-a-q\le q|I'|\le b-a+q\Rightarrow -q-1\le q|I'|-|I|\le q-1\Rightarrow \big|q|I'|-|I|\big|\le q+1\le 2q
\text{.}
\]
\end{proof} 
We now formulate and prove the variant of Proposition~$\ref{MainMajPro}$ for $t=0$.
\begin{proposition}\label{t=0}Assume $k\in\mathbb{Q}_{>0}$ is such that $\sqrt{k}\not\in\mathbb{Q}$. Then there exists a positive constant $C=C(k)$ such that for every nonempty interval $I\subseteq \mathbb{Z}$, $a,b\in\mathbb{Z}$ and $q\in\mathbb{N}$, we have
\begin{equation}\label{notfreetarc}
\bigg|\mathbb{E}_{n\in I}e\Big(\frac{a+b\sqrt{k}}{q}n\lfloor n\sqrt{k}\rfloor\Big)-\G_k(a,b,q)\F\bigg(\frac{b}{2q}\bigg)\bigg|
\le Cq^2(1+|b|)(1+\log |I|)|I|^{-1/2} \text{,}
\end{equation}
where $\G_k,\F$ are defined in \eqref{GFdef}.
\end{proposition}
\begin{proof}
Let $k_1,k_2\in\mathbb{N}$ be such that $k=k_1/k_2$ with $\gcd(k_1,k_2)=1$ and let $D$ be a natural number to be determined later. Similarly to the proof of Lemma~$\ref{easylemma2}$, for every $r,s\in[2qk_2]$ and $d\in \{0,\dotsc,D-1\}$ we define 
\[
I_{r,s,d}\coloneqq \bigg\{n\in I:\,n\equiv r\Mod{2qk_2},\,\lfloor n\sqrt{k}\rfloor\equiv s\Mod{2qk_2},\,\{n\sqrt{k}\}\in\bigg[\frac{d}{D},\frac{d+1}{D}\bigg)\bigg\}\text{,}
\]
and we partition the interval as follows $I=\bigcup_{\substack{r,s\in[2qk_2]\\0\le d<D}}I_{r,s,d}$. By Lemma~$\ref{easylemma3}$ we get
\begin{multline}\label{firstmanuever}
\sum_{n\in I}e\Big(\frac{a+b\sqrt{k}}{q}n\lfloor n\sqrt{k}\rfloor\Big)=\sum_{n\in I}e\Big(\frac{a}{q}n\lfloor n\sqrt{k}\rfloor+\frac{b}{q}n\sqrt{k}\lfloor n\sqrt{k}\rfloor \Big)
\\
=\sum_{n\in I}e\bigg(\frac{a}{q}n\lfloor n\sqrt{k}\rfloor+\frac{b}{2q}
\bigg(\frac{k_1}{k_2} n^2+\lfloor  n\sqrt{k}\rfloor^2-\{n\sqrt{k}\}^2\bigg) \bigg)
\\
=\sum_{\substack{r,s\in[2qk_2]\\0\le d<D}}\sum_{n\in I_{r,s,d}}e\bigg(\frac{a}{q}n\lfloor n\sqrt{k}\rfloor+\frac{b}{2qk_2}
\Big(k_1n^2+k_2\lfloor  n\sqrt{k}\rfloor^2\Big)\bigg)e\Big(-\frac{b}{2q}\{n\sqrt{k}\}^2\Big)
\\
=\sum_{\substack{r,s\in[2qk_2]\\0\le d<D}}\sum_{n\in I_{r,s,d}}e\Big(\frac{a}{q}rs+\frac{b}{2qk_2}
\big(k_1 r^2+k_2s^2\big) \Big)e\Big(-\frac{b}{2q}\{n\sqrt{k}\}^2\Big)
\\
=\sum_{r,s\in[2qk_2]}e\Big(\frac{a}{q}rs+\frac{b}{2qk_2}(k_1 r^2+k_2s^2)\Big)\sum_{0\le d<D}\sum_{n\in I_{r,s,d}}e\Big(-\frac{b}{2q}\{n\sqrt{k}\}^2\Big)
\\
=\sum_{r,s\in[2qk_2]}e\Big(\frac{a}{q}rs+\frac{b}{2qk_2}(k_1 r^2+k_2s^2)\Big)\sum_{0\le d<D}|I_{r,s,d}|e\Big(-\frac{bd^2}{2qD^2}\Big)
\\
+\sum_{r,s\in[2qk_2]}e\Big(\frac{a}{q}rs+\frac{b}{2qk_2}(k_1 r^2+k_2s^2)\Big)\sum_{0\le d<D}\sum_{n\in I_{r,s,d}}\bigg(e\Big(-\frac{b}{2q}\{n\sqrt{k}\}^2\Big)-e\Big(-\frac{bd^2}{2qD^2}\Big)\bigg)\eqqcolon S_1+S_2 \text{.}
\end{multline}
For $S_2$ we note that by triangle inequality and an application of Lemma~$\ref{easylemma2}$ we get
\begin{multline}\label{estforS_2}
|S_2|\lesssim\sum_{r,s\in[2qk_2]}\sum_{0\le d<D}\sum_{n\in I_{r,s,d}} \frac{|b|}{q}\Big|\{n\sqrt{k}\}^2-\frac{d^2}{D^2}\Big|\lesssim \frac{|b|}{q}\sum_{r,s\in[2qk_2]}\sum_{0\le d<D}|I_{r,s,d}|\frac{(d+1)^2-d^2}{D^2}
\\
\lesssim_k \frac{|b|q}{D^2}\sum_{0\le d<D}(2d+1)\bigg(\frac{|I|}{4q^2k_2^2D}+\log|I|+1\bigg)\lesssim_k \frac{|I||b|}{qD}+|b|q(1+\log |I|)\text{.}
\end{multline}
For $S_1$ we note
\begin{multline}
\sum_{r,s\in[2qk_2]}e\Big(\frac{a}{q}rs+\frac{b}{2qk_2}(k_1 r^2+k_2s^2)\Big)\sum_{0\le d<D}|I_{r,s,d}|e\Big(-\frac{bd^2}{2qD^2}\Big)
\\
=\sum_{r,s\in[2qk_2]}e\Big(\frac{a}{q}rs+\frac{b}{2qk_2}(k_1 r^2+k_2s^2)\Big)\frac{|I|}{4q^2k_2^2D}\sum_{0\le d<D}e\Big(-\frac{bd^2}{2qD^2}\Big)
\\
+\sum_{r,s\in[2qk_2]}e\Big(\frac{a}{q}rs+\frac{b}{2qk_2}(k_1 r^2+k_2s^2)\Big)\sum_{0\le d<D}\Big(|I_{r,s,d}|-\frac{|I|}{4q^2k_2^2D}\Big)e\Big(-\frac{bd^2}{2qD^2}\Big)
\\
=|I|\G_k(a,b,q)\bigg(\frac{1}{D}\sum_{0\le d<D}e\Big(-\frac{bd^2}{2qD^2}\Big)\bigg)
\\
+\sum_{r,s\in[2qk_2]}e\Big(\frac{a}{q}rs+\frac{b}{2qk_2}(k_1 r^2+k_2s^2)\Big)\sum_{0\le d<D}\Big(|I_{r,s,d}|-\frac{|I|}{4q^2k_2^2D}\Big)e\Big(-\frac{bd^2}{2qD^2}\Big)
\\
=|I|\G_k(a,b,q)\bigg(\frac{1}{D}\int_{0}^{D}e\Big(-\frac{bx^2}{2qD^2}\Big)dx\bigg)
\\
+|I|\G_k(a,b,q)\bigg(\frac{1}{D}\sum_{0\le d<D}e\Big(-\frac{bd^2}{2qD^2}\Big)-\frac{1}{D}\int_{0}^{D}e\Big(-\frac{bx^2}{2qD^2}\Big)dx\bigg)
\\
+\sum_{r,s\in[2qk_2]}e\Big(\frac{a}{q}rs+\frac{b}{2qk_2}(k_1 r^2+k_2s^2)\Big)\sum_{0\le d<D}\Big(|I_{r,s,d}|-\frac{|I|}{4q^2k_2^2D}\Big)e\Big(-\frac{bd^2}{2qD^2}\Big)\eqqcolon M+E_1+E_2\text{.}
\end{multline}
By a change of variables
\[
\frac{1}{D}\int_{0}^{D}e\Big(-\frac{bx^2}{2qD^2}\Big)dx=\F\bigg(\frac{b}{2q}\bigg)\text{,}
\]
and we are left with the task of bounding the error terms $E_1,E_2$.
\\\,\\
\textbf{Estimates for $E_1$.} We have
\begin{equation}\label{e1step}
|E_1|\le\frac{|I|}{D}\bigg|\sum_{d=0}^{D-1}e\Big(-\frac{bd^2}{2qD^2}\Big)-\int_{0}^{D}e\Big(-\frac{bx^2}{2qD^2}\Big)dx\bigg|\text{,}
\end{equation}
and we may estimate as follows
\begin{multline}
\Big|\sum_{d=0}^{D-1}e\Big(-\frac{bd^2}{2qD^2}\Big)-\int_{0}^{D}e\Big(-\frac{bx^2}{2qD^2}\Big)dx\Big|\le\sum_{d=0}^{D-1}\int_{d}^{d+1}\Big|e\Big(-\frac{bd^2}{2qD^2}\Big)-e\Big(-\frac{bx^2}{2qD^2}\Big)\Big|dx
\\
\lesssim \sum_{d=0}^{D-1}\int_{d}^{d+1}\Big|\frac{bd^2}{2qD^2}-\frac{bx^2}{2qD^2}\Big|dx\lesssim \frac{|b|}{qD^2}\sum_{d=0}^{D-1}(2d+1)\lesssim \frac{|b|}{q}\text{,}
\end{multline}
and thus 
\begin{equation}\label{estforE_1}
|E_1|\lesssim \frac{|I||b|}{Dq}\text{.}
\end{equation}\\
\textbf{Estimates for $E_2$.} Here we use Lemma~$\ref{easylemma2}$ to obtain
\begin{equation}\label{estforE_2}
|E_2|\le \sum_{r,s\in[2qk_2]}\sum_{0\le d<D-1}\Big||I_{r,s,d}|-\frac{|I|}{4q^2k_2^2D}\Big| \lesssim_k q^2D(1+\log|I|)\text{.}
\end{equation}
Combining \eqref{estforE_1}, \eqref{estforE_2}, as well as \eqref{estforS_2}, we get
\begin{multline}
\bigg|\mathbb{E}_{n\in I}e\Big(\frac{a+b\sqrt{k}}{q}n\lfloor n\sqrt{k}\rfloor\Big)-\G_k(a,b,q)\F\bigg(\frac{b}{2q}\bigg)\bigg|\lesssim_k \frac{\frac{|I||b|}{qD}+|b|q(1+\log |I|)}{|I|}+\frac{|E_1|+|E_2|}{|I|}
\\
\lesssim \frac{|b|}{qD}+|b|q\frac{1+\log |I|}{|I|}+q^2D\frac{1+\log|I|}{|I|}\text{.}
\end{multline} 
A convenient choice here is to let $D\coloneqq \lfloor |I|^{1/2}\rfloor$ to obtain
\begin{multline}
\bigg|\mathbb{E}_{n\in I}e\Big(\frac{a+b\sqrt{k}}{q}n\lfloor n\sqrt{k}\rfloor\Big)-\G_k(a,b,q)\F\bigg(\frac{b}{2q}\bigg)\bigg|
\\
\lesssim_k \frac{|b|}{|I|^{1/2}}+|b|q\frac{1+\log |I|}{|I|}+q^2\frac{1+\log|I|}{|I|^{1/2}}\lesssim q^2(1+|b|)(1+\log |I|)|I|^{-1/2} \text{,}\quad\text{as desired.}
\end{multline}  
\end{proof}
We are now ready to prove Proposition~$\ref{MainMajPro}$.
\begin{proof}[Proof of Proposition~$\ref{MainMajPro}$.]
Let $\kappa\in(0,1)$ and $M\coloneqq \lfloor N^{\kappa}\rfloor$. For every $j\in\mathbb{N}$ we define $I_j=[M]+(j-1)M$ so that
\[
[N]=\Big(\bigcup_{j=1}^{\lfloor N/M\rfloor}I_j\Big)\cup \Big\{M\Big\lfloor\frac{N}{M}\Big\rfloor+1,\dotsc,\lfloor N\rfloor\Big\}\text{,}
\]
and note that the last set in the union above has size at most $\lfloor N\rfloor -M\big\lfloor\frac{N}{M}\big\rfloor< N-M(N/M-1)=M$, and thus we have
\begin{equation}\label{pt1}
\bigg|\frac{1}{\lfloor N\rfloor}\sum_{n\in [N]}e\Big(\Big(\frac{a+b\sqrt{k}}{q}+t\Big)n\lfloor n\sqrt{k}\rfloor\Big)-\frac{1}{\lfloor N\rfloor}\sum_{j\in[N/M]}\sum_{n\in I_j}e\Big(\Big(\frac{a+b\sqrt{k}}{q}+t\Big)n\lfloor n\sqrt{k}\rfloor\Big)\bigg|\lesssim\frac{M}{N}\text{.}
\end{equation}
We focus on each $I_j$,  $j\in[N/M]$, and note that
\begin{multline}
\bigg|\sum_{n\in I_j}e\Big(\Big(\frac{a+b\sqrt{k}}{q}+t\Big)n\lfloor n\sqrt{k}\rfloor\Big)-\sum_{n\in I_j}e\Big(\frac{a+b\sqrt{k}}{q}n\lfloor n\sqrt{k}\rfloor\Big)e\big(t\sqrt{k}j^2M^2\big)\bigg|
\\
=\bigg|\sum_{n\in I_j}e\Big(\frac{a+b\sqrt{k}}{q}n\lfloor n\sqrt{k}\rfloor\Big)\Big(e\big(tn\lfloor n\sqrt{k}\rfloor\big)-e\big(t\sqrt{k}j^2M^2\big)\Big)\bigg|\lesssim\sum_{n\in I_j}|t|\big|n\lfloor n\sqrt{k}\rfloor-\sqrt{k}j^2M^2\big|\lesssim_k|t|NM^2\text{,} 
\end{multline}
where for the last estimate we used the fact that for all $n\in I_j=\{(j-1)M+1,\dotsc,(j-1)M+M=jM\}$ we have 
\begin{multline}
\big|n\lfloor n\sqrt{k}\rfloor-\sqrt{k}j^2M^2\big|\le \sqrt{k}j^2M^2-(j-1)M\lfloor (j-1)M\sqrt{k}\rfloor
\\
\le \sqrt{k}j^2M^2-(j-1)M\big( (j-1)M\sqrt{k}-1\big) =\sqrt{k}j^2M^2-\sqrt{k}(j-1)^2M^2+(j-1)M 
\\
\lesssim \sqrt{k}M^2j+jM\lesssim_k M^2j \lesssim NM\text{.}
\end{multline}
But then we see that 
\begin{multline}\label{pt2}
\bigg|\frac{1}{\lfloor N\rfloor}\sum_{j\in[N/M]}\sum_{n\in I_j}e\Big(\Big(\frac{a+b\sqrt{k}}{q}+t\Big)n\lfloor n\sqrt{k}\rfloor\Big)-\frac{1}{\lfloor N\rfloor}\sum_{j\in[N/M]}\sum_{n\in I_j}e\Big(\frac{a+b\sqrt{k}}{q}n\lfloor n\sqrt{k}\rfloor\Big)e\big(t\sqrt{k}j^2M^2\big)\bigg|
\\
\lesssim_k \frac{1}{N}\sum_{j\in[N/M]}|t|NM^2\lesssim|t|NM\text{.}
\end{multline}
Finally, since $|I_j|=M$, we get
\begin{multline}
\frac{1}{\lfloor N\rfloor}\sum_{j\in[N/M]}\sum_{n\in I_j}e\Big(\frac{a+b\sqrt{k}}{q}n\lfloor n\sqrt{k}\rfloor\Big)e\big(t\sqrt{k}j^2M^2\big)
\\
=\frac{M}{\lfloor N\rfloor}\sum_{j\in[N/M]}\bigg(\frac{1}{|I_j|}\sum_{n\in I_j}e\Big(\frac{a+b\sqrt{k}}{q}n\lfloor n\sqrt{k}\rfloor\Big)\bigg)e\big(t\sqrt{k}j^2M^2\big)
\\
=\frac{M}{\lfloor N\rfloor}\sum_{j\in[N/M]}\G_k(a,b,q)\F\bigg(\frac{b}{2q}\bigg)e\big(t\sqrt{k}j^2M^2\big)
\\
+\frac{M}{\lfloor N\rfloor}\sum_{j\in[N/M]}\bigg(\frac{1}{|I_j|}\sum_{n\in I_j}e\Big(\frac{a+b\sqrt{k}}{q}n\lfloor n\sqrt{k}\rfloor\Big)-\G_k(a,b,q)\F\bigg(\frac{b}{2q}\bigg)\bigg)e\big(t\sqrt{k}j^2M^2\big)
\\
=\G_k(a,b,q)\F\bigg(\frac{b}{2q}\bigg)V_{N;k}(t)
\\
+\G_k(a,b,q)\F\bigg(\frac{b}{2q}\bigg)\bigg(\frac{M}{\lfloor N\rfloor}\sum_{j\in[N/M]}e\big(t\sqrt{k}j^2M^2\big)-\frac{1}{N}\int_0^Ne\big(t\sqrt{k}x^2\big)dx\bigg)
\\
+\frac{M}{\lfloor N\rfloor}\sum_{j\in[N/M]}\bigg(\frac{1}{|I_j|}\sum_{n\in I_j}e\Big(\frac{a+b\sqrt{k}}{q}n\lfloor n\sqrt{k}\rfloor\Big)-\G_k(a,b,q)\F\bigg(\frac{b}{2q}\bigg)\bigg)e\big(t\sqrt{k}j^2M^2\big)
\\
\eqqcolon\G_k(a,b,q)\F\bigg(\frac{b}{2q}\bigg)V_{N;k}(t)+E_1+E_2 \text{,}
\end{multline}
and we remind the reader that $V_{N;k}$ is defined in \eqref{Vdef}. It remains to bound $E_1$ and $E_2$.
\\\,\\
\textbf{Estimates for $E_1$.} We begin by noting that
\begin{multline}
\Big|\sum_{j\in[N/M]}Me(t\sqrt{k}j^2M^2)-\int_{0}^{\lfloor N/M\rfloor M}e(t\sqrt{k}x^2)dx\Big|
\\
=\Big|\sum_{j\in[N/M]}\int_{(j-1)M}^{jM}\big(e(t\sqrt{k}j^2M^2)-e(t\sqrt{k}x^2)\big)dx\Big|
\lesssim_k\sum_{j\in[N/M]}|t|M^3j\lesssim |t|M^3N^2/M^2=|t|MN^2\text{.}
\end{multline}
Also, since $\lfloor N/M\rfloor M> (N/M-1) M=N-M$, we get
\[
\bigg|\frac{1}{\lfloor N\rfloor}\int_{\lfloor N/M\rfloor M}^Ne(t\sqrt{k}x^2)dx\bigg|\lesssim \frac{N-\lfloor N/M\rfloor M}{N}\le \frac{M}{N}\text{,}
\]
and thus, we can bound as follows
\begin{multline}\label{pt3}
|E_1|\lesssim \bigg|\frac{M}{\lfloor N\rfloor}\sum_{j\in[N/M]}e\big(t\sqrt{k}j^2M^2\big)-\frac{1}{N}\int_0^Ne\big(t\sqrt{k}x^2\big)dx\bigg|
\\
\le \bigg|\frac{1}{\lfloor N\rfloor}\sum_{j\in[N/M]}Me\big(t\sqrt{k}j^2M^2\big)-\frac{1}{\lfloor N\rfloor}\int_0^{\lfloor N/M\rfloor M }e\big(t\sqrt{k}x^2\big)dx\bigg|
\\
+\bigg|\frac{1}{\lfloor N\rfloor}\int_0^{\lfloor N/M\rfloor M }e\big(t\sqrt{k}x^2\big)dx-\frac{1}{\lfloor N\rfloor}\int_0^{N}e\big(t\sqrt{k}x^2\big)dx\bigg|
\\
+\bigg|\frac{1}{\lfloor N\rfloor}\int_0^{N}e\big(t\sqrt{k}x^2\big)dx-\frac{1}{N}\int_0^{N}e\big(t\sqrt{k}x^2\big)dx\bigg|
\\
\lesssim |t|MN +\frac{M}{N}+\bigg(\frac{1}{\lfloor N
\rfloor}-\frac{1}{N}\bigg)N=|t|MN +\frac{M}{N}+\frac{N-\lfloor N\rfloor}{\lfloor N
\rfloor}\lesssim|t|MN +\frac{M}{N} \text{.}
\end{multline}\\
\textbf{Estimates for $E_2$.} By Proposition~$\ref{t=0}$ we get that 
\begin{equation}\label{BforE2}
|E_2|\lesssim_k \frac{1}{N/M}\sum_{j\in[N/M]}q^2(1+|b|)(1+\log M)M^{-1/2}=q^2(1+|b|)(1+\log M)M^{-1/2}\text{.}
\end{equation}
Remembering that $M=\lfloor N^{\kappa}\rfloor$ for some $\kappa\in(0,1)$ and combining \eqref{pt1}, \eqref{pt2}, \eqref{pt3}, \eqref{BforE2} yield that the left-hand side of \eqref{freetarc} is bounded by
\[
\lesssim_{k,\kappa} N^{\kappa-1}+|t|N^{1+\kappa}+q^2(1+|b|)(1+\log N)N^{-\kappa/2}\text{,}\quad\text{and the proof is complete.}
\]
\end{proof} 
\subsection{Estimates for the factors $\G_k$, $\F$, $V_{X;k}$.}The goal of this section is to collect some bounds for $\G_k,\F,V_{X;k}$ that will be used in the sequel. 
\begin{lemma}[Estimates for $\F$ and $V_{X;k}$]\label{frVXbounds}Assume that $k\in(0,\infty)$. Then for every $\xi\in\mathbb{R}\setminus\{0\}$ and $X\in[1,\infty)$ we have that
\begin{equation}\label{VXFest}
|V_{X;k}(\xi)|\lesssim k^{-1/4}|\xi X^2|^{-1/2}\text{,}\quad|V_{X;k}(\xi)-1|\lesssim k^{1/2}|\xi X^2|\quad\text{and}\quad
|\F(\xi)|\lesssim \min\big(1,|\xi|^{-1/2}\big)\text{,}
\end{equation}
where the implied constants are absolute.
\end{lemma}
\begin{proof}We begin with the third estimate. Clearly, we have that $|\F|\le 1$, and an application of van der Corput lemma, see Proposition~2.6.7(b) in \cite{Graf}, yields that
\[
|\F(\xi)|\lesssim |\xi|^{-1/2}\text{,}\quad\text{as desired.}
\]
For the first estimate we may use the previous estimate after a change of variables as follows
\[
|V_{X;k}(\xi)|=\Big|\int_{0}^1e(\xi\sqrt{k}(Xu)^2)du\Big|=\Big|\int_{0}^1e\big((\xi\sqrt{k}X^2)u^2\big)du\Big|\lesssim |\sqrt{k}\xi X^2|^{-1/2}=k^{-1/4}|\xi X^2|^{-1/2}\text{,}
\]
and for the second one we note that
\[
|V_{X;k}(\xi)-1|\le\frac{1}{X}\int_{0}^X|e(\xi\sqrt{k}x^2)-e(0)|dx\lesssim \frac{\sqrt{k}|\xi|}{X}\int_{0}^Xx^2dx\lesssim k^{1/2}|\xi X^2|\text{.}
\]
\end{proof}
We now turn our attention to $G_k$. To establish the appropriate estimate we use some standard properties of the so-called generalized quadratic Gauss sum. For the convenience of the reader we collect them here.
\begin{lemma}[Generalized quadratic Gauss sum]\label{genquadgausssum}For every $a,b\in\mathbb{Z}$ and $q\in\mathbb{N}$ let
\[
g(a,b,q)\coloneqq \mathbb{E}_{r\in[q]}e\Big(\frac{ar^2+br}{q}\Big)\text{.}
\]
Then the following hold:
\begin{enumerate}
\item[\normalfont{(i)}] for every $a,b\in\mathbb{Z}$ and $q\in\mathbb{N}$ such that $\gcd(a,q)=1$ we have
\begin{equation*}
|g(a,b,q)|\lesssim q^{-1/2}\text{,}
\end{equation*}
where the implied constant is absolute.
\item[\normalfont{(ii)}]for every $a,b\in\mathbb{Z}$ and $c,q\in\mathbb{N}$ such that $\gcd(a,q)=1$ we have
\begin{equation*}
g(ca,b,cq)=g\big(a,b/c,q\big)1_{c|b}\text{.}
\end{equation*}
\end{enumerate}
\end{lemma}
\begin{proof}
The proof is standard, see for example Lemma~2.5 and Lemma~2.6 in \cite{NL}. 
\end{proof}
\begin{proposition}[Bounds for $\G_k$]\label{GaussSumstype}
Assume $k_1,k_2\in\mathbb{N}$ are such that $\gcd(k_1,k_2)=1$. Then there exists a positive constant $C=C(k_1,k_2)$ such that for all $q\in\mathbb{N}$ and $a,b\in\mathbb{Z}$ such that $\gcd(a,b,q)=1$ we have
\begin{equation}\label{GaussBound}
|\G_{k_1/k_2}(a,b,q)|\le C q^{-1/2}\text{.}
\end{equation}
\end{proposition}
\begin{proof}We allow here all implicit constants to depend on $k_1,k_2$. We have
\begin{multline}\label{reductiontosingleaverage}\G_k(a,b,q)=\mathbb{E}_{r,s\in[2qk_2]}e\Big(\frac{a}{q}rs+\frac{b}{2qk_2}(k_1r^2+k_2s^2)\Big)
\\
=\frac{1}{2qk_2}\sum_{s\in[2qk_2]}e\Big(\frac{bs^2}{2q}\Big)\frac{1}{2qk_2}\sum_{r\in[2qk_2]}e\Big(\frac{2k_2as}{2qk_2}r+\frac{bk_1}{2qk_2}r^2\Big)=\frac{1}{2qk_2}\sum_{s\in[2qk_2]}e\Big(\frac{bs^2}{2q}\Big)g(bk_1,2k_2as,2qk_2)\text{.}
\end{multline}
Let $c_k=c_k(b,q)\coloneqq \gcd(bk_1,2qk_2)$, and then note that there exists $b'$ and $q'$ with $\gcd(b',q')=1$ such that $bk_1=c_kb'$ and $ 2qk_2=c_kq'$. Also, note that if we let $c\coloneqq \gcd(b,q)$, then $\gcd(c,a)=1$, $c|c_k$ and
\[
c_k=\gcd(bk_1,2qk_2)\le\gcd(2k_1k_2b,2k_1k_2q)=2k_1k_2\gcd(b,q)=2k_1k_2c\text{,}
\]
and thus $c\ge\frac{c_k}{2k_1k_2}$. We now split the analysis to two cases.
\\
\textbf{Case 1: $a\neq 0$.} Returning to \eqref{reductiontosingleaverage} and applying Lemma~$\ref{genquadgausssum}$ yield
\begin{multline}|\G_k(a,b,q)|
\lesssim\frac{1}{q}\sum_{s\in[2qk_2]}|g(c_kb',2k_2as,c_kq')|=\frac{1}{q}\sum_{s\in[2qk_2]}\Big|g\Big(b',\frac{2k_2as}{c_k},q'\Big)\Big|1_{c_k|2k_2as}
\\
\lesssim\frac{1}{q}\sum_{s\in[2qk_2]}(q')^{-1/2}1_{c|a2sk_2}=\frac{(q')^{-1/2}}{q}\sum_{s\in[2qk_2]}1_{c|2sk_2}
\\
\lesssim \frac{(q')^{-1/2}}{c}\le\frac{(q')^{-1/2}}{c_k/(2k_1k_2)}\lesssim c_k^{-1/2}(q'c_k)^{-1/2}\le q^{-1/2}\text{,}
\end{multline}
as desired. Note that we have used the fact that $a\neq 0$ for the equality in the second line.
\\
\textbf{Case 2: $a=0$.} We will now have that $c=\gcd(b,q)=1$, and thus $c_k\le 2k_1k_2$. But then in a similar manner we may estimate as follows
\begin{equation}|\G_k(a,b,q)|
\lesssim\frac{1}{q}\sum_{s\in[2qk_2]}\Big|g\Big(b',\frac{k_22as}{c_k},q'\Big)\Big|1_{c_k|2k_2as}\lesssim (q')^{-1/2}
=c_k^{1/2}(c_kq')^{-1/2}\lesssim q^{-1/2}\text{,}
\end{equation}
and the proof is complete.
\end{proof}
\section{Minor arc estimates}\label{minsection}
This section is devoted to the proof of the minor arc estimates. In contrast to the situation for the major arcs, it is more convenient here to immediately express our estimates using the notation introduced in section~3. The main result of this section is the following.
\begin{proposition}[Minor arc estimate]\label{minarcnew}
Assume $k\in\mathbb{Q}_{>0}$ is such that $\sqrt{k}\not\in\mathbb{Q}$. Then there exists an absolute positive constant $c_0$ such that for all $\gamma,\gamma'\in(0,1/10)$ with $\frac{\gamma}{\gamma'}<c_0$ the following holds. There exist positive constants $C=C(k,\gamma,\gamma')$ and $\chi=\chi(\gamma,\gamma')$ such that for every $N\in[1,\infty)$ and $\xi\in\mathfrak{m}(N;k,\gamma,\gamma')$ we have
\begin{equation}\label{notequidistribution}
\big|\mathbb{E}_{n\in[N]}e(\xi n\lfloor n\sqrt{k}\rfloor)\big|\le CN^{-\chi}\text{.}
\end{equation}
\end{proposition}
The proposition is an immediate consequence of the the following two lemmas.
\begin{lemma}[Estimate for $\mathfrak{m}_1$]\label{minLemma1}
Assume $k\in\mathbb{Q}_{>0}$ is such that $\sqrt{k}\not\in\mathbb{Q}$ and let $\gamma,\gamma'\in(0,1/10)$. Then there exists positive constant $C=C(k,\gamma,\gamma')$ and $\chi=\chi(\gamma)$ such that for every $N\in[1,\infty)$ and $\xi\in\mathfrak{m}_1(N;k,\gamma,\gamma')$ we have
\begin{equation}\label{notequidistribution1}
\big|\mathbb{E}_{n\in[N]}e(\xi n\lfloor n\sqrt{k}\rfloor)\big|\le CN^{-\chi}\text{.}
\end{equation}
\end{lemma}
\begin{lemma}[Estimate for $\mathfrak{m}_2$]\label{minLemma2}
Assume $k\in\mathbb{Q}_{>0}$ is such that $\sqrt{k}\not\in\mathbb{Q}$. Then there exists an absolute positive constant $c_0$ such that for all $\gamma,\gamma'\in(0,1/10)$ with $\frac{\gamma}{\gamma'}<c_0$ the following holds. There exists positive constant $C=C(k,\gamma,\gamma')$ and $\chi=\chi(\gamma,\gamma')$ such that for every $N\in[1,\infty)$ and $\xi\in\mathfrak{m}_2(N;k,\gamma,\gamma')$ we have
\begin{equation}\label{notequidistribution1}
\big|\mathbb{E}_{n\in[N]}e(\xi n\lfloor n\sqrt{k}\rfloor)\big|\le CN^{-\chi}\text{.}
\end{equation}
\end{lemma}
We give the proofs of lemmas~$\ref{minLemma1}$ and $\ref{minLemma2}$ in subsections~$\ref{m1subsection}$ and $\ref{m2subsection}$ respectively. As discussed earlier, see subsection~$\ref{minarcstrat}$, we will appeal to the quantitative Leibman theorem from \cite{BG} and in the following subsection we introduce some notation, make the precise formulation of the theorem and establish some results used in the sequel. 
\subsection{Quantitative behaviour of polynomial orbits on the Heisenberg nilmanifold}\label{HNil}
The minor arc analysis is based on the following result on the quantitative equidistribution of polynomial orbits on nilmanifolds which we state below, see Theorem~2.9 in \cite{BG}. The definitions of the relevant terms appearing can be found in sections~1 and 2 in \cite{BG} and we chose not to repeat the discussion here.
\begin{theorem}[Quantitative Leibman theorem \cite{BG}]\label{QLT} Let $m,d\in\mathbb{N}_0$, $\delta\in(0,1/2)$ and $N\ge 1$. Assume $G/\Gamma$ is an $m$-dimensional nilmanifold equipped with a filtration $G_{\bullet}$ of degree $d$ and a $\frac{1}{\delta}$-rational Mal'cev basis $\mathcal{X}$ adapted to the filtration and let $g\in \poly(\mathbb{Z},G_{\bullet})$. Then there exists a constant $C_0=C_0(m,d)$ such that the following holds. If $\big(g(n)\Gamma\big)_{n\in[N]}$ is not $\delta$-equidistributed in $G/\Gamma$, then there exists a nontrivial horizontal character $\eta\colon G\to \mathbb{T}$ with $|\eta|\lesssim\delta^{-C_0}$ and such that 
\[
\|\eta\circ g\|_{C^{\infty}[N]}\lesssim\delta^{-C_0}\text{,}
\]
and the implied constants are absolute.
\end{theorem}
 We will use this in two rather concrete setups, and in this subsection we simply wish to give a brief idea on how the above theorem will translate for our considerations. For the proof of Lemma~$\ref{minLemma1}$ we apply the theorem above for the  for the 3-dimensional Heisenberg nilmanifold, namely for $G/\Gamma$ where
\[
G=\Big\{
\left(\begin{smallmatrix}1&x&z
\\
0&1&y\\
0&0&1
\end{smallmatrix}\right)
,\,\,x,y,z\in\mathbb{R}\Big\}\quad\text{and}\quad\Gamma=\Big\{\left(\begin{smallmatrix}1&x&z
\\
0&1&y\\
0&0&1
\end{smallmatrix}\right),\,\,x,y,z\in\mathbb{Z}\Big\}\text{,}
\] with its standard Mal'cev basis, see section~5 in \cite{BG}, and the coordinate map $\psi: G \to \mathbb R^3$ with
\[
\psi\Big(\left(\begin{smallmatrix}
1 & x & z\\
0 & 1 & y\\
0 & 0 & 1
\end{smallmatrix}\right)\Big) = (x,y, z-xy)\text{.}
\]
For convenience, whenever $g\in G$, we denote $x_g=g_{1,2}$, $y_g=g_{2,3}$ and $z_g=g_{1,3}$. Let 
\begin{equation}\label{funddom}
\mathcal{F}\coloneqq\{g\in G:\,x_g,y_g,z_g\in[0,1)\}\text{,}
\end{equation}
and note that it is a fundamental domain in the sense that for each $g
\in G$ there exists a unique $g'\in \mathcal{F}$ such that $g\Gamma=g'\Gamma$. Establishing that this is indeed a fundamental domain is standard but, nevertheless, we give a short proof below to give an idea on why the Heisenberg nilmanifold is the suitable object for understanding exponential sums with phases $\big(\xi n\lfloor n\sqrt{k}\rfloor\big)_{n\in[N]}$. On the one hand, we have
\begin{equation}\label{calctoshow}
\begin{pmatrix}
1 & x & z\\
0 & 1 & y\\
0 & 0 & 1
\end{pmatrix}\Gamma=\begin{pmatrix}
1 & x & z\\
0 & 1 & y\\
0 & 0 & 1
\end{pmatrix}\begin{pmatrix}
1 & -\lfloor x\rfloor & -\lfloor z-x\lfloor y\rfloor\rfloor\\
0 & 1 & -\lfloor y\rfloor\\
0 & 0 & 1
\end{pmatrix}\Gamma=
\begin{pmatrix}
1 & \{x\} & \{z-x\lfloor y\rfloor\}\\
0 & 1 & \{y\}\\
0 & 0 & 1
\end{pmatrix}\Gamma\text{.}
\end{equation}
On the other hand, if $f,f'\in \mathcal{F}$ are such that $f\Gamma= f'\Gamma$, then there exists $m,n,l\in\mathbb{Z}$ such that
\[
\begin{pmatrix}
1 & x_{f} & z_{f}\\
0 & 1 & y_{f}\\
0 & 0 & 1
\end{pmatrix}=\begin{pmatrix}
1 & x_{f'} & z_{f'}\\
0 & 1 & y_{f'}\\
0 & 0 & 1
\end{pmatrix}\begin{pmatrix}
1 & m & n\\
0 & 1 & l\\
0 & 0 & 1
\end{pmatrix}=\begin{pmatrix}
1 & x_{f'}+m & z_{f'}+n+x_{f'}l\\
0 & 1 & y_{f'}+l\\
0 & 0 & 1
\end{pmatrix}\text{,}
\]
and since $f,f'\in\mathcal{F}$, we must have that $m,l=0$ and thus $n=0$. Thus $f=f'$ as desired.

We equip the nilmanifold with a metric $d_{G/\Gamma}$, the definition of which is given with respect to $\psi$ and can be found in \cite{BG}, see Definition~2.2 from the aforementioned paper.  

Finally, we equip $G/\Gamma$ with its unique normalized Haar measure, which in our case corresponds to the Lebesgue measure. 

For every $N\in\mathbb{N}$, $\delta>0$, we say that $\big(g(n)\Gamma\big)_{n\in[N]}$ is $\delta$-equidistributed in $G/\Gamma$ if we have
\[
\Big|\mathbb{E}_{n\in[N]}F\big(g(n)\Gamma\big)-\int_{G/\Gamma}F\Big|\le \delta\|F\|_{\Lip}\text{,}
\] 
for all Lipschitz functions $F\colon G/\Gamma\to
\mathbb{C}$, where
\[
\|F\|_{\Lip}\coloneqq \|F\|_{L^{\infty}}+\sup_{\substack{x,y\in G/\Gamma\\x\neq y}}\frac{|F(x)-F(y)|}{d_{G/\Gamma}(x,y)}\text{.}
\]
We also note that for every nontrivial horizontal character $\eta\colon G\to\mathbb{T}$ there exists a unique $l\in\mathbb{Z}^3\setminus\{(0,0,0)\}$ such that $\eta(g)=l\cdot \psi(g)$ for all $g\in G$ and we have $|\eta|\coloneqq \|l\|_{\ell^{\infty}}$, see Definition~2.6 in \cite{BG}.

In view of the calculation in \eqref{calctoshow} and given our aim to understand phases of the form $(\xi n\lfloor n\sqrt{k}\rfloor)_{n\in[N]}$ it is natural to define the following function
\[
F(g\Gamma)=e(z_g-x_g\lfloor y_g\rfloor)\text{,}
\]
which can be equivalently defined as the unique function $F\colon G/\Gamma\to \mathbb{C}$ such that \begin{equation}\label{Fdefinition2}
F(f\Gamma)=e(z_f)\text{, for all $f\in\mathcal{F}$.}
\end{equation}
It is not continuous and we will require a suitable approximation of the function which we formulate below.
\begin{lemma}\label{liptogether}Let $G/\Gamma$ be the $3$-dimensional Heisenberg nilmanifold\footnote{Here and in the sequel, unless we specify otherwise, we equip the Heisenberg nilmanifold with its standard Mal'cev basis, which is adapted to
 the lower central series filtration of $G$, the standard coordinate map $\psi$ and its corresponding metric $d_{G/\Gamma}$, and the Haar measure, which is the Lebesgue measure, see section~2 and section~5 in \cite{BG} .}.  For every $\tau\in(0,1/100)$, let 
$\widetilde{\chi}_{\tau}\colon \mathcal{F}\to[0,1]$ be such that
\begin{equation}\label{defchitilde}
\widetilde{\chi}_{\tau}(g)=\left\{
\begin{array}{ll}
1, &\text{if}\quad \|x_g \|>\tau/5\text{ and }\|y_g \|>\tau/5 \\
      0, &\text{if}\quad \|x_g \|\le\tau/10\text{ or }\|y_g \|\le\tau/10\\
\end{array} 
\right. 
\end{equation}
and such that \begin{equation}\label{chielip}|\widetilde{\chi}_{\tau}(g)-\widetilde{\chi}_{\tau}(g')|\lesssim \tau^{-1}\max(\|x_g-x_{g'}\|,\|y_g-y_{g'}\|)\text{ for all }g,g'\in\mathcal{F}\text{,}
\end{equation}
where the implied constant is absolute, and define $\widetilde{F}_{\tau}\colon\mathcal{F}\to\mathbb{C}$ by $\widetilde{F}_{\tau}(g)=\widetilde{\chi}_{\tau}(g)e(z_g)$.

Define $F_{\tau}\colon G/\Gamma\to\mathbb{C}$ and $\chi_{\tau}\colon G/\Gamma\to\mathbb{C}$ such that 
\begin{equation}\label{ftdef}
F_{\tau}(g\Gamma)=\widetilde{F}_{\tau}(f)=e(z_{f})\widetilde{\chi}_{\tau}(f)\quad\text{and}\quad \chi_{\tau}(g\Gamma)=\widetilde{\chi}_{\tau}(f)\text{,}
\end{equation}
where $f$ is the unique element in $\mathcal{F}$ such that $g\Gamma=f\Gamma$. Then we have that 
\begin{equation}\label{liptechnical}
\|F_{\tau}\|_{\Lip}\lesssim\tau^{-1}\quad\text{and}\quad\|\chi_{\tau}\|_{\Lip}\lesssim \tau^{-1}\text{,}
\end{equation}
where the implied constant is absolute.
\end{lemma}
\begin{proof}
Note that $\|F_{\tau}\|_{\infty},\|\chi_{\tau}\|_{\infty}\le 1\le \tau^{-1}$ and thus we may focus on the differences. Moreover, for $x,y\in G/\Gamma$ such that $d_{G/\Gamma}(x,y)\ge\tau/50$ we immediately get that 
\[
|\chi_{\tau}(x)-\chi_{\tau}(y)|,|F_{\tau}(x)-F_{\tau}(y)|\le 2\lesssim \tau^{-1}d_{G/\Gamma}(x,y)\text{,}
\]
so we may assume from now on that $x,y\in G/\Gamma$ are such that $d_{G/\Gamma}(x,y)<\tau/50$. There exists a unique $(g,g')\in\mathcal{F}^2$ such that $x=g\Gamma$ and $y=g'\Gamma$. Thus
\begin{equation}\label{chilipproof}
|\chi_{\tau}(x)-\chi_{\tau}(y)|=|\widetilde{\chi}_{\tau}(g)-\widetilde{\chi}_{\tau}(g')|
\end{equation}
and
\begin{multline}\label{bSS}
|F_{\tau}(x)-F_{\tau}(y)|=|\widetilde{F}_{\tau}(g)-\widetilde{F}_{\tau}(g')|=|\widetilde{\chi}_{\tau}(g)e(z_g)-\widetilde{\chi}_{\tau}(g')e(z_{g'})|
\\
\le|\widetilde{\chi}_{\tau}(g)e(z_g)-\widetilde{\chi}_{\tau}(g)e(z_{g'})|+|\widetilde{\chi}_{\tau}(g)e(z_{g'})-\widetilde{\chi}_{\tau}(g')e(z_{g'})|\le \widetilde{\chi}_{\tau}(g)|e(z_g)-e(z_{g'})|+|\widetilde{\chi}_{\tau}(g)-\widetilde{\chi}_{\tau}(g')|\text{.}
\end{multline}
\\
\textbf{Claim 1:} For every $t,t'\in G$ we have that  $d(t,t')\ge \max(|x_t-x_{t'}|,|y_t-y_{t'}|)$ and $d_{G/\Gamma}(t\Gamma,t'\Gamma)\ge \max(\|x_t-x_{t'}\|,\|y_t-y_{t'}\|)$, see Definition~2.2 from \cite{BG} for the definition of $d$ and $d_{G/\Gamma}$.
\begin{proof}
Note that for every $s,l\in G$ we have 
\[
\min(\|\psi(sl^{-1})\|_{\infty},\|\psi(ls^{-1})\|_{\infty})\ge \max(|x_s-x_l|,|y_s-y_l|)\text{.}
\]
Then for any $t,t'\in G$ and any $t=t_0,\dotsc,t_{N+1}=t'\in G$, we have that 
\[
\sum_{i=0}^N \min(\|\psi(t_it_{i+1}^{-1})\|_{\infty},\|\psi(t_{i+1}t_i^{-1})\|_{\infty})\ge \sum_{i=0}^N\max(|x_{t_i}-x_{t_{i+1}}|,|y_{t_i}-y_{t_{i+1}}|)\ge \max(|x_t-x_{t'}|,|y_t-y_{t'}|)\text{,}
\]
and the first assertion immediately follows. For the second one we note that by our first bound we get
\begin{multline}
d_{G/\Gamma}(t\Gamma,t'\Gamma)=\inf\{d(s,s'):\,s\equiv_{\Gamma} t\text{ and }s'\equiv_{\Gamma} t'\}
\\
\ge\inf\{\max(|x_s-x_{s'}|,|y_s-y_{s'}|):\,s\equiv_{\Gamma} t\text{ and }s'\equiv_{\Gamma} t'\}
=\max(\|x_t-x_{t'}\|,\|y_t-y_{t'}\|)\text{,}\quad\text{as desired.}
\end{multline}
\end{proof}
Using \eqref{chielip} and \eqref{chilipproof}, the claim immediately implies that $|\chi_{\tau}(x)-\chi_{\tau}(y)|\lesssim \tau^{-1}d_{G/\Gamma}(x,y)$, since
\[
|\chi_{\tau}(x)-\chi_{\tau}(y)|=|\widetilde{\chi}_{\tau}(g)-\widetilde{\chi}_{\tau}(g')|\lesssim \tau^{-1}\max(\|x_g-x_{g'}\|,\|y_g-y_{g'}\|)\le \tau^{-1}d_{G/\Gamma}(x,y)\text{.}
\]
Note that one the one hand, this concludes the proof of $\|\chi_{\tau}\|_{\Lip}\lesssim \tau^{-1}$, and on the other hand, it handles the second summand in \eqref{bSS}. For the first summand in the last expression in \eqref{bSS} we will use a similar claim.\\
\textbf{Claim 2:} For every $t,t'\in G$ such that $y_t,y_{t'}\in(-1,1)$ we have that $d(t,t')\ge \frac{1}{4}\|z_t-z_t'\|$.
\begin{proof}
Firstly, we note that for every $s,l\in G$ we have 
\begin{multline}
\min(\|\psi(sl^{-1})\|_{\infty},\|\psi(ls^{-1})\|_{\infty})\ge
\min(|z_s-z_l+y_s(x_l-x_s)|,|z_l-z_s+y_l(x_s-x_l)|)
\\
\ge \min(|z_s-z_l|-|y_s||x_l-x_s|,|z_l-z_s|-|y_l||x_l-x_s||)\ge |z_s-z_l|-\max(|y_s|,|y_l|)|x_s-x_l|\text{.}
\end{multline}
Let $t=t_0,\dotsc,t_{N+1}=t'$ in $G$, with $y_t,y_t'\in(-1,1)$. If we assume that $|y_{t_i}|\le 2$ for every $i\in \{0,\dotsc,N+1\}$ then we get 
\[
\sum_{i=0}^N \min\{\|\psi(t_it_{i+1}^{-1})\|_{\infty},\|\psi(t_{i}t_{i+1}^{-1})\|_{\infty}\}\ge \sum_{i=0}^N|z_{t_i}-z_{t_{i+1}}|-2|x_{t_i}-x_{t_{i+1}}|\text{.} 
\]
Now we distinguish two cases.
\\
Case 1: $\sum_{i=0}^N|x_{t_i}-x_{t_{i+1}}|\le \frac{1}{4}\sum_{i=0}^N|z_{t_i}-z_{t_{i+1}}|$. Here we will have that  
\[
\sum_{i=0}^N \min(\|\psi(t_it_{i+1}^{-1})\|_{\infty},\|\psi(t_{i}t_{i+1}^{-1})\|_{\infty})\ge \frac{1}{2}\sum_{i=0}^N|z_{t_i}-z_{t_{i+1}}|\ge \frac{1}{4}|z_t-z_{t'}|\ge \frac{1}{4}\|z_t-z_{t'}\|\text{.}
\]
Case 2: $\sum_{i=0}^N|x_{t_i}-x_{t_{i+1}}|> \frac{1}{4}\sum_{i=0}^N|z_{t_i}-z_{t_{i+1}}|$. Here from the proof of the previous claim we have that
\[
\sum_{i=0}^N \min(\|\psi(t_it_{i+1}^{-1})\|_{\infty},\|\psi(t_{i+1}t_i^{-1})\|_{\infty})\ge \sum_{i=0}^N|x_{t_i}-x_{t_{i+1}}|\ge \frac{1}{4}\sum_{i=0}^N|z_{t_i}-z_{t_{i+1}}|\ge \frac{1}{4}|z_t-z_{t'}|\ge \frac{1}{4}\|z_t-z_{t'}\|\text{,}
\]
and the proof is complete for the case $|y_{t_i}|\le 2$ for all $i\in\{0,\dotsc,N+1\}$.

For the other case, namely, for the case where there exists $i_0\in\{0,\dotsc,N+1\}$ such that $|y_{t_{i_0}}|>2$, we note that
\[
\sum_{i=0}^N \min(\|\psi(t_it_{i+1}^{-1})\|_{\infty},\|\psi(t_{i+1}t_i^{-1})\|_{\infty})\ge \sum_{i=0}^N|y_{t_i}-y_{t_{i+1}}|\ge |y_{t}-y_{t_{i_0}}|+|y_{t_{i_0}}-y_{t'}|\ge 2\ge \frac{1}{4}\|z_t-z_{t'}\|
\]
since $|y_t|,|y_{t'}|\le 1$.
\end{proof}
To conclude it suffices to show that for all $g,g'\in\mathcal{F}$ with $d_{G/\Gamma}(g\Gamma,g'\Gamma)<\tau/50$, we have that $\widetilde{\chi}_{\tau}(g)|e(z_g)-e(z_{g'})|\lesssim d_{G/\Gamma}(g\Gamma,g'\Gamma)$. We may assume that $\|x_g\|>\tau/10$ and $\|y_g\|>\tau/10$ since otherwise the estimate trivially holds by \eqref{defchitilde}. There exists $\gamma_0$ such that $d(g,g'\gamma_0)\le 2d_{G/\Gamma}(g\Gamma,g'\Gamma)<\tau/25$, and note that  $\tau/25>d(g,g'\gamma_0)\ge \max(|x_g-x_{g'\gamma_0}|,|y_g-y_{g'\gamma_0}|)$, by our first claim. Since $\|x_g\|>\tau/10$ and $\|y_g\|>\tau/10$ we have $x_{\gamma_0}=y_{\gamma_0}=0$ and thus $y_{g'\gamma_0}=y_{g'}\in[0,1)$, and the previous claim is applicable, yielding $d(g,g'\gamma_0)\ge \frac{1}{4}\|z_g-z_{g'\gamma_0}\|$. But then, since $y_{\gamma_0}=0$, we get
\[
\widetilde{\chi}_{\tau}(g)|e(z_g)-e(z_{g'})|\lesssim \|z_g-z_{g'}\|=\|z_g-z_{g'\gamma_0}\|\lesssim d(g,g'\gamma_0)\lesssim d_{G/\Gamma}(g\Gamma,g'\Gamma)\text{,}
\]
and the proof is complete.
\end{proof}
We may use the definition of $\delta$-equidistribution and the previous lemma to prove the following.
\begin{lemma}\label{Nilinput}
Let $G/\Gamma$ be the $3$-dimensional Heisenberg nilmanifold. If $a,b\in\mathbb{R}$, $\delta\in(0,1/2)$ and $N\in[1,\infty)$ are such that 
\[
\Big(\Big(\begin{smallmatrix}
1 & a n & 0\\
0 & 1 & b n\\
0 & 0 & 1 
\end{smallmatrix}\Big)\Gamma\Big)_{n\in[N]}
\quad\text{is $\delta$-equidistributed  in $G/\Gamma$,}
\]
then we have
\begin{equation}\label{delta0equid}
\big|\mathbb{E}_{n\in[N]}e(-a n\lfloor nb\rfloor)\big|\lesssim  \delta^{1/2}\text{,}
\end{equation}
where the implied constant is absolute.
\end{lemma}
\begin{proof}
Let $g(n)=\Big(\begin{smallmatrix}
1 & a n & 0\\
0 & 1 & b n\\
0 & 0 & 1 
\end{smallmatrix}\Big)$ and $\tau=\delta^{1/2}$. We use the definition of $\delta$-equidistribution for $\big(g(n)\Gamma\big)_{n\in[N]}$ and the functions $F_{\tau},\chi_{\tau}$ from the previous lemma, see \eqref{ftdef}, to conclude that
\[
\bigg|\mathbb{E}_{n\in[N]}F_{\tau}(g(n)\Gamma)-\int_{G/\Gamma}F_{\tau}\bigg|\le \delta\|F_{\tau}\|_{\Lip}\lesssim \delta\tau^{-1}=\delta^{1/2}
\]
and
\[
\bigg|\mathbb{E}_{n\in[N]}\chi_{\tau}(g(n)\Gamma)-\int_{G/\Gamma}\chi_{\tau}\bigg|\le \delta\|\chi_{\tau}\|_{\Lip}\lesssim \delta\tau^{-1}=\delta^{1/2}\text{,}
\]
where we have used \eqref{liptechnical}.
We also note that the function $F$ defined in \eqref{Fdefinition2} has the following two properties:
\[
F(g(n)\Gamma)=F\Big(\Big(\begin{smallmatrix}
1 & \{a n\} & \{0-a n\lfloor b n\rfloor \}\\
0 & 1 & \{b n\}\\
0 & 0 & 1 
\end{smallmatrix}\Big)\Gamma\Big)=e(-a n\lfloor b n\rfloor )\text{,}
\]
and 
\[
\int_{G/\Gamma}F=\int_{x,y,z\in[0,1)}e(z)dxdydz=0\text{.}
\]
Combining everything we get
\begin{multline}\label{delta0equid}
\big|\mathbb{E}_{n\in[N]}e(-a n\lfloor nb\rfloor)\big|=\bigg|\mathbb{E}_{n\in[N]}F(g(n)\Gamma)-\int_{G/\Gamma}F\bigg|
\\
\le \bigg|\mathbb{E}_{n\in[N]}F(g(n)\Gamma)-\mathbb{E}_{n\in[N]}F_{\tau}(g(n)\Gamma)\bigg|+\bigg|\mathbb{E}_{n\in[N]}F_{\tau}(g(n)\Gamma)-\int_{G/\Gamma}F_{\tau}\bigg|+\bigg|\int_{G/\Gamma}F_{\tau}-\int_{G/\Gamma}F\bigg|
\\
\lesssim \mathbb{E}_{n\in[N]}\big|F(g(n)\Gamma)-F_{\tau}(g(n)\Gamma)\big|+\delta^{1/2}+\int_{G/\Gamma}|F_{\tau}-F|
\\
\le\mathbb{E}_{n\in[N]}\big(1-\chi_{\tau}(g(n)\Gamma)\big)+\delta^{1/2}+\int_{G/\Gamma}\big(1-\chi_{\tau}\big)
\\
\le\bigg|\mathbb{E}_{n\in[N]}\big(1-\chi_{\tau}(g(n)\Gamma)\big)-\int_{G/\Gamma}\big(1-\chi_{\tau}\big)\bigg|+\delta^{1/2}+2\int_{G/\Gamma}\big(1-\chi_{\tau}\big)
\\
\le\bigg|\mathbb{E}_{n\in[N]}\chi_{\tau}(g(n)\Gamma)-\int_{G/\Gamma}\chi_{\tau}\bigg|+\delta^{1/2}+2\int_{G/\Gamma}\big(1-\chi_{\tau}\big)\lesssim \delta^{1/2}+\int_{G/\Gamma}\big(1-\chi_{\tau}\big)\lesssim\delta^{1/2}+\tau\lesssim \delta^{1/2}\text{,}
\end{multline}
where we have used the fact that
\[
\int_{G/\Gamma}(1-\chi_{\tau})\lesssim \tau\text{,} 
\] 
which is a immediate from \eqref{defchitilde} and \eqref{ftdef}, and the proof is complete. 
\end{proof}
\subsection{Proof of Lemma~$\ref{minLemma1}$.}\label{m1subsection} Here we will use the tools introduced in the previous subsection to prove the following lemma, which immediately implies Lemma~$\ref{minLemma1}$.
\begin{lemma}\label{minLemma1'}
Assume $k\in\mathbb{Q}_{>0}$ is such that $\sqrt{k}\not\in\mathbb{Q}$. Then there exist two absolute constant $C_1,C_2$ and a positive constant $C_3(k)$ such that for every $N\in[1,\infty)$, $\delta\in(0,1/2)$ and $\xi\in[-1/2,1/2)$ such that
\begin{equation}\label{biggerdelta}
\big|\mathbb{E}_{n\in[N]}e(\xi n\lfloor n\sqrt{k}\rfloor)\big|>\delta\text{,}
\end{equation}
we have that either $N\le C_3\delta^{-2C_1}$ or there exists $q\in\mathbb{N}$ and $a,b\in\mathbb{Z}$ with $q\le C_3\delta^{-C_1} $, $|b|\le C_3\delta^{-C_1}$ and 
\[
\Big|\xi-\frac{a+b\sqrt{k}}{q}\Big|< C_2\delta^{-C_1}N^{-1}\text{.}
\]
\end{lemma}
\begin{proof}Fix $k\in\mathbb{Q}_{>0}$ such that $\sqrt{k}\not\in \mathbb{Q}$ and note that without loss of generality we may assume that $N\in\mathbb{N}$. Since \eqref{biggerdelta} holds, we have by Lemma~$\ref{Nilinput}$ that there exists an absolute positive constant $C$ such that the sequence $(g_{\xi}(n)\Gamma)_{n\in[N]}$ with
\begin{equation}\label{definitionofg}
g_{\xi}(n)=g_{\xi;k}(n)\coloneqq\begin{pmatrix}
1 & -\xi n & 0\\
0 & 1 & \sqrt{k} n\\
0 & 0 & 1 
\end{pmatrix}=\begin{pmatrix}
1 & -\xi  & -\xi \sqrt{k}\\
0 & 1 & 2\sqrt{k}\\
0 & 0 & 1 
\end{pmatrix}^n\begin{pmatrix}
1 & 0  & 0\\
0 & 1 & -\sqrt{k}\\
0 & 0 & 1
\end{pmatrix}^n
\end{equation}
will not be $\delta^{2}/C$-equidistributed in $G/\Gamma$. We wish to appropriately apply Theorem~$\ref{QLT}$ for the 3-dimensional Heisenberg nilmanifold $G/\Gamma$. It is not difficult to see using the Definition~1.8 from \cite{BG} that $g_{\xi}$  is a polynomial sequence with coefficients in the following filtration
\[
G=G_0=G_1\supseteq G_2\coloneqq \Big\{\Big(\begin{smallmatrix}
1 & 0 & x\\
0 & 1 & 0\\
0 & 0 & 1 
\end{smallmatrix}\Big),\,x\in\mathbb{R}\Big\}\supseteq G_3\coloneqq G_2\supseteq G_4\coloneqq\{\id_G\}\text{,}
\]
and one can check that the standard Mal'cev basis is adapted to this filtration, see Definition~2.1 in \cite{BG}. Thus, since 
$(g_{\xi}(n)\Gamma)_{n\in[N]}$ is not $\delta^2/C$-equidistributed in $G/\Gamma$, we may apply Theorem~$\ref{QLT}$. If we let $C_0=C_0(3,3)$ be the positive constant from Theorem~$\ref{QLT}$ we get that there exists a nontrivial horizontal character $\eta\colon G\to \mathbb{T}$ with $|\eta|\lesssim \delta^{-2C_0}$ and such that $\|\eta\circ g_{\xi}\|_{C^{\infty}[N]}\lesssim\delta^{-2C_0}$, where the implied constants are absolute. Thus there exists $l=(l_1,l_2,l_3)\in\mathbb{Z}^3\setminus\{(0,0,0)\}$ with $\|l\|_{\infty}\lesssim\delta^{-2C_0}$ such that
\[
\|l\cdot\psi(g_{\xi})\|_{C^{\infty}[N]}\lesssim\delta^{-2C_0}\text{, see Definition~2.7 for the smoothness norms $\|\cdot\|_{C^{\infty}[N]}$.}
\] 
But note that
\begin{multline}
l\cdot\psi(g_{\xi}(n))=(l_1,l_2,l_3)\cdot\big(-\xi n,\sqrt{k}n,0-(-\xi n)\sqrt{k} n\big)=-l_1\xi n+l_2\sqrt{k} n+l_3\xi \sqrt{k} n^2
\\
=l_3\xi \sqrt{k} n^2+(l_2\sqrt{k}-l_1\xi)n=(-l_1\xi+l_2\sqrt{k}+l_3\xi \sqrt{k})\binom{n}{1}+2l_3\xi\sqrt{k}\binom{n}{2}\text{,}
\end{multline}
and thus since $\|l\cdot\psi(g_{\xi})\|_{C^{\infty}[N]}\lesssim \delta^{-2C_0}$ we have 
\begin{equation}\label{condls}
N\|-l_1\xi+l_2\sqrt{k}+l_3\xi \sqrt{k}\|\lesssim \delta^{-2C_0}\quad\text{and}\quad N^2\|2l_3\xi\sqrt{k}\|\lesssim\delta^{-2C_0}\text{.}
\end{equation}
We choose $C_1\coloneqq 2C_0$,  break the analysis into several cases, and show that in every case, the conclusion of Lemma~$\ref{minLemma1'}$ holds.

If $l_3\neq 0$, then the second condition implies that there exists
$m\in\mathbb{Z}$ such that
\begin{equation*}
|2l_3\xi\sqrt{k}-m|\lesssim N^{-2}\delta^{-C_1}\Rightarrow \bigg|\xi-\frac{m}{2l_3\sqrt{k}}\bigg|\lesssim \frac{N^{-2}\delta^{-C_1}}{2|l_3|\sqrt{k}}\lesssim N^{-2}\delta^{-C_1} \Rightarrow \bigg|\xi-\frac{m\sqrt{k}}{2kl_3}\bigg|\lesssim N^{-2}\delta^{-C_1}\text{.}
\end{equation*}
Note that since $|l_3|\lesssim \delta^{-C_1}$ and $\xi\in[-1/2,1/2)$ we must have that 
\[
|m|\le |m-2l_3\xi\sqrt{k}|+|2l_3\xi\sqrt{k}|\lesssim N^{-2}\delta^{-C_1}+\sqrt{k}\delta^{-C_1}\lesssim_k\delta^{-C_1}\text{.}
\]
Thus, there exists a number of the form $\frac{a+b\sqrt{k}}{q}=\frac{m\sqrt{k}}{2kl_3}$ such that 
\[
\Big|\xi-\frac{a+b\sqrt{k}}{q}\Big|\lesssim \delta^{-C_1}N^{-2}\le\delta^{-C_1}N^{-1}\text{,}
\]
with $|b|\le |m|\lesssim_k\delta^{-C_1}$ and $|q|\lesssim_k|l_3|\lesssim\delta^{-C_1}$, and the conclusion of the theorem holds.

If $l_3=0$, we begin by showing that if we also have $l_1= 0$, then $N\lesssim_k \delta^{-2C_1}$ . To see this, note that in such a case we have that  $l_2\neq 0$ since $l=(l_1,l_2,l_3)\neq (0,0,0)$, and thus the first condition in \eqref{condls} yields that there exists $m\in\mathbb{Z}$ such that
\[
|l_2\sqrt{k}-m|\lesssim N^{-1}\delta^{-C_1}\Rightarrow \Big|\sqrt{k}-\frac{m}{l_2}\Big|\lesssim \frac{N^{-1}\delta^{-C_1}}{|l_2|}\text{.}
\]
But note that 
\begin{equation}\label{algnumber}
|l_2|^{-2}\lesssim_k \Big|\sqrt{k}-\frac{m}{l_2}\Big|\text{,}
\end{equation}
where we have used Liouville's theorem \cite{LV} as in \eqref{Liouvillemanouver}. Therefore we get
\[
|l_2|^{-2}\lesssim_k \frac{N^{-1}\delta^{-C_1}}{|l_2|}\Rightarrow N\delta^{C_1}\lesssim_k |l_2|\lesssim \delta^{-C_1} \Rightarrow N\lesssim_k \delta^{-2C_1}\text{,}\quad\text{as desired.}
\]
If $l_3=0$ and $l_1\neq 0$, then the first condition in \eqref{condls} gives that there exists $m\in\mathbb{Z}$ such that
\[
|-l_1\xi+l_2\sqrt{k}-m|\lesssim \delta^{-C_1}N^{-1}\Rightarrow\Big|\xi-\frac{-m+l_2\sqrt{k}}{l_1}\Big|\lesssim \frac{\delta^{-C_1}N^{-1}}{|l_1|}\le\delta^{-C_1}N^{-1}\text{.}
\]
Thus, for this final case, we see that there exists a number of the form $\frac{a+b\sqrt{k}}{q}=\frac{-m+l_2\sqrt{k}}{l_1}$ such that 
\[
\Big|\xi-\frac{a+b\sqrt{k}}{q}\Big|\lesssim\delta^{-C_1}N^{-1}\text{,}
\]
with $|b|\le|l_2|\lesssim\delta^{-C_1}$ and $|q|\le |l_1|\lesssim \delta^{-C_1}$, and thus the conclusion of the theorem holds. The proof is complete.
\end{proof}
We end this subsection with a quick proof of how the previous lemma implies Lemma~$\ref{minLemma1}$.
\begin{proof}[Proof of Lemma~$\ref{minLemma1}$.]Let $k\in\mathbb{Q}_{>0}$ be such that $k\not\in\mathbb{Q}$ and let $C_1,C_2,C_3(k)$ be the constants guaranteed by Lemma~$\ref{minLemma1'}$. Let $\gamma,\gamma'\in(0,1/10)$ and $\chi=\gamma/(C_1+1)$. Let $N\ge 1$ and $\xi\in\mathfrak{m}_1(N;k,\gamma,\gamma')$. We claim that there exists a constant $C=C(\gamma,\gamma',k)$ such that for all $N\ge C$ we have 
\[
|\mathbb{E}_{n\in[N]}e(\xi n\lfloor n\sqrt{k})|\le N^{-\chi}\text{,}
\]
and this clearly would imply the desired result. If we assume for the sake of a contradiction that this is not the case, then we may apply Lemma~$\ref{minLemma1'}$ with $\delta=N^{-\chi}$ for sufficiently large $N$ to obtain that either $N\le C_3 N^{2C_1\chi}$ or there exist $q\in\mathbb{N}$ and $a,b\in\mathbb{Z}$ with $|q|\le C_3N^{C_1\chi} $, $|b|\le C_3N^{C_1\chi}$ and 
\[
\Big|\xi-\frac{a+b\sqrt{k}}{q}\Big|< C_2N^{C_1\chi-1}\text{.}
\]
The condition $N\le C_3 N^{2C_1\chi}$ becomes impossible for $N\gtrsim_{k,\gamma}1$ since $2C_1\chi<2\gamma<1$. Moreover, since $C_1\chi<\gamma$, the second case is also impossible for $N\gtrsim_{k,\gamma,\gamma'}1$ since $\xi\in\mathfrak{m}_1(N;k,\gamma,\gamma')$ and thus such $(a,b,q)$ cannot exist. The proof is complete.
\end{proof}
\subsection{Proof of Lemma~$\ref{minLemma2}$.} This subsection is devoted to the proof of Lemma~$\ref{minLemma2}$. Although we made some effort to organise our arguments, the proof is quite long, partly because we provide a great deal of details, and partly because we avoid introducing the concept of total $\delta$-equidistribution and opt for an approach essentially relying exclusively on Theorem~$\ref{QLT}$. 
\begin{proof}[{Proof of Lemma~$\ref{minLemma2}$}]\label{m2subsection}Let $k\in\mathbb{Q}_{>0}$ be such that $k\not\in\mathbb{Q}$ and $\gamma,\gamma'\in(0,1/10)$, and note that it suffices to establish the estimate for $N\gtrsim_{k,\gamma,\gamma'}1$. Similarly to before, let $G/\Gamma$ be the $3$-dimensional nilmanifold as discussed in subsection~$\ref{HNil}$, and let $F\colon G/\Gamma\to\mathbb{C}$, $g_{\xi}\colon\mathbb{Z}\to G$ be the same as in \eqref{Fdefinition2} and \eqref{definitionofg}, respectively, i.e.:
\begin{equation}\label{defFuseful}
F\Big(\Big(\begin{smallmatrix}
1 & x & z\\
0 & 1 & y\\
0 & 0 & 1
\end{smallmatrix}\Big)\Gamma\Big)=e(z-x\lfloor y\rfloor)\quad\text{and}\quad
g_{\xi}(n)=g_{\xi;k}(n)\coloneqq\bigg(\begin{smallmatrix}
1 & -\xi n & 0\\
0 & 1 & \sqrt{k} n\\
0 & 0 & 1 
\end{smallmatrix}\bigg)\text{.}
\end{equation}
Since $\xi\in\mathfrak{m}_2(N;k,\gamma,\gamma')$, there exist $a\in\mathbb{Z}$, $b\in[\pm N^{\gamma}]$, $q\in[N^{\gamma}]$ and $t\in\mathbb{R}$ with $N^{-2+\gamma'}<|t|\le N^{-1+\gamma}$, such that
\begin{equation}\label{formxi3}
\xi=\frac{a+b\sqrt{k}}{q}+t\text{,}
\end{equation}
and thus we may factorise $g_{\xi}$ as follows
\begin{multline}\label{gxifactorization}
g_{\xi}(n)=\begin{pmatrix}
1 & -\xi n & 0\\
0 & 1 & n\sqrt{k}\\
0 & 0 & 1 
\end{pmatrix}=\begin{pmatrix}
1 & -na/q-nb\sqrt{k}/q-nt & 0\\
0 & 1 & n\sqrt{k} \\
0 & 0 & 1 
\end{pmatrix}
\\
=\begin{pmatrix}
1 & -nt & 0\\
0 & 1 & 0\\
0 & 0 & 1 
\end{pmatrix}
\begin{pmatrix}
1 & -nb\sqrt{k}/q & n^2t\sqrt{k}\\
0 & 1 & n\sqrt{k}\\
0 & 0 & 1 
\end{pmatrix}
\begin{pmatrix}
1 & -na/q & 0\\
0 & 1 & 0\\
0 & 0 & 1 
\end{pmatrix}\eqqcolon \sigma(n)g'_{\xi}(n)\rho(n)\text{,}
\end{multline}
where for the sake of the exposition we choose to suppress the dependence on $a,b,q,t$ for $\sigma,g'_{\xi},\rho$. We wish to partition $[N]$ into subprogressions where $\rho(n)\Gamma$ is constant and $\sigma(n)$ can be treated as a constant. To do this, let $\kappa\coloneqq1-\gamma'/6\in(0,1)$ and $M\coloneqq \lfloor N^{\kappa}\rfloor$, and for every $r\in\{0,\dotsc,q-1\}$ and every $m\in[N/M+1]$ we define 
\[
I_{r,m}\coloneqq \big\{n\in[N]:\,1+(m-1)M\le n\le mM\text{, }n\equiv r\Mod{q}\big\}\text{.}
\]
Note that $[N]=\bigcup_{\substack{r\in\{0,\dotsc,q-1\}\\m\in[N/M+1]}}I_{r,m}$, and thus
\begin{multline}\label{splitsmall1}
\mathbb{E}_{n\in[N]}e(\xi n\lfloor n\sqrt{k} \rfloor)=\mathbb{E}_{n\in[N]}F(g_{\xi}(n)\Gamma)=\mathbb{E}_{n\in[N]}F(\sigma(n)g'_{\xi}(n)\rho(n)\Gamma)
\\
=\frac{1}{\lfloor N\rfloor}\sum_{r=0}^{q-1}\sum_{m=1}^{\lfloor N/M\rfloor +1}\sum_{n\in I_{r,m} }F(\sigma(n)g'_{\xi}(n)\rho(n)\Gamma)=\frac{1}{\lfloor N\rfloor}\sum_{r=0}^{q-1}\sum_{m=1}^{\lfloor N/M\rfloor }\sum_{n\in I_{r,m} }F(\sigma(n)g'_{\xi}(n)\rho(n)\Gamma)+O\big(MN^{-1}\big)\text{.}
\end{multline}\\
\textbf{Step 1: Replacing $F(\sigma(n)g'_{\xi}(n)\rho(n)\Gamma)$ with $F(\sigma(mM)g'_{\xi}(n)\rho(r)\Gamma)$ for $n\in I_{r,m}$.} We claim that for all $n\in I_{r,m}$ we have
\[
\big|F(\sigma(n)g'_{\xi}(n)\rho(n)\Gamma)-F(\sigma(mM)g'_{\xi}(n)\rho(r)\Gamma)\big|\lesssim N^{-1+\kappa+\gamma}\text{.}
\]
To see this note that on the one hand, by \eqref{formxi3} and \eqref{gxifactorization} we get 
\[
F(\sigma(n)g'_{\xi}(n)\rho(n)\Gamma)=F(g_{\xi}(n)\Gamma)=e\bigg(\Big(\frac{na}{q}+\frac{nb\sqrt{k}}{q}+nt\Big)\lfloor n\sqrt{k}\rfloor\bigg)\text{,}
\]
and on the other hand, a simple computation shows that
\[
\sigma(mM)g'_{\xi}(n)\rho(r)=\begin{pmatrix}
1 & -mMt-\frac{nb\sqrt{k}}{q}-\frac{ra}{q} & n^2t\sqrt{k}-mMtn\sqrt{k}\\
0 & 1 & n\sqrt{k}\\
0 & 0 & 1 
\end{pmatrix}
\text{,}
\]
and thus
\[
F(\sigma(mM)g'_{\xi}(n)\rho(r)\Gamma)=e\bigg(n^2t\sqrt{k}-mMtn\sqrt{k}+\Big(mMt+\frac{nb\sqrt{k}}{q}+\frac{ra}{q}\Big)\lfloor n\sqrt{k}\rfloor\bigg)\text{.}
\]
Therefore for every $n\in I_{r,m}$ we get
\begin{multline}
\big|F(\sigma(n)g'_{\xi}(n)\rho(n)\Gamma)-F(\sigma(mM)g'_{\xi}(n)\rho(r)\Gamma)\big|
\\
=\bigg|e\bigg(\Big(\frac{na}{q}+\frac{nb\sqrt{k}}{q}+nt\Big)\lfloor n\sqrt{k}\rfloor\bigg)-e\bigg(n^2t\sqrt{k}-mMtn\sqrt{k}+\Big(mMt+\frac{nb\sqrt{k}}{q}+\frac{ra}{q}\Big)\lfloor n\sqrt{k}\rfloor\bigg)\bigg|
\\
=\bigg|e\bigg(\Big(\frac{ra}{q}+\frac{nb\sqrt{k}}{q}\Big)\lfloor n\sqrt{k}\rfloor\bigg)\bigg|\cdot\Big|e\big(nt\lfloor n\sqrt{k}\rfloor\big)-e\big(n^2t\sqrt{k}-mMtn\sqrt{k}+mMt\lfloor n\sqrt{k}\rfloor\big)\Big|
\\
\lesssim\big|nt\lfloor n\sqrt{k}\rfloor-n^2t\sqrt{k}+mMtn\sqrt{k}-mMt\lfloor n\sqrt{k}\rfloor\big|=|t|\big|n-mM\big|\big|n\sqrt{k}-\lfloor n \sqrt{k}\rfloor\big|\lesssim N^{-1+\kappa+\gamma}\text{.}
\end{multline}
Thus returning to \eqref{splitsmall1} and using the estimate above we get
\begin{multline}\label{concludeafter}
\mathbb{E}_{n\in[N]}e(\xi n\lfloor n\sqrt{k} \rfloor)=\frac{1}{\lfloor N\rfloor}\sum_{r=0}^{q-1}\sum_{m=1}^{\lfloor N/M\rfloor }\sum_{n\in I_{r,m} }F(\sigma(mM)g'_{\xi}(n)\rho(r)\Gamma)
\\
+\frac{1}{\lfloor N\rfloor}\sum_{r=0}^{q-1}\sum_{m=1}^{\lfloor N/M\rfloor}\sum_{n\in I_{r,m} }F(\sigma(n)g'_{\xi}(n)\rho(n)\Gamma)-F(\sigma(mM)g'_{\xi}(n)\rho(r)\Gamma)+O\big(MN^{-1}\big)
\\
=\frac{1}{\lfloor N\rfloor}\sum_{r=0}^{q-1}\sum_{m=1}^{\lfloor N/M\rfloor }\sum_{n\in I_{r,m} }F(\sigma(mM)g'_{\xi}(n)\rho(r)\Gamma)+O(N^{-1+\kappa+\gamma})
\text{.}
\end{multline}\\
\textbf{Step 2: Introducing a lower dimensional nilmanifold.} Now we focus on the study of $\big(g_{\xi}'(n)\big)_{n\in [N]}$, the form of which motivates the following definitions. Let 
\begin{equation}\label{subnil}
G'\coloneqq\Big\{\Big(\begin{smallmatrix}
1 & -bx & y\\
0 & 1 & qx\\
0 & 0 & 1 
\end{smallmatrix}\Big),\quad x,y\in\mathbb{R}\Big\}\text{,}\quad \Gamma'\coloneqq\Big\{\Big(\begin{smallmatrix}
1 & -bn & m\\
0 & 1 & qn\\
0 & 0 & 1 
\end{smallmatrix}\Big),\quad n,m\in\mathbb{Z}\Big\}\text{,}
\end{equation}
and one may check that $G'/\Gamma'$ is a $2$-dimensional $1$-step nilmanifold with fundamental domain
\[
\mathcal{F}'\coloneqq\Big\{\Big(\begin{smallmatrix}
1 & -bx & y\\
0 & 1 & qx\\
0 & 0 & 1 
\end{smallmatrix}\Big),\quad x,y\in[0,1)\Big\}\text{,}
\]
and the Haar measure for the nilmanifold is the Lebesgue measure. Note that $g'_{\xi}(n)\in G'$. We take as our Mal'cev basis $\mathcal{X}'=\{X_1',X_2'\}$, where
\[
X_1'=\begin{pmatrix}
0 & -b & -bq/2\\
0 & 0 & q\\
0 & 0 & 0 
\end{pmatrix}\quad\text{and}\quad X_2'=\begin{pmatrix}
0 & 0 & 1\\
0 & 0 & 0\\
0 & 0 & 0 
\end{pmatrix}\text{,}
\]
which is trivially $1$-rational since $G'$ is abelian, for the definition of rationality of a Mal'cev basis we refer the reader to Definition~2.1 and Definition~2.4 in \cite{BG}. We note that every $g\in G'$ can be written uniquely as $\exp(t_1X_1')\exp(t_2X_2')$, $t_1,t_2\in\mathbb{R}$, where 
\[
\exp\Big(\Big(\begin{smallmatrix}
0 & x & z\\
0 & 0 & y\\
0 & 0 & 0 
\end{smallmatrix}\Big)\Big)\coloneqq \begin{pmatrix}
1 & x & z+\frac{1}{2}xy\\
0 & 1 & y\\
0 & 0 & 1 
\end{pmatrix}\text{.}
\]
To see this note that 
\begin{multline}
\exp(t_1X_1')\exp(t_2X_2')=\exp\Big(\Big(\begin{smallmatrix}
0 & -t_1b & -t_1bq/2\\
0 & 0 & qt_1\\
0 & 0 & 0 
\end{smallmatrix}\Big)\Big)\exp\Big(\Big(\begin{smallmatrix}
0 & 0 & t_2\\
0 & 0 & 0\\
0 & 0 & 0 
\end{smallmatrix}\Big)\Big)
\\
=\begin{pmatrix}
1 & -t_1b & -t_1bq/2-t_1^2bq/2\\
0 & 1 & qt_1\\
0 & 0 & 1 
\end{pmatrix}\begin{pmatrix}
1 & 0 & t_2\\
0 & 1 & 0\\
0 & 0 & 1 
\end{pmatrix}=\begin{pmatrix}
1 & -bt_1 & t_2-t_1bq/2-t_1^2bq/2\\
0 & 1 & qt_1\\
0 & 0 & 1 
\end{pmatrix}
\end{multline}
and thus for $y,z\in\mathbb{R}$ and $x=-by/q$,  the equation $\exp(t_1X_1')\exp(t_2X_2')=\Big(\begin{smallmatrix}
1 & x & z\\
0 & 1 & y\\
0 & 0 & 1 
\end{smallmatrix}\Big)$ has the unique solution
\[
t_1=y/q\quad\text{and}\quad t_2=z+yb/2+by^2/(2q)\text{.}
\]
Moreover, we have that $g'\in\Gamma'$ if and only if the corresponding $t_1,t_2$ are integers, since on the hand, we have that if $x=-bn$, $y=qn$ and $z=m$ for some $(n,m)\in\mathbb{Z}^2$, then
\[
t_1=qn/q=n\in\mathbb{Z}\quad\text{and}\quad t_2=m+qnb/2+b(qn)^2/(2q)=m+qb\frac{n^2+n}{2}=m+qb\frac{n(n+1)}{2}\in\mathbb{Z}
\]
and conversely, if $t_1=n,t_2=m\in\mathbb{Z}$ then
\[
y=qn\text{,}\quad x=-bn\quad\text{and}\quad z=m-nbq/2-n^2bq/2=m-qb\frac{n(n+1)}{2}\in\mathbb{Z}\text{.}
\]
Thus the Mal'cev coordinate
 map $\psi'=\psi'_{\mathcal{X}'}\colon G'\to\mathbb{R}^2$ is such that
\begin{equation}\label{psi'def}
\psi'\Big(\Big(\begin{smallmatrix}
1 & x & z\\
0 & 1 & y\\
0 & 0 & 1 
\end{smallmatrix}\Big)\Big)=(y/q,z+yb/2+by^2/(2q))\text{.}
\end{equation}
Now we note that 
\begin{equation}\label{polyproof}
g'_{\xi}(n)=\begin{pmatrix}
1 & -nb\sqrt{k}/q & n^2t\sqrt{k}\\
0 & 1 & n\sqrt{k}\\
0 & 0 & 1 
\end{pmatrix}=\begin{pmatrix}
1 & -b\sqrt{k}/q & 0\\
0 & 1 & \sqrt{k}\\
0 & 0 & 1 
\end{pmatrix}^n\begin{pmatrix}
1 & 0 & \frac{bk}{2q}\\
0 & 1 & 0\\
0 & 0 & 1 
\end{pmatrix}^{n^2-n}\begin{pmatrix}
1 & 0 & t\sqrt{k}\\
0 & 1 & 0\\
0 & 0 & 1 
\end{pmatrix}^{n^2}\text{,}
\end{equation}
and one may check this using the identity
\[
\Big(\begin{smallmatrix}
1 & L_1 & 0\\
0 & 1 & L_2\\
0 & 0 & 1 
\end{smallmatrix}\Big)^n=\Big(\begin{smallmatrix}
1 & nL_1 & \frac{L_1L_2}{2}(n^2-n)\\
0 & 1 & nL_2\\
0 & 0 & 1 
\end{smallmatrix}\Big)\text{.}
\]
All matrices in \eqref{polyproof} are in $G'$ and thus $g_{\xi}'(n)$ is indeed a polynomial sequence in $G'$, see Definition~1.8 and the comments following it in \cite{BG}. More precisely, one may check that $g'_{\xi}\in \poly(\mathbb{Z},G'_{\bullet})$ for the filtration $G'_{\bullet}$ given below
\[
G'=G'_0=G'_1\supseteq G'_2\coloneqq \Big\{\Big(\begin{smallmatrix}
1 & 0 & x\\
0 & 1 & 0\\
0 & 0 & 1 
\end{smallmatrix}\Big)\,,x\in\mathbb{R}\Big\}\supseteq G'_3\coloneqq\{id_{G'}\}\text{,}
\]
It is also not difficult to check that for every $r'\in\mathbb{Z}$ we have that $g'_{\xi}(qn+r'),\,n\in\mathbb{Z}$ is a polynomial sequence with coefficients in $G'_{\bullet}$. Finally, the Mal'cev basis $\mathcal{X}'$ is adapted to $G'_{\bullet}$ and we are now ready to apply Theorem~$\ref{QLT}$.
\\\,\\
\textbf{Step 3: Equidistribution of $(g'_{\xi}(qn+r')\Gamma')_{n\in[M/q]}$ in $G'/\Gamma'$.} Let $C_0=C_0(2,2)$, be the constant appearing in aforementioned theorem, and without loss of generality, we assume that $C_0>1$. We claim that for every $r'\in\mathbb{Z}$  we have that $(g'_{\xi}(qn+r'))_{n\in[M/q]}$ is $N^{-\gamma'/(2C_0)}$-equidistributed in $G'/\Gamma'$ for $N\gtrsim_{k,\gamma,\gamma'} 1$. If we assume for the sake of a contradiction that $(g'_{\xi}(qn+r'))_{n\in[M/q]}$ is not $N^{-\gamma'/(2C_0)}$-equidistributed in $G'/\Gamma'$, then by Theorem~$\ref{QLT}$ there exists a nontrivial horizontal character $\eta$ with $|\eta|\lesssim \big(N^{-\gamma'/(2C_0)}\big)^{-C_0}=N^{\gamma'/2}$ and such that 
\[
\|\eta\circ g'_{\xi}(q(\cdot)+r')\|_{C^{\infty}[M/q]}\lesssim N^{\gamma'/2}\text{.}
\]
Then there exists  $l=(l_1,l_2)\in\mathbb{Z}^2\setminus\{(0,0)\}$ with $\|l\|_{\ell^{\infty}}\lesssim N^{\gamma'/2}$ such that $\eta(g_{\xi}'(qn+r'))=l\cdot \psi'(g_{\xi}'(qn+r'))$, and note that
\begin{multline}
l\cdot \psi'(g'_{\xi}(qn+r'))=(l_1,l_2)\cdot\psi'(\begin{pmatrix}
1 & -(qn+r')b\sqrt{k}/q & (qn+r')^2t\sqrt{k}\\
0 & 1 & (qn+r')\sqrt{k}\\
0 & 0 & 1 
\end{pmatrix}
)
\\
=l_1(qn+r')\sqrt{k}/q+l_2(qn+r')^2t\sqrt{k}+l_2(qn+r')\sqrt{k}b/2+l_2(qn+r')^2kb/( 2q)
\\
=\binom{n}{2}(2l_2q^2t\sqrt{k}+kbl_2q)+\binom{n}{1}(l_1\sqrt{k}+l_22qr't\sqrt{k}+l_2q\sqrt{k}b/2+l_2r'kb+l_2q^2t\sqrt{k}+kbl_2q/2)+C\text{,}
\end{multline}
where $C$ does not depend on $n$. Since $\|l\cdot \psi'(g'_{\xi}(q(\cdot)+r'))\|_{C^{\infty}[M/q]}\lesssim N^{\gamma'/2}$, we get
\begin{multline}\label{Restrictionsonl's}
\frac{M^2}{q^2}\|2l_2q^2t\sqrt{k}+kbl_2q\|=\frac{M^2}{q^2}\|2l_2q^2t\sqrt{k}\|\lesssim N^{\gamma'/2}\quad\text{and}
\\
\frac{M}{q}\|l_1\sqrt{k}+l_22qr't\sqrt{k}+l_2q\sqrt{k}b/2+l_2r'kb+l_2q^2t\sqrt{k}+kbl_2q/2\|\lesssim N^{\gamma'/2}\text{.}
\end{multline}

We begin by showing that for $N\gtrsim_{k,\gamma,\gamma'}1$ we have that $l_2\neq 0$. If not, then we have that $l_1\neq 0$ and the second estimate in \eqref{Restrictionsonl's} yields that $\|l_1\sqrt{k}\|\lesssim \frac{qN^{\gamma'/2}}{M}$, and thus there exists $m\in\mathbb{Z}$ such that 
\[
|l_1\sqrt{k}-m|\lesssim\frac{qN^{\gamma'/2}}{M}\Rightarrow |l_1|^{-2}\lesssim_k\Big|\sqrt{k}-\frac{m}{l_1}\Big|\lesssim\frac{qN^{\gamma'/2}}{|l_1|M}\Rightarrow N^{\kappa-\gamma}\lesssim \frac{M}{q}\lesssim_k |l_1|N^{\gamma'/2}\lesssim N^{\gamma'}\text{,}
\]
where we have used the fact that $\sqrt{k}$ is an irrational algebraic number of degree $2$, in an identical manner to \eqref{algnumber}. The last estimate leads to a contradiction for $N\gtrsim_{k,\gamma,\gamma'} 1$ because we have that $\kappa-\gamma>\gamma'$, since $\kappa=1-\gamma'/6$ and $\gamma,\gamma'<1/10$.

Thus we have that $l_2\neq 0$ and now we use the first estimate in \eqref{Restrictionsonl's}. Firstly, let us note that for $N\gtrsim_{k,\gamma,\gamma'} 1$ we have that $\|2l_2q^2t\sqrt{k}\|=|2l_2q^2t\sqrt{k}|$, since $|2l_2q^2t\sqrt{k}|\lesssim_k N^{\gamma'/2+2\gamma-1+\gamma}=N^{3\gamma+\gamma'/2-1}$ and we have $\gamma,\gamma'<1/10$. Thus by the first estimate in \eqref{Restrictionsonl's} we get 
\[
|2l_2q^2t\sqrt{k}|\lesssim q^2M^{-2}N^{\gamma'/2}\Rightarrow|t|\lesssim_{k} M^{-2}N^{\gamma'/2}\lesssim N^{-2\kappa+\gamma'/2}\text{,}
\]
which will yield a contradiction for $N\gtrsim_{k,\gamma,\gamma'} 1$ since we will have that
$N^{-2+\gamma'}\le|t|\lesssim_k N^{-2\kappa+\gamma'/2}$ and
\[
-2+\gamma'\le-2\kappa+\gamma'/2\iff \kappa\le 1-\gamma'/4\text{,}
\]
but we have chosen $\kappa=1-\gamma'/6$. 

In either case we reached a contradiction and thus we have shown that for every $r'\in\mathbb{Z}$  we have that $(g'_{\xi}(qn+r'))_{n\in[M/q]}$ is $N^{-\gamma'/(2C_0)}$-equidistributed in $G'/\Gamma'$ for $N\gtrsim_{k,\gamma,\gamma'} 1$, and we will use that to conclude by applying the definition of $\delta$-equidistribution to appropriate functions defined below.
\\\,\\
\textbf{ Step 4: Defining approximants of the functions $x\Gamma'\to F(\sigma(mM)x\rho(r)\Gamma)$.} For every $m\in[N/M]$, $r\in\{0,\dotsc,q-1\}$ and $\tau\in(0,1/100)$, we define the functions 
\[
F_{m,r},F_{\tau,m,r},\chi'_{{\tau,m,r}}\colon G'/\Gamma'\to\mathbb{C}\quad\text{such that}
\]
\[
F_{m,r}(x\Gamma')=F(\sigma(mM)x\rho(r)\Gamma)\text{,}\quad F_{\tau,m,r}(x\Gamma')=F_{\tau}(\sigma(mM)x\rho(r)\Gamma)\text{,}\quad \chi'_{{\tau,m,r}}(x\Gamma')=\chi_{\tau}(\sigma(mM)x\rho(r)\Gamma)
\]
and we note that they are well-defined since for every $x,x\in G'$ we have
\begin{equation*}
x\Gamma'=x'\Gamma'\Rightarrow \exists \gamma\in\Gamma':\,x=x'\gamma\Rightarrow\exists \gamma\in\Gamma':\,\sigma(mM)x\rho(r)=\sigma(mM)x'\rho(r)\big(\rho(r)^{-1}\gamma\rho(r)\big)\text{,}
\end{equation*}
and $\rho(r)^{-1}\gamma\rho(r)\in\Gamma$ since there exist $n,m\in\mathbb{Z}$ such that 
\begin{multline}\label{gammawellbehaved}
\rho(r)^{-1}\gamma\rho(r)=\begin{pmatrix}
1 & ra/q & 0\\
0 & 1 & 0\\
0 & 0 & 1 
\end{pmatrix}
\begin{pmatrix}
1 & -bn & m\\
0 & 1 & qn\\
0 & 0 & 1 
\end{pmatrix}
\begin{pmatrix}
1 & -ra/q & 0\\
0 & 1 & 0\\
0 & 0 & 1 
\end{pmatrix}
\\
=\begin{pmatrix}
1 & ra/q & 0\\
0 & 1 & 0\\
0 & 0 & 1 
\end{pmatrix}\begin{pmatrix}
1 & -bn-ra/q & m\\
0 & 1 & qn\\
0 & 0 & 1 
\end{pmatrix}=\begin{pmatrix}
1 & -bn & m+ran\\
0 & 1 & qn\\
0 & 0 & 1 
\end{pmatrix}\in\Gamma'\subseteq \Gamma\text{.}
\end{multline}
Thus we have
\begin{equation*}
x\Gamma'=x'\Gamma'\Rightarrow \sigma(mM)x\rho(r)\Gamma=\sigma(mM)x'\rho(r)\Gamma\text{,}
\end{equation*}
and thus all functions above are well-defined.

Since for every $r'\in\mathbb{Z}$  we have that $(g'_{\xi}(qn+r'))_{n\in[M/q]}$ is $N^{-\gamma'/(2C_0)}$-equidistributed in $G'/\Gamma'$ for $N\gtrsim_{k,\gamma,\gamma'} 1$, we will have that for all $r\in\{0,\dotsc,q-1\}$ and $m\in[N/M]$ the following estimates hold
\begin{equation}\label{Ftequid}
\big|\mathbb{E}_{n\in I_{r,m}}F_{\tau,m,r}(g'_\xi(n)\Gamma')-\int_{G'/\Gamma'}F_{\tau,m,r}\big|\lesssim N^{-\gamma'/(2C_0)} \|F_{\tau,m,r}\|_{\Lip}
\end{equation}
and
\begin{equation}\label{chitequid}
\big|\mathbb{E}_{n\in I_{r,m}}\chi_{\tau,m,r}(g'_\xi(n)\Gamma')-\int_{G'/\Gamma'}\chi_{\tau,m,r}\big|\lesssim N^{-\gamma'/(2C_0)} \|\chi_{\tau,m,r}\|_{\Lip}\text{.}
\end{equation}
In the next and final step of the proof we will use an argument similar to the one given for Lemma~$\ref{Nilinput}$, so let us estimate the above Lipschitz constants. Note that $\max\big(\|\chi_{\tau,m,r}\|_{\infty},\|F_{\tau,m,r}\|_{\infty}\big)\le 1$ and we will show that there exists a positive constant $C$ such that
\begin{equation}\label{goallipnew}
\max\big(|\chi_{\tau,m,r}(g\Gamma')-\chi_{\tau,m,r}(h\Gamma')|, |F_{\tau,m,r}(g\Gamma')-F_{\tau,m,r}(h\Gamma')|\big)\lesssim \tau^{-1}N^{C\gamma} d_{G'/\Gamma'}(g\Gamma',h\Gamma')
\text{,}
\end{equation}
where the implicit constant is absolute.
Firstly, we have that
\begin{multline}\label{lipargumentnew}
\max\big(|\chi_{\tau,m,r}(g\Gamma')-\chi_{\tau,m,r}(h\Gamma')|, |F_{\tau,m,r}(g\Gamma')-F_{\tau,m,r}(h\Gamma')|\big)
\\
=\max\big(|\chi_{\tau}(\sigma(mM)g\rho(r)\Gamma)-\chi_{\tau}(\sigma(mM)h\rho(r)\Gamma)|,|F_{\tau}(\sigma(mM)g\rho(r)\Gamma)-F_{\tau}(\sigma(mM)h\rho(r)\Gamma)|\big)
\\
\lesssim \tau^{-1}d_{G/\Gamma}(\sigma(mM)g\rho(r)\Gamma,\sigma(mM)h\rho(r)\Gamma)\text{,}
\end{multline}
where we have used that $\max\big(\|\chi_{\tau}\|_{\Lip},\|F_{\tau}\|_{\Lip}\big)\lesssim\tau^{-1}$, see \eqref{liptechnical}, and thus to establish \eqref{goallipnew} it suffices to prove that there exists an absolute positive constant $C$ such that
\begin{equation}\label{goalfinallip}
d_{G/\Gamma}(\sigma(mM)g\rho(r)\Gamma,\sigma(mM)h\rho(r)\Gamma)\lesssim N^{C\gamma}d_{G'/\Gamma'}(g\Gamma',h\Gamma')\text{.}
\end{equation}
To prove this we will use the results from the Appendix~A in \cite{BG} and the first step is to show that
\begin{equation}\label{goalfinallip2}
d_{G/\Gamma}(\sigma(mM)g\rho(r)\Gamma,\sigma(mM)h\rho(r)\Gamma)\lesssim N^{C'\gamma}d_{G/\Gamma}(g\rho(r) \Gamma,h\rho(r)\Gamma)\text{,}
\end{equation}for some absolute constant $C'$. There exists $\gamma_1,\gamma_2\in\Gamma$ such that $g\rho(r)\gamma_1,h\rho(r)\gamma_2\in\mathcal{F}$ and we have
\[
d_{G/\Gamma}(g\rho(r)\Gamma,h\rho(r)\Gamma)=d_{G/\Gamma}\big((g\rho(r)\gamma_1)\Gamma,(h\rho(r)\gamma_2)\Gamma\big)
\] and there exists $\gamma_0\in\Gamma$ such that
\[
1\gtrsim d_{G/\Gamma}(g\rho(r)\Gamma,h\rho(r)\Gamma)\gtrsim d\big(g\rho(r)\gamma_1,(h\rho(r)\gamma_2)\gamma_0\big)\text{,}
\]
by Lemma~A.16 in \cite{BG}. Note that $\big\|\psi(g\rho(r)\gamma_1)\big\|_{\ell^{\infty}},\|\psi(\sigma(mM))\|_{\ell^{\infty}}\lesssim 1+|mMt|\lesssim NN^{-1+\gamma}=N^{\gamma}$ and $\big\|\psi(h\rho(r)\gamma_2)\big\|_{\ell^{\infty}}\lesssim 1$. We wish to apply Lemma~A.5 in \cite{BG} and thus we need to bound $\big\|\psi(h\rho(r)\gamma_2\gamma_0)\big\|_{\ell^{\infty}}$. By Lemma~A.4 in \cite{BG} we have that
\begin{multline}
\big\|\psi(h\rho(r)\gamma_2\gamma_0)\big\|_{\ell^{\infty}}\le \big\|\psi(h\rho(r)\gamma_2\gamma_0)-\psi(g\rho(r)\gamma_1)\big\|_{\ell^{\infty}}+\big\|\psi(g\rho(r)\gamma_1)\big\|_{\ell^{\infty}}
\\
\lesssim d(h\rho(r)\gamma_2\gamma_0,g\rho(r)\gamma_1)+1\lesssim 1\text{,}
\end{multline}
since $d(g\rho(r)\gamma_1,1_G)\le \big\|\psi(g\rho(r)\gamma_1)\big\|_{\ell^{\infty}}\lesssim 1$ and 
\[
d(h\rho(r)\gamma_2\gamma_0,1_G)\le d(h\rho(r)\gamma_2\gamma_0,g\rho(r)\gamma_1)+d(g\rho\gamma_1,1_G)\lesssim 1\text{.}
\] 
Thus we may apply Lemma~A.5 in \cite{BG} and obtain that
\begin{multline}
d_{G/\Gamma}(\sigma(mM)g\rho(r)\Gamma,\sigma(mM)h\rho(r)\Gamma)\le d(\sigma(mM)g\rho(r)\gamma_1,\sigma(mM)h\rho(r)\gamma_2\gamma_0)
\\
\lesssim N^{C'\gamma}d(g\rho(r)\gamma_1,h\rho(r)\gamma_2\gamma_0)\lesssim N^{C'\gamma} d_{G/\Gamma}(g\rho(r)\Gamma,h\rho(r)\Gamma)\text{,}
\end{multline}
and we have established the desired initial estimate \eqref{goalfinallip2}. Now we wish to show that there exists an absolute constant $C''$ such that 
\begin{equation}\label{lipextra}
d_{G/\Gamma}(g\rho(r)\Gamma,h\rho(r)\Gamma)\lesssim N^{C''\gamma}d_{G'/\Gamma'}(g\Gamma',h\Gamma')
\end{equation}
which combined with \eqref{goalfinallip2}  yields the desired estimate \eqref{goalfinallip}. We begin similarly to before and note that using Lemma~A.14 in \cite{BG} (or even directly using \eqref{psi'def}), we get that there exists $\gamma_1',\gamma_2'\in\Gamma'$ such that $\|\psi'(g\gamma_1')\|,\|\psi'(h\gamma_2')\|_{\ell^{\infty}}\lesssim 1$, and also there exists $\gamma'_0\in \Gamma'$ such that
\[
1\gtrsim d_{G'/\Gamma'}(g\Gamma',h\Gamma')\gtrsim d'(g\gamma_1',h\gamma_2'\gamma'_0)
\]
where we have used Lemma~A.16 in \cite{BG} for the first estimate. In an identical manner to before we have
\begin{equation}
\big\|\psi'(h\gamma_2'\gamma'_0)\big\|_{\ell^{\infty}}\le \big\|\psi'(h\gamma_2'\gamma'_0)-\psi'(g\gamma_1')\big\|_{\ell^{\infty}}+\big\|\psi'(g\gamma_1')\big\|_{\ell^{\infty}}\lesssim d'(h\gamma_2'\gamma'_0,g\gamma_1')+1\lesssim 1\text{,}
\end{equation}
since $d'(g\gamma_1',1_{G'})\le \big\|\psi'(g\gamma_1')\big\|_{\ell^{\infty}}\lesssim 1$ and 
\[
d'(h\gamma_2'\gamma'_0,1_{G'})\le d'(h\gamma_2'\gamma'_0,g\gamma_1')+d'(g\gamma_1',1_{G'})\lesssim 1\text{.}
\]
By Lemma~A.6 in \cite{BG} we have
\[
d(g\gamma_1',h\gamma_2'\gamma'_0)\lesssim \big(N^{2\gamma }\big)^{C'''}d'(g\gamma_1',h\gamma_2'\gamma'_0)\text{.}
\]
By the right invariance and the fact that $\rho(r)^{-1}\gamma'\rho(r)\in\Gamma$, whenever $\gamma'\in\Gamma'$, see \eqref{gammawellbehaved}, we have that there exists $\gamma''\in\Gamma$ such that 
\begin{multline}
N^{2C'''\gamma}d_{G'/\Gamma'}(g\Gamma',h\Gamma')\gtrsim d(g\gamma_1',h\gamma_2'\gamma'_0)=d(g,h\gamma_2'\gamma'_0(\gamma_1')^{-1})=
d(g\rho(r),h\gamma_2'\gamma'_0(\gamma_1')^{-1}\rho(r))
\\ 
=d(g\rho(r),h\rho(r)\gamma'')\ge d_{G/\Gamma}(g\rho(r)\Gamma,h\rho(r)\Gamma)\text{,}
\end{multline}
and this establishes \eqref{lipextra} for $C''\coloneq2C'''$. Thus we have established that there exists an absolute positive constant $C$ such that
\begin{equation}\label{Lipfinalconstants}
\|\chi_{\tau,m,r}\|_{\Lip},\|F_{\tau,m,r}\|_{\Lip}\lesssim 1+\tau^{-1}N^{C\gamma}\text{.}
\end{equation}
\\
\textbf{Step 5: Concluding the argument.} We have gathered all the elements to conclude. Returning to \eqref{concludeafter} we have
\begin{multline}\label{firsttoconclude}
\mathbb{E}_{n\in[N]}e(\xi n\lfloor n\sqrt{k} \rfloor)=\frac{1}{\lfloor N\rfloor}\sum_{r=0}^{q-1}\sum_{m=1}^{\lfloor N/M\rfloor }\sum_{n\in I_{r,m} }F(\sigma(mM)g'_{\xi}(n)\rho(r)\Gamma)+O(N^{-1+\kappa+\gamma})
\\
=\frac{1}{\lfloor N\rfloor}\sum_{r=0}^{q-1}\sum_{m=1}^{\lfloor N/M\rfloor}\sum_{n\in I_{r,m} }F_{\tau,m,r}(g'_{\xi}(n)\Gamma')
\\
+\frac{1}{\lfloor N\rfloor}\sum_{r=0}^{q-1}\sum_{m=1}^{\lfloor N/M\rfloor}\sum_{n\in I_{r,m} }F(\sigma(mM)g'_{\xi}(n)\rho(r)\Gamma)-F_{\tau}(\sigma(mM)g'_{\xi}(n)\rho(r)\Gamma)+O(N^{-1+\kappa+\gamma})\text{,}
\end{multline}
but note that by \eqref{chitequid} and \eqref{Lipfinalconstants} we have
\begin{multline}\label{errorimpchitau1}
\bigg|\frac{1}{\lfloor N\rfloor}\sum_{r=0}^{q-1}\sum_{m=1}^{\lfloor N/M\rfloor}\sum_{n\in I_{r,m} }F(\sigma(mM)g'_{\xi}(n)\rho(r)\Gamma)-F_{\tau}(\sigma(mM)g'_{\xi}(n)\rho(r)\Gamma)\bigg|
\\
\lesssim\frac{1}{N}\sum_{r=0}^{q-1}\sum_{m=1}^{\lfloor N/M\rfloor}\sum_{n\in I_{r,m} }\big(1-\chi_{\tau}(\sigma(mM)g'_{\xi}(n)\rho(r)\Gamma)\big)=\frac{1}{N}\sum_{r=0}^{q-1}\sum_{m=1}^{\lfloor N/M\rfloor}\sum_{n\in I_{r,m} }\big(1-\chi_{\tau,m,r}(g'_{\xi}(n)\Gamma')\big)
\\
\le\frac{1}{N}\sum_{r=0}^{q-1}\sum_{m=1}^{\lfloor N/M\rfloor}\Big(|I_{r,m}|\Big|\mathbb{E}_{n\in I_{r,m} }\big(1-\chi_{\tau,m,r}(g'_{\xi}(n)\Gamma')\big)-\int_{G'/\Gamma'}\big(1-\chi_{\tau,m,r}\big)\Big|+|I_{r,m}|\int_{G'/\Gamma'}\big(1-\chi_{\tau,m,r}\big)\Big)
\\
\lesssim N^{-\gamma'/(2C_0)}\|\chi_{\tau,m,r}\|_{\Lip}+\frac{1}{N}\sum_{r=0}^{q-1}\sum_{m=1}^{\lfloor N/M\rfloor}|I_{r,m}|m_{[0,1)^2}\bigg(\bigg\{(x,y)\in[0,1)^2:\,\chi_{\tau,m,r}\bigg(\bigg(\begin{smallmatrix}
1 & -bx & y\\
0 &1 & qx\\
0 & 0 & 1 
\end{smallmatrix}\bigg)\Gamma'\bigg)\neq 1\bigg\}\bigg)
\\
\lesssim N^{-\gamma'/(2C_0)}(1+\tau^{-1}N^{C\gamma})+\tau\text{,}
\end{multline}
where to establish the estimate we used for the measure of the sets in the second summand of the penultimate line one may argue as follows. We note that
\begin{multline}\label{integralevaluationhelp}
\chi_{\tau,m,r}\bigg(\bigg(\begin{smallmatrix}
1 & -bx & y\\
0 & 1 & qx\\
0 & 0 & 1 
\end{smallmatrix}\bigg)\Gamma'\Big)=\chi_{\tau}\bigg(\sigma(mM)\bigg(\begin{smallmatrix}
1 & -bx & y\\
0 & 1 & qx\\
0 & 0 & 1 
\end{smallmatrix}\bigg)\rho(r)\Gamma\bigg)
=\chi_{\tau}\bigg(\bigg(\begin{smallmatrix}
1 & -bx-raq-mMt & y-mMtqx\\
0 & 1 & qx\\
0 & 0 & 1 
\end{smallmatrix}\bigg)
\Gamma\bigg)
\\
=\chi_{\tau}\Bigg(\begin{pmatrix}
1 & \{-bx-raq-mMt\} & \{y-mMtqx+(bx+raq+mMt)\lfloor qx \rfloor\}\\
0 & 1 & \{qx\}\\
0 & 0 & 1 
\end{pmatrix}\Gamma\Bigg)\text{,}
\end{multline}
and thus 
\begin{multline}\label{tauforrm}
m_{[0,1)^2}\bigg(\bigg\{(x,y)\in[0,1)^2:\,\chi_{\tau,m,r}\bigg(\bigg(\begin{smallmatrix}
1 & -bx & y\\
0 &1 & qx\\
0 & 0 & 1 
\end{smallmatrix}\bigg)\Gamma'\bigg)\neq 1\bigg\}\bigg)
\\
\le m_{[0,1)^2}\big(\big\{(x,y)\in[0,1)^2:\,\{-bx-raq-mMt\}\text{ or }\{qx\}\in[0,\tau/5]\cup[1-\tau/5,1]\big\}\big)\lesssim \tau\text{.}
\end{multline}
Using \eqref{errorimpchitau1} in \eqref{firsttoconclude} yields
\begin{multline}\label{firsttoconclude2}
\mathbb{E}_{n\in[N]}e(\xi n\lfloor n\sqrt{k} \rfloor)=\frac{1}{\lfloor N\rfloor}\sum_{r=0}^{q-1}\sum_{m=1}^{\lfloor N/M\rfloor}\sum_{n\in I_{r,m} }F_{\tau,m,r}(g'_{\xi}(n)\Gamma')
\\
+O(N^{-\gamma'/(2C_0)}(1+\tau^{-1}N^{C\gamma})+\tau+N^{-1+\kappa+\gamma})\text{,}
\end{multline}
and for the first summand we note that by \eqref{Ftequid} and \eqref{Lipfinalconstants} we get
\begin{multline}
\frac{1}{\lfloor N\rfloor}\bigg|\sum_{r=0}^{q-1}\sum_{m=1}^{\lfloor N/M\rfloor}\sum_{n\in I_{r,m} }F_{\tau,m,r}(g'_{\xi}(n)\Gamma')\bigg|
\\
\lesssim\frac{1}{N}\sum_{r=0}^{q-1}\sum_{m=1}^{\lfloor N/M\rfloor}|I_{r,m}|\big|\mathbb{E}_{n\in I_{r,m}}F_{\tau,m,r}(g'_\xi(n)\Gamma')-\int_{G'/\Gamma'}F_{\tau,m,r}\big|+\frac{1}{N}\sum_{r=0}^{q-1}\sum_{m=1}^{\lfloor N/M\rfloor}|I_{r,m}|\Big|\int_{G'/\Gamma'}F_{\tau,m,r}\Big|
\\
\lesssim N^{-\gamma'/(2C_0)}(1+\tau^{-1}N^{C\gamma})+
\frac{1}{N}\sum_{r=0}^{q-1}\sum_{m=1}^{\lfloor N/M\rfloor}|I_{r,m}|\Big|\int_{G'/\Gamma'}F_{m,r}\Big|+\frac{1}{N}\sum_{r=0}^{q-1}\sum_{m=1}^{\lfloor N/M\rfloor}|I_{r,m}|\int_{G'/\Gamma'}|F_{\tau,m,r}-F_{m,r}|\text{.}
\end{multline}
The second summand above vanishes since with a calculation identical to \eqref{integralevaluationhelp} we get
\begin{multline}
\int_{G'/\Gamma'}F_{m,r}=\int_{(x,y)\in[0,1)^2}F_{m,r}\bigg(\bigg(\begin{smallmatrix}
1 & -bx & y\\
0 & 1 & qx\\
0 & 0 & 1 
\end{smallmatrix}\bigg)\Gamma'\bigg)dxdy
\\
=\int_{(x,y)\in[0,1)^2}F\Bigg(\begin{pmatrix}
1 & \{-bx-ra{q}-mMt\} & \{y-mMtqx+(bx+ra{q}+mMt)\lfloor qx \rfloor\}\\
0 & 1 & \{qx\}\\
0 & 0 & 1 
\end{pmatrix}\Gamma\Bigg)
\\
=\int_{(x,y)\in[0,1)^2}e(y-mMtqx+(bx+ra{q}+mMt)\lfloor qx \rfloor\})dxdy=0\text{.}
\end{multline}
For the third summand we note that by \eqref{tauforrm} we get
\begin{multline}
\frac{1}{N}\sum_{r=0}^{q-1}\sum_{m=1}^{\lfloor N/M\rfloor }|I_{r,m}|\int_{G'/\Gamma'}|F_{\tau,m,r}-F_{m,r}|\le\frac{1}{N}\sum_{r=0}^{q-1}\sum_{m=1}^{\lfloor N/M\rfloor}|I_{r,m}|\int_{G'/\Gamma'}(1-\chi_{\tau,m,r})
\\
\le \frac{1}{N}\sum_{r=0}^{q-1}\sum_{m=1}^{\lfloor N/M\rfloor}|I_{r,m}|m_{[0,1)^2}\bigg(\bigg\{(x,y)\in[0,1)^2:\,\chi_{\tau,m,r}\bigg(\bigg(\begin{smallmatrix}
1 & -bx & y\\
0 &1 & qx\\
0 & 0 & 1 
\end{smallmatrix}\bigg)\Gamma'\bigg)\neq 1\bigg\}\bigg)\le \tau\text{.}
\end{multline}
Finally, by combining everything we get that for all $\xi\in\mathfrak{m}_2(N,\gamma,\gamma',k)$ we have
\[
\Big|\mathbb{E}_{n\in[N]}e(\xi n\lfloor n\sqrt{k} \rfloor)\Big|\lesssim N^{-\gamma'/(2C_0)}(1+\tau^{-1}N^{C\gamma})+\tau+N^{-1+\kappa+\gamma}\text{.}
\]
By letting $\tau=N^{-\gamma}$, and by remembering that $\kappa=1-\gamma'/6$, the bound becomes
\[
N^{-\gamma'/(2C_0)}+N^{(C+1)\gamma-\gamma'/(2C_0)}+N^{-\gamma}+N^{\gamma-\gamma'/6}\text{,}
\]
and it is easy to see that it suffices to choose $\gamma'>\max(4C_0(C+1),12) \gamma$ to have that 
\[
\Big|\mathbb{E}_{n\in[N]}e(\xi n\lfloor n\sqrt{k} \rfloor)\Big|\lesssim N^{-\chi}\text{,}
\]
for some $\chi=\chi(\gamma,\gamma')>0$. We see that we can choose $c_0\coloneqq \max(4C_0(C+1),12)^{-1}$, and the proof is complete.
\end{proof}
\section{The circle method input}\label{circleinputsection}
Here we collect the previous results and state the input from the circle method used in the sequel. Throughout the last two sections we fix
\begin{equation}\label{choicesgamma}
\gamma'=\frac{1}{20}\text{ and }\gamma=\min\Big(\frac{1}{20},\frac{c_0}{40}\Big)\text{, where $c_0$ is the absolute constant guaranteed by Proposition~$\ref{minarcnew}$,}
\end{equation}
Let $k\in\mathbb{Q}_{>0}$ be such that $\sqrt{k}\not\in\mathbb{Q}$ and $\lambda\in(1,2]$. For every $s,t\in\mathbb{N}_0$ define
\[
\mathcal{P}_{s,t}\coloneqq\big\{(a,b,q)\in\mathbb{N}_0\times\mathbb{Z}\times \mathbb{N}:\,q\in(\lambda^{s-1},\lambda^s],\,b\in A(t),\,a\in\{0,\cdots,q-1\},\,\gcd(a,b,q)=1\big\}\text{,}
\]
where $A(t)=[-\lambda^t,-\lambda^{t-1})\cup(\lambda^{t-1},\lambda^t]$ for $t\ge 1$ and $A(0)=[-1,1]$, and for convenience we let 
$l(a,b,q;k)$ be the unique integer such that $\frac{a+b\sqrt{k}}{q}-l(a,b,q;k)\in[-1/2,1/2)$ and
\[
\alpha^{(k)}_{a,b,q}\coloneq\frac{a-l(a,b,q)q+b\sqrt{k}}{q}\in[-1/2,1/2)\text{.}
\]
We fix a real-valued function $\eta=\eta_{\lambda}\in\mathcal{C}^{\infty}(\mathbb{R})$ such that 
\begin{equation}\label{etalambda}
1_{[-1,1]}\le \eta\le 1_{[-\lambda,\lambda]}\text{.}
\end{equation}
Finally, we define the localized periodization of $V_{X;k}$, namely, we define $\widetilde{V}_{X;k}\colon \mathbb{T}\to \mathbb{C}$
\begin{equation}\label{defVtilde}
\widetilde{V}_{X;k}(t)=\sum_{m\in\mathbb{Z}}V_{X;k}(t+m)1_{[-1/2,1/2)}(t+m)\text{.}
\end{equation}
The circle method carried out in the previous two sections, and specifically Proposition~$\ref{MainMajPro}$ and Proposition~$\ref{minarcnew}$, yields the following.
\begin{proposition}\label{cmethod}
Fix $k\in\mathbb{Q}_{>0}$ such that $\sqrt{k}\not\in\mathbb{Q}$ and $\lambda\in(1,2]$. Then there exist positive constants $C=C(k,\lambda)$ and $\chi$ such that for all $f\in\ell^2(\mathbb{Z})$ we have
\begin{multline}\label{nonotationapproximation}
\bigg\|A_{\lambda^n;k}f-T_{\mathbb{Z}}\bigg[\sum_{(a,b,q)\in\bigcup_{0\le s,t\le \gamma n}\mathcal{P}_{s,t}}\G_k(a,b,q)\F\bigg(\frac{b}{2q}\bigg)\widetilde{V}_{\lambda^n;k}(\cdot-\alpha^{(k)}_{a,b,q})\eta\big(\lambda^{(2-\gamma')n}\|\cdot-\alpha^{(k)}_{a,b,q}\|\big)\bigg]f\bigg\|_{\ell^2(\mathbb{Z})}
\\
\le C\lambda^{-\chi n}\|f\|_{\ell^2(\mathbb{Z})}\text{,}
\end{multline}
where $A_{t;k}$ is defined in \eqref{avop} and $\gamma,\gamma'$ in \eqref{choicesgamma}.
\end{proposition}
\begin{proof}
Firstly, we note that in this proof all implicit constant will be allowed to depend additionally on $k,\lambda$ without further mention, and secondly, it is enough to establish the estimate for $n\gtrsim 1$. By Plancherel theorem it suffices to prove that for all $\xi\in[-1/2,1/2)$ we have
\begin{multline}\label{xinonotation}
\bigg|\frac{1}{\lfloor \lambda^n\rfloor}\sum_{n\le\lambda^n}e(\xi n\lfloor n\sqrt{k} \rfloor)-\sum_{(a,b,q)\in\bigcup_{0\le s,t\le \gamma n}\mathcal{P}_{s,t}}\G_k(a,b,q)\F\Big(\frac{b}{2q}\Big)\widetilde{V}_{\lambda^n;k}(\xi-\alpha^{(k)}_{a,b,q})\eta\big(\lambda^{(2-\gamma')n}\|\xi-\alpha^{(k)}_{a,b,q}\|\big)\bigg|\\
\lesssim \lambda^{-\chi n}
\text{.}
\end{multline}
We split $[-1/2,1/2)=\mathfrak{M}(\lambda^n)\cup\mathfrak{m}(\lambda^n)$. And using Lemma~$\ref{seperationgeneral}$ one may easily show that supports of 
\[
\Big(\eta\big(\lambda^{(2-\gamma')n}\|\cdot-\alpha^{(k)}_{a,b,q}\|\big)\Big)_{(a,b,q)\in\bigcup_{0\le s,t\le \gamma n}\mathcal{P}_{s,t}}
\]
are mutually disjoint for $n\gtrsim 1$.

For $\xi\in\mathfrak{M}(\lambda^n)$ there exists a unique $(a,b,q)\in\mathbb{N}_0\times\mathbb{Z}\times \mathbb{N}$ with $q\le\lambda^{\gamma n}$, $|b|\le \lambda^{\gamma n}$, $a\in\{0,\cdots,q-1\}$ and $\gcd(a,b,q)=1$ such that
$\|\xi-\alpha^{(k)}_{a,b,q}\|\le \lambda^{(-2+\gamma')n}$. This  guarantees that at most one summand in the approximant multliplier does not vanish and we have
\begin{multline}
\bigg|\frac{1}{\lfloor \lambda^n\rfloor}\sum_{n\le\lambda^n}e(\xi n\lfloor n\sqrt{k} \rfloor)-\sum_{(a,b,q)\in\bigcup_{0\le s,t\le \gamma n}\mathcal{P}_{s,t}}\G_k(a,b,q)\F\Big(\frac{b}{2q}\Big)\widetilde{V}_{\lambda^n;k}(\xi-\alpha^{(k)}_{a,b,q})\eta\big(\lambda^{(2-\gamma')n}\|\xi-\alpha^{(k)}_{a,b,q}\|\big)\bigg|
\\
=\bigg|\frac{1}{\lfloor \lambda^n\rfloor}\sum_{n\le\lambda^n}e(\xi n\lfloor n\sqrt{k} \rfloor)-\G_k(a,b,q)\F\Big(\frac{b}{2q}\Big)\widetilde{V}_{\lambda^n;k}(\xi-\alpha^{(k)}_{a,b,q})\eta\big(\lambda^{(2-\gamma')n}\|\xi-\alpha^{(k)}_{a,b,q}\|\big)\bigg|\text{,}
\end{multline}
and $\xi=\alpha^{(k)}_{a,b,q}+t+m\in[-1,2/1,2)$, where $m\in\mathbb{Z}$ and $|t|\le \lambda^{(-2+\gamma')n}$, and thus the expression above becomes
\begin{equation}\label{goodformmaj}
\bigg|\frac{1}{\lfloor \lambda^n\rfloor}\sum_{n\le\lambda^n}e((\alpha^{(k)}_{a,b,q}+t) n\lfloor n\sqrt{k} \rfloor)-\G_k(a,b,q)\F\Big(\frac{b}{2q}\Big)V_{\lambda^n;k}(t)\bigg|\text{,}
\end{equation}
where one the one hand we used \eqref{etalambda} and the range of $t$, as well as the fact that for $n\gtrsim 1$ we have
\begin{equation}\label{periodicarg}
\widetilde{V}_{\lambda^n}(m+t)=\widetilde{V}_{\lambda^n}(t)=V_{\lambda^n}(t)\text{.}
\end{equation}
We bound \eqref{goodformmaj} by applying Proposition~$\ref{MainMajPro}$ for some $\kappa\in(0,1)$ as follows \begin{multline}
\bigg|\frac{1}{\lfloor \lambda^n\rfloor}\sum_{n\le\lambda^n}e((\alpha^{(k)}_{a,b,q}+t) n\lfloor n\sqrt{k} \rfloor)-\G_k(a,b,q)\F\Big(\frac{b}{2q}\Big)V_{\lambda^n;k}(t)\bigg|
\\
\lesssim_{\kappa}  \lambda^{n(\kappa-1)}+|t|\lambda^{n(1+\kappa)}+q^2(|b|+1)(\log(\lambda^{n})+1)\lambda^{-n\kappa/2}
\\
\lesssim \lambda^{n(\kappa-1)}+\lambda^{n(1+\kappa-2+\gamma')}+n\lambda^{2\gamma n}\lambda^{\gamma n}\lambda^{-n\kappa/2}\lesssim\lambda^{n(\kappa-1+\gamma')}+n\lambda^{n(3\gamma-\kappa/2)}\text{,}
\end{multline}
and we choose $\kappa=\frac{2}{3}(3\gamma-\gamma'+1)\in(0,1)$, so that 
\[
3\gamma-\kappa/2=\kappa-1+\gamma'=\frac{6\gamma+\gamma'-1}{3}\text{,}
\]  
which is negative since $6\gamma+\gamma'<1$ by \eqref{choicesgamma}. Thus we have shown that there exists $\chi>0$ such that for every $\xi\in\mathfrak{M}(\lambda^n)$ \eqref{xinonotation} holds.

For $\xi\in\mathfrak{m}(\lambda^n)$ we note that one the one hand by Proposition~$\ref{minarcnew}$ we have that there exists $\chi>0$ such that
\[
\bigg|\frac{1}{\lfloor \lambda^n\rfloor}\sum_{n\le\lambda^n}e(\xi n\lfloor n\sqrt{k} \rfloor)\bigg|\lesssim \lambda^{-\chi n}\text{,}
\]
and on the other hand
\begin{equation}\label{minorarcform}
\bigg|\sum_{(a,b,q)\in\bigcup_{0\le s,t\le \gamma n}\mathcal{P}_{s,t}}\G_k(a,b,q)\F\Big(\frac{b}{2q}\Big)\widetilde{V}_{\lambda^n;k}(\xi-\alpha^{(k)}_{a,b,q})\eta\big(\lambda^{(2-\gamma')n}\|\xi-\alpha^{(k)}_{a,b,q}\|\big)\bigg|=0\text{,}
\end{equation}
unless there exists a unique $(a,b,q)\in\mathbb{N}_0\times\mathbb{Z}\times \mathbb{N}$ with $q\le\lambda^{\gamma n}$, $|b|\le \lambda^{\gamma n}$, $a\in\{0,\cdots,q-1\}$ and $\gcd(a,b,q)=1$ such that
$\|\xi-\alpha^{(k)}_{a,b,q}\|\le \lambda^{(-2+\gamma')n+1}$. Since $\xi\in\mathfrak{m}(\lambda^n)$ we get that $\|\xi-\alpha^{(k)}_{a,b,q}\|>\lambda^{(-2+\gamma')n}$, and thus the expression in the left-hand side of \eqref{minorarcform} becomes for $n\gtrsim 1$
\begin{multline}
\big|\G_k(a,b,q)\F\Big(\frac{b}{2q}\Big)\widetilde{V}_{\lambda^n;k}(\xi-\alpha^{(k)}_{a,b,q})\eta\big(\lambda^{(2-\gamma')n}\|\xi-\alpha^{(k)}_{a,b,q}\|\big)\big|
\\
\le |\widetilde{V}_{\lambda^n;k}(\xi-\alpha^{(k)}_{a,b,q})|\lesssim (\lambda^{(-2+\gamma')n}\lambda^{2n})^{-1/2}=\lambda^{-\gamma'n/2}\text{,}
\end{multline}
where for the last estimate one may use an argument as in \eqref{periodicarg} and Lemma~$\ref{frVXbounds}$. The proof is complete.
\end{proof}
\section{Proof of Theorem~$\ref{osc}$: Oscillation Estimates on $L^2$ for lacunary scales}\label{oscsection}
The goal of this section is to establish Theorem~$\ref{osc}$ and through Calder\'on's transference principle it suffices to prove the following.
\begin{proposition}[Oscillation estimates along lacunary scales on $\ell^2(\mathbb{Z})$]\label{goalosc}Let $k\in\mathbb{Q}_{>0}$ with $\sqrt{k}\not\in\mathbb{Q}$, $\lambda\in(1,2]$. Then there exist a positive constant $C=C(k,\lambda)$ such that for all $f\in\ell^2(\mathbb{Z})$ we have
\begin{equation}\label{ell2osc}
\sup_{J\in\mathbb{N}}\sup_{I\in\mathfrak{S}_J(\mathbb{N}_0)}\|O^2_{I,J}(A_{\lambda^n;k}f:n\in\mathbb{N}_0)\|_{\ell^2(\mathbb{Z})}\le C\|f\|_{\ell^2(\mathbb{Z})}\text{,}
\end{equation}
where $A_{\lambda^n;k}$ is defined in\eqref{avop}.
\end{proposition}
Proposition~$\ref{goalosc}$ immediately implies Theorem~$\ref{osc}$, which as explained in the introduction yields Theorem~$\ref{stronger}$ and thus Theorem~$\ref{Main}$.

Apart from Proposition~$\ref{cmethod}$, we will also be requiring the following estimate
\begin{equation}\label{easygfnow}
\bigg|G_k(a,b,q)F\bigg(\frac{b}{2q}\bigg)\bigg|\lesssim_k q^{-1/2}\min\big(1,|b|^{-1/2}q^{1/2}\big)=\min\big(q^{-1/2},|b|^{-1/2}\big)\text{,}
\end{equation}
for all $(a,b,q)\in\mathbb{Z}\times\mathbb{Z}\times\mathbb{N}$ with $\gcd(a,b,q)=1$, which is immediate from Proposition~$\ref{GaussSumstype}$ and Lemma~$\ref{frVXbounds}$.

The proof of Proposition~$\ref{goalosc}$ will be carried out in steps which are performed in the following subsections. Before we proceed we introduce some notation. For every $s,t,n\in\mathbb{N}_0$ and a pair of functions $g\colon \mathbb{Z}\times \mathbb{Z}\times \mathbb{N}$ and $h\colon \mathbb{T}\to\mathbb{C}$, we define the function $\Pi_{s,t,\le n}^{g,h}\colon \mathbb{T}\to\mathbb{C}$ by
\[
\Pi_{s,t,\le n}^{g,h}(\xi)\coloneqq \sum_{(a,b,q)\in\mathcal{P}_{s,t}}g(a,b,q)h(\xi-\alpha^{(k)}_{a,b,q})\eta(\lambda^{(2-\gamma')n}\|\xi-\alpha^{(k)}_{a,b,q}\|)\text{,}
\]
we remind the reader that we have fixed $\gamma,\gamma'$ in \eqref{choicesgamma}. Note that $\Pi_{s,t,\le n}^{g,h}$ also depends on $k, \lambda$ but we choose to suppress the dependence here. 
Choosing $g_k(a,b,q)\coloneqq \G_k(a,b,q)\F\big(\frac{b}{2q}\big)$, equation \eqref{nonotationapproximation} becomes
\[
\Big\|A_{\lambda^n;k}f-T_{\mathbb{Z}}\Big[\sum_{0\le s,t\le \gamma n}\Pi_{s,t,\le n}^{g_k,\widetilde{V}_{\lambda^n;k}}\Big]f\Big\|_{\ell^2(\mathbb{Z})}\le C\lambda^{-\chi n}\|f\|_{\ell^2(\mathbb{Z})}\text{.}
\]
\subsection{Passing to our first approximant}
A standard square function estimate utilizing Proposition~$\ref{cmethod}$ allows us to deduce the estimate \eqref{ell2osc} from the analogous one for the approximant derived from the circle method. More precisely, we prove the following lemma.
\begin{lemma}\label{tidying}Assume $k\in\mathbb{Q}_{>0}$ is such that $\sqrt{k}\not\in\mathbb{Q}$ and $\lambda\in(1,2]$. Then there exist a positive constant $C=C(k,\lambda)$ such that for all $f\in\ell^2(\mathbb{Z})$ we have
\begin{multline}\label{tidyingest}
\sup_{J\in\mathbb{N}}\sup_{I\in\mathfrak{S}_J(\mathbb{N}_0)}\|O^2_{I,J}(A_{\lambda^n;k}f:n\in\mathbb{N}_0)\|_{\ell^2(\mathbb{Z})}
\\
\le C\sup_{J\in\mathbb{N}}\sup_{I\in\mathfrak{S}_J(\mathbb{N}_0)}\Big\|O^2_{I,J}\Big(T_{\mathbb{Z}}\Big[\sum_{0\le s,t\le \gamma n}\Pi_{s,t,\le n}^{g_k,\widetilde{V}_{\lambda^n;k}}\Big]f:n\in\mathbb{N}_0\Big)\Big\|_{\ell^2(\mathbb{Z})}+C\|f\|_{\ell^2(\mathbb{Z})}\text{.}
\end{multline}
\end{lemma}
\begin{proof}
For the sake of exposition, we write
\begin{equation}\label{notationsimpler}
A_{\lambda^n}=A_{\lambda^n;k}\text{,}\quad g=g_k\text{,}\quad \widetilde{V}_{\lambda^n}=\widetilde{V}_{\lambda^n;k}\quad\text{and we allow all implicit constants to depend on $k,\lambda$.} 
\end{equation}
We have that
\begin{multline}
\sup_{J\in\mathbb{N}}\sup_{I\in\mathfrak{S}_J(\mathbb{N}_0)}\|O^2_{I,J}(A_{\lambda^n}f:n\in\mathbb{N}_0)\|_{\ell^2(\mathbb{Z})}
\\
\lesssim\sup_{J\in\mathbb{N}}\sup_{I\in\mathfrak{S}_J(\mathbb{N}_0)}\Big\|O^2_{I,J}\Big(T_{\mathbb{Z}}\Big[\sum_{0\le s,t\le \gamma n}\Pi_{s,t,\le n}^{g,\widetilde{V}_{\lambda^n}}\Big]f:n\in\mathbb{N}_0\Big)\Big\|_{\ell^2(\mathbb{Z})}
\\
+\sup_{J\in\mathbb{N}}\sup_{I\in\mathfrak{S}_J(\mathbb{N}_0)}\Big\|O^2_{I,J}\Big(A_{\lambda^n}f-T_{\mathbb{Z}}\Big[\sum_{0\le s,t\le \gamma n}\Pi_{s,t,\le n}^{g,\widetilde{V}_{\lambda^n}}\Big]f:n\in\mathbb{N}_0\Big)\Big\|_{\ell^2(\mathbb{Z})}\text{.}
\end{multline}
The second summand can be bounded as follows
\begin{multline}
\sup_{J\in\mathbb{N}}\sup_{I\in\mathfrak{S}_J(\mathbb{N}_0)}\Big\|O^2_{I,J}\Big(A_{\lambda^n}f-T_{\mathbb{Z}}\Big[\sum_{0\le s,t\le \gamma n}\Pi_{s,t,\le n}^{g,\widetilde{V}_{\lambda^n}}\Big]f:n\in\mathbb{N}_0\Big)\Big\|_{\ell^2(\mathbb{Z})}
\\
\lesssim \Big\|\Big(\sum_{n\in\mathbb{N}_0}\Big|A_{\lambda^n}f-T_{\mathbb{Z}}\Big[\sum_{0\le s,t\le \gamma n}\Pi_{s,t,\le n}^{g,\widetilde{V}_{\lambda^n}}\Big]f\Big|^2\Big)^{1/2}\Big\|_{\ell^2(\mathbb{Z})}=\bigg(\sum_{n\in\mathbb{N}_0}\Big\|A_{\lambda^n}f-T_{\mathbb{Z}}\Big[\sum_{0\le s,t\le \gamma n}\Pi_{s,t,\le n}^{g,\widetilde{V}_{\lambda^n}}\Big]f\Big\|_{\ell^2(\mathbb{Z})}^2\bigg)^{1/2}
\\
\lesssim \Big(\sum_{n\in\mathbb{N}_0}\lambda^{-2\chi n}\|f\|^2_{\ell^2(\mathbb{Z})}\Big)^{1/2}\lesssim \|f\|_{\ell^2(\mathbb{Z})}\text{,}
\end{multline}
where $\chi>0$ is the constant from Proposition~$\ref{cmethod}$ which we have used for the penultimate estimate.
\end{proof}
\subsection{Fixing the scale of the $(s,t)$-parameters} From the previous proposition it is clear that one may focus on establishing the oscillation estimates for the operator derived by the approximant multiplier. We formulate a proposition below which will immediately imply this estimate. For convenience for every $s,t\in\mathbb{N}_0$ we let
\begin{equation}\label{mstdef}
m_{s,t}\coloneqq \max(s,t)/\gamma\text{.} 
\end{equation}
\begin{proposition}\label{beginningofthedance2}
Assume $k\in\mathbb{Q}_{>0}$ is such that $\sqrt{k}\not\in\mathbb{Q}$ and $\lambda\in(1,2]$. Then there exist a positive constant $C=C(k,\lambda)$ and a function $\varphi_{\lambda}\colon \mathbb{N}_0\times\mathbb{N}_0\to[0,\infty)$ such that 
\begin{equation}\label{phiadmissible}
\sum_{s,t\in\mathbb{N}_0}\varphi_{\lambda}(s,t)<\infty\text{,}
\end{equation}
and such that for all $s,t\in\mathbb{N}_0$ and $f\in\ell^2(\mathbb{Z})$ we have
\begin{equation}\label{goaleststfixed}
\sup_{J\in\mathbb{N}}\sup_{I\in\mathfrak{S}(\mathbb{N}_0)}\Big\|O_{I,J}^2\Big(T_{\mathbb{Z}}\big[\Pi_{s,t,\le n}^{g_k,\widetilde{V}_{\lambda^n;k}}1_{n\ge m_{s,t}}\big]f:\,n\in\mathbb{N}_{0}\Big)\Big\|_{\ell^2(\mathbb{Z})}\le C\varphi_{\lambda}(s,t)\|f\|_{\ell^2(\mathbb{Z})}\text{.}
\end{equation}
\end{proposition}
The following subsections are devoted to proving Proposition~$\ref{beginningofthedance2}$ and here we simply show that it immediately yields Proposition~$\ref{goalosc}$.
\begin{proof}[Proof of Proposition~$\ref{goalosc}$ using Proposition~$\ref{beginningofthedance2}$]
We again consider $k,\lambda$ fixed and we proceed suppressing the dependence in those parameters as in \eqref{notationsimpler}. Assuming that Proposition~$\ref{beginningofthedance2}$ holds, for every $J\in\mathbb{N}$ and $I=\{I_0<\dotsc<I_{J}\}\in\mathfrak{S}_{J}(\mathbb{N}_0)$ we have
\begin{multline}\label{longestimate}
\Big\|O^2_{I,J}\Big(T_{\mathbb{Z}}\Big[\sum_{0\le s,t\le \gamma n}\Pi_{s,t,\le n}^{g,\widetilde{V}_{\lambda^n}}\Big]f:n\in\mathbb{N}_0\Big)\Big\|_{\ell^2(\mathbb{Z})}
\\
=\bigg\|\bigg(\sum_{j=0}^{J-1}\sup_{I_{j}\le n<I_{j+1}}\Big|T_{\mathbb{Z}}\Big[\sum_{s,t\in\mathbb{N}_0}1_{\gamma n\ge \max(s,t)}\Pi_{s,t,\le n}^{g,\widetilde{V}_{\lambda^n}}\Big]f-T_{\mathbb{Z}}\Big[\sum_{s,t\in\mathbb{N}_0}1_{\gamma I_j\ge \max(s,t)}\Pi_{s,t,\le I_j}^{g,\widetilde{V}_{\lambda^{I_j}}}\Big]f\Big|^2\bigg)^{1/2}\bigg\|_{\ell^2(\mathbb{Z})}
\\
=\bigg\|\bigg(\sum_{j=0}^{J-1}\sup_{I_{j}\le n<I_{j+1}}\Big|\sum_{s,t\in\mathbb{N}_0}T_{\mathbb{Z}}\big[1_{n\ge m_{s,t}}\Pi_{s,t,\le n}^{g,\widetilde{V}_{\lambda^n}}\big]f-T_{\mathbb{Z}}\big[1_{I_j\ge m_{s,t}}\Pi_{s,t,\le I_j}^{g,\widetilde{V}_{\lambda^{I_j}}}\big]f\Big|^2\bigg)^{1/2}\bigg\|_{\ell^2(\mathbb{Z})}
\\
\le\bigg\|\bigg(\sum_{j=0}^{J-1}\Big(\sum_{s,t\in\mathbb{N}_0}\sup_{I_{j}\le n<I_{j+1}}\big|T_{\mathbb{Z}}\big[1_{n\ge m_{s,t}}\Pi_{s,t,\le n}^{g,\widetilde{V}_{\lambda^n}}\big]f-T_{\mathbb{Z}}\big[1_{I_j\ge\max(s,t)/\gamma}\Pi_{s,t,\le I_j}^{g,\widetilde{V}_{\lambda^{I_j}}}\big]f\big|\Big)^2\bigg)^{1/2}\bigg\|_{\ell^2(\mathbb{Z})}
\\
\le\bigg\|\sum_{s,t\in\mathbb{N}_0}\Big(\sum_{j=0}^{J-1}\sup_{I_{j}\le n<I_{j+1}}\big|T_{\mathbb{Z}}\big[1_{n\ge m_{s,t}}\Pi_{s,t,\le n}^{g,\widetilde{V}_{\lambda^n}}\big]f-T_{\mathbb{Z}}\big[1_{I_j\ge m_{s,t}}\Pi_{s,t,\le I_j}^{g,\widetilde{V}_{\lambda^{I_j}}}\big]f\big|^2\Big)^{1/2}\bigg\|_{\ell^2(\mathbb{Z})}
\\
\le\sum_{s,t\in\mathbb{N}_0}\bigg\|\Big(\sum_{j=0}^{J-1}\sup_{I_{j}\le n<I_{j+1}}\big|T_{\mathbb{Z}}\big[1_{n\ge m_{s,t}}\Pi_{s,t,\le n}^{g,\widetilde{V}_{\lambda^n}}\big]f-T_{\mathbb{Z}}\big[1_{I_j\ge m_{s,t}}\Pi_{s,t,\le I_j}^{g,\widetilde{V}_{\lambda^{I_j}}}\big]f\big|^2\Big)^{1/2}\bigg\|_{\ell^2(\mathbb{Z})}
\\
=\sum_{s,t\in\mathbb{N}_0}\Big\|O_{I,J}^2\Big(T_{\mathbb{Z}}\big[\Pi_{s,t,\le n}^{g,\widetilde{V}_{\lambda^n}}1_{n\ge m_{s,t}}\big]f:\,n\in\mathbb{N}_{0}\Big)\Big\|_{\ell^2(\mathbb{Z})}\lesssim\sum_{s,t\in\mathbb{N}_0}\varphi_{\lambda}(s,t)\|f\|_{\ell^2(\mathbb{Z})}\lesssim\|f\|_{\ell^2(\mathbb{Z})}\text{.}
\end{multline}
Taking suprema in $I\in\mathfrak{S}_J(\mathbb{N}_0)$ and $J\in\mathbb{N}$, and taking into account Lemma~$\ref{tidying}$, we immediately obtain the estimate \eqref{ell2osc} and the proof is complete.
\end{proof}
\subsection{The factorization}For the next reduction we will prove and exploit the fact that for every $n\ge m_{s,t}$ we have
\begin{equation}\label{projpure}
\Pi^{g_k,\widetilde{V}_{\lambda^n;k}}_{s,t,\le n}(\xi)=\Pi^{1,\widetilde{V}_{\lambda^n;k}}_{s,t,\le n}(\xi)\Pi^{g_k,1}_{s,t,\le m_{s,t}-1}(\xi)\text{,}
\end{equation}
for all $s,t$ such that $\max(s,t)\gtrsim1$. We formulate the factorization result in the following lemma.
\begin{lemma}\label{factorizationlemma1}
Let $k\in\mathbb{Q}_{>0}$ with $\sqrt{k}\not\in\mathbb{Q}$ and $\lambda\in(1,2]$. Then there exist a positive constant $C=C(k,\lambda)$
such that for all $s,t\in\mathbb{N}_0$ with $m_{s,t}\ge C$, $n\in\mathbb{N}$ with $n\ge m_{s,t}$ and $\xi\in[-1/2,1/2)$ we have
\begin{equation}\label{projpurelemma}
\Pi^{g_k,\widetilde{V}_{\lambda^n;k}}_{s,t,\le n}(\xi)=\Pi^{1,\widetilde{V}_{\lambda^n;k}}_{s,t,\le n}(\xi)\Pi^{g_k,1}_{s,t,\le m_{s,t}-1}(\xi)\text{.}
\end{equation}
\end{lemma}
\begin{proof}
We suppress the dependences similarly to \eqref{notationsimpler}. Note that for every $n\ge m_{s,t}\gtrsim 1$ we have
\begin{multline}\label{onlydiag1}
\Pi^{1,\widetilde{V}_{\lambda^n}}_{s,t,\le n}(\xi)\Pi^{g,1}_{s,t,\le m_{s,t}-1}(\xi)
\\
=\Big(\sum_{(a,b,q)\in\mathcal{P}_{s,t}}\widetilde{V}_{\lambda^n}(\xi-\alpha_{a,b,q})\eta(\lambda^{(2-\gamma')n}\|\xi-\alpha_{a,b,q}\|)\Big)\cdot
\\
\cdot\Big(\sum_{(a',b',q')\in\mathcal{P}_{s,t}}g(a',b',q')\eta(\lambda^{(2-\gamma')(m_{s,t}-1)}\|\xi-\alpha_{a',b',q'}\|)\Big)
\\
=\sum_{(a,b,q)\in\mathcal{P}_{s,t}}g(a,b,q)\widetilde{V}_{\lambda^n}(\xi-\alpha_{a,b,q})\eta(\lambda^{(2-\gamma')n}\|\xi-\alpha_{a,b,q}\|)=\Pi^{g,\widetilde{V}_{\lambda^n}}_{s,t,\le n}(\xi)\text{.}
\end{multline}
To see this, note that on the one hand whenever $(a,b,q)=(a',b',q')\in\mathcal{P}_{s,t}$ we have
\begin{equation}\label{projeq}
\eta(\lambda^{(2-\gamma')n}\|\xi-\alpha_{a,b,q}\|)\eta(\lambda^{(2-\gamma')(m_{s,t}-1)}\|\xi-\alpha_{a',b',q'}\|)=\eta(\lambda^{(2-\gamma')n}\|\xi-\alpha_{a,b,q}\|)
\end{equation}
because if the right hand side equals $0$, then clearly this is true, and if not, then by \eqref{etalambda} we get
\[
\lambda^{(2-\gamma')n}\|\xi-\alpha_{a,b,q}\|<\lambda\Rightarrow\lambda^{(2-\gamma')n-1}\|\xi-\alpha_{a,b,q}\|<1\Rightarrow\lambda^{(2-\gamma')(m_{s,t}-1)}\|\xi-\alpha_{a,b,q}\|<1\text{,}
\]
since $(2-\gamma')n-1\ge(2-\gamma')m_{s,t}-1\ge(2-\gamma')(m_{s,t}-1)$
and thus by \eqref{etalambda} we get
\[
\eta(\lambda^{(2-\gamma')(m_{s,t}-1)}\|\xi-\alpha_{a',b',q'}\|)=1\text{,}
\]
making \eqref{projeq} true. On the other hand, whenever $(a,b,q)\neq (a',b',q')$ are elements of $\mathcal{P}_{s,t}$ we have
\[
\eta(\lambda^{(2-\gamma')n}\|\xi-\alpha_{a,b,q}\|)\eta(\lambda^{(2-\gamma')(m_{s,t}-1)}\|\xi-\alpha_{a',b',q'}\|)=0\text{,}
\]
because for $m_{s,t}\gtrsim 1$ the supports of $\big(\eta(\lambda^{(2-\gamma')(m_{s,t}-1)}\|\cdot-\alpha_{a,b,q}\|)\big)_{(a,b,q)\in\mathcal{P}_{s,t}}$ are disjoint. We give a short proof here for the sake of completeness. If this is not the case, then there exists $\xi\in[-1/2,1/2)$ such that
\[
\|\xi-\alpha_{a,b,q}\|\le\lambda^{(-2+\gamma')(m_{s,t}-1)+1} \quad\text{and}\quad\|\xi-\alpha_{a',b',q'}\|\le\lambda^{(-2+\gamma')(m_{s,t}-1)+1}
\] 
for some $(a,b,q),(a',b',q')\in\mathcal{P}_{s,t}$ with $(a,b,q)\neq (a',b',q')$, and thus by Lemma~$\ref{seperationverygeneral}$ we get
\begin{equation}\label{easycase11}
2\lambda^{(-2+\gamma')(m_{s,t}-1)+1}\ge \|\alpha_{a,b,q}-\alpha_{a',b',q'}\|\gtrsim \lambda^{-3s-t}\text{.}
\end{equation}
Thus  
\[
\lambda^{(2-\gamma')(m_{s,t}-1)-3s-t-1}\lesssim 1 \Rightarrow\lambda^{m_{s,t}-4\gamma m_{s,t}-2}\lesssim 1\Rightarrow\lambda^{(1-4\gamma)m_{s,t}-2}\lesssim 1\text{,}
\]
which does yield a contradiction for $m_{s,t}\gtrsim 1$. The proof is complete.
\end{proof}
\begin{proposition}\label{projpro}
Let $k\in\mathbb{Q}_{>0}$ with $\sqrt{k}\not\in\mathbb{Q}$ and $\lambda\in(1,2]$. Then there exist a positive constant $C=C(k,\lambda)$
such that for all $s,t\in\mathbb{N}_0$ with $m_{s,t}\ge C$ and $f\in\ell^2(\mathbb{Z})$ we have
\begin{equation}\label{goaleststfixed1}
\sup_{J\in\mathbb{N}}\sup_{I\in\mathfrak{S}(\mathbb{N}_0)}\Big\|O_{I,J}^2\Big(T_{\mathbb{Z}}\big[\Pi_{s,t,\le n}^{1,\widetilde{V}_{\lambda^n;k}}1_{n\ge m_{s,t}}\big]f:\,n\in\mathbb{N}_0\Big)\Big\|_{\ell^2(\mathbb{Z})}\le C(s^2+t^2+1)\|f\|_{\ell^2(\mathbb{Z})}\text{.}
\end{equation}
Also, for every $s,t\in\mathbb{N}_0$ there exists a positive constant $C'=C'(s,t,k,\lambda)$ such that
\begin{equation}\label{00case}
\sup_{J\in\mathbb{N}}\sup_{I\in\mathfrak{S}(\mathbb{N}_0)}\Big\|O_{I,J}^2\Big(T_{\mathbb{Z}}\big[\Pi_{s,t,\le n}^{g_k,\widetilde{V}_{\lambda^n;k}}1_{n\ge m_{s,t}}\big]f:\,n\in\mathbb{N}_0\Big)\Big\|_{\ell^2(\mathbb{Z})}\le C'\|f\|_{\ell^2(\mathbb{Z})}\text{.}
\end{equation}
\end{proposition}
Again, establishing these estimates will be achieved in the following subsections and we simply show here that it implies Proposition~$\ref{beginningofthedance2}$, which, in turn, establishes Proposition~$\ref{goalosc}$.
\begin{proof}[Proof of Proposition~$\ref{beginningofthedance2}$ using Proposition~$\ref{projpro}$]
As always, we fix $k,\lambda$ and suppress the dependence. If we let $C=C(k,\lambda)$ be the constant from Lemma~$\ref{factorizationlemma1}$ we get that for all $s,t\in\mathbb{N}_0$ such that $m_{s,t}\ge C$ we have
\[
T_{\mathbb{Z}}\big[\Pi^{g,\widetilde{V}_{\lambda^n}}_{s,t,\le n}1_{n\ge m_{s,t}}\big]f(x)=T_{\mathbb{Z}}[\Pi^{1,\widetilde{V}_{\lambda^n}}_{s,t,\le n}1_{n\ge m_{s,t}}](T_{\mathbb{Z}}[\Pi^{g,1}_{s,t,m_{s,t}-1}]f)(x)\text{,}
\]
and we note that
\[
\|T_{\mathbb{Z}}[\Pi^{g,1}_{s,t,m_{s,t}-1}]f\|_{\ell^2(\mathbb{Z})}\lesssim \|\Pi^{g,1}_{s,t,m_{s,t}-1}\|_{L^{\infty}(\mathbb{T})}\|f\|_{\ell^2(\mathbb{Z})}\text{.}
\]
Also, for $m_{s,t}\gtrsim 1$ we have that for all $\xi\in[-1/2,1/2)$ the following sum
\[
\Pi^{g,1}_{s,t,m_{s,t}-1}(\xi)=\sum_{(a,b,q)\in\mathcal{P}_{s,t}}g(a,b,q)\eta(\lambda^{(2-\gamma')(m_{s,t}-1)}\|\xi-\alpha_{a,b,q}\|)
\]
is such that the supports of its summands are disjoint, see the proof of the previous lemma, and thus using \eqref{easygfnow} we get
\[
|\Pi^{g,1}_{s,t,m_{s,t}-1}(\xi)| \lesssim\min(\lambda^{-s/2},\lambda^{-t/2})\text{,}
\]
yielding the bound
\begin{equation}\label{minlslt}
\|T_{\mathbb{Z}}[\Pi^{g,1}_{s,t,m_{s,t}-1}]f\|_{\ell^2(\mathbb{Z})}\lesssim\min(\lambda^{-s/2},\lambda^{-t/2})\|f\|_{\ell^2(\mathbb{Z})}\text{.}
\end{equation}
Then for all $s,t$ with $m_{s,t}\gtrsim 1$, $J\in\mathbb{N}$ and $I\in\mathfrak{S}_J(\mathbb{N}_0)$ we have
\begin{multline}\label{stbigcalc}
\Big\|O_{I,J}^2\Big(T_{\mathbb{Z}}\big[\Pi_{s,t,\le n}^{g,\widetilde{V}_{\lambda^n}}1_{n\ge m_{s,t}}\big]f:\,n\in\mathbb{N}_0\Big)\Big\|_{\ell^2(\mathbb{Z})}
\\
=\Big\|O_{I,J}^2\Big(T_{\mathbb{Z}}[\Pi^{1,\widetilde{V}_{\lambda^n}}_{s,t,\le n}1_{n\ge m_{s,t}}](T_{\mathbb{Z}}[\Pi^{g,1}_{s,t,m_{s,t}-1}]f):\,n\in\mathbb{N}_0\Big)\Big\|_{\ell^2(\mathbb{Z})}
\\
\lesssim(s^2+t^2+1)\|T_{\mathbb{Z}}[\Pi^{g,1}_{s,t,m_{s,t}-1}]f\|_{\ell^2(\mathbb{Z})}\lesssim(s^2+t^2+1)\min(\lambda^{-s/2},\lambda^{-t/2})\|f\|_{\ell^2(\mathbb{Z})}
\end{multline}
by \eqref{goaleststfixed1} and \eqref{minlslt}. It is easy to check that
\begin{equation}\label{phibounds}
\sum_{s,t\in\mathbb{N}_0}(s^2+t^2+1)\min(\lambda^{-s/2},\lambda^{-t/2})<\infty\text{.}
\end{equation}
Thus to conclude the proof of Proposition~$\ref{beginningofthedance2}$ we may take $\varphi_{\lambda}(s,t)=(s^2+t^2+1)\min(\lambda^{-s/2},\lambda^{-t/2})$ which is an admissible choice since \eqref{phibounds} holds. On the one hand for all $s,t$ with $\max(s,t)\gtrsim 1$, we have shown that \eqref{stbigcalc} holds. On the other hand, one may easily find an appropriate constant for \eqref{goaleststfixed} suitable for the finitely many values of $s,t$ with $m_{s,t}\lesssim 1$, using \eqref{00case} for all such $s,t$.
\end{proof}
\subsection{Freezing the arc tightness} We have shown that it suffices to prove Proposition~$\ref{projpro}$, and, in fact, the only difficult estimate is the first one. The first step in establishing it is a comparison argument relying on the following lemma.
\begin{lemma}\label{freezethelen}
Assume $h\colon \mathbb{Z}\times\mathbb{Z}\times\mathbb{N}\to\mathbb{C}$ is $1$-bounded, $k\in\mathbb{Q}_{>0}$ with $\sqrt{k}\not\in\mathbb{Q}$ and $\lambda\in(1,2]$. Then there exist a positive constant $C=C(k,\lambda)$
such that for all $s,t\in\mathbb{N}_0$ with $m_{s,t}\ge C$, $n\ge m_{s,t}$ and $f\in\ell^2(\mathbb{Z})$ we have
\begin{equation}
\|T_{\mathbb{Z}}[\Pi^{h,\widetilde{V}_{\lambda^n;k}}_{s,t,\le n}]f-T_{\mathbb{Z}}[\Pi^{h,\widetilde{V}_{\lambda^n;k}}_{s,t,\le m_{s,t}}]f\|_{\ell^2(\mathbb{Z})}\le C\lambda^{-\frac{\gamma'n}{2}}\|f\|_{\ell^2(\mathbb{Z})}\text{,}
\end{equation}
and
\begin{multline}\label{freezethearctightness}
\sup_{J\in\mathbb{N}}\sup_{I\in\mathfrak{S}(\mathbb{N}_0)}\Big\|O_{I,J}^2\Big(T_{\mathbb{Z}}\big[\Pi_{s,t,\le n}^{h,\widetilde{V}_{\lambda^n;k}}1_{n\ge m_{s,t}}\big]f:\,n\in\mathbb{N}_0\Big)\Big\|_{\ell^2(\mathbb{Z})}
\\
\le C\sup_{J\in\mathbb{N}}\sup_{I\in\mathfrak{S}(\mathbb{N}_0)}\Big\|O_{I,J}^2\Big(T_{\mathbb{Z}}\big[\Pi_{s,t,\le m_{s,t}}^{h,\widetilde{V}_{\lambda^n;k}}1_{n\ge m_{s,t}}\big]f:\,n\in\mathbb{N}_0\Big)\Big\|_{\ell^2(\mathbb{Z})}+C\|f\|_{\ell^2(\mathbb{Z})}\text{.}
\end{multline}
\end{lemma}
\begin{proof}
We fix $k,\lambda$ and suppress the dependence. Note that
\begin{equation}
\|T_{\mathbb{Z}}[\Pi^{h,\widetilde{V}_{\lambda^n}}_{s,t,\le n}]f-T_{\mathbb{Z}}[\Pi^{h,\widetilde{V}_{\lambda^n}}_{s,t,\le m_{s,t}}]f\|_{\ell^2(\mathbb{Z})}
\le \|\Pi^{h,\widetilde{V}_{\lambda^n}}_{s,t,\le n}-\Pi^{h,\widetilde{V}_{\lambda^n}}_{s,t,\le m_{s,t}}\|_{L^{\infty}(\mathbb{T})}\|f\|_{\ell^2(\mathbb{Z})}\text{,}
\end{equation}
and for $n\ge m_{s,t}\gtrsim 1$ we have that for every $\xi\in[-1/2,1/2)$, if $\big|\Pi^{h,\widetilde{V}_{\lambda^n}}_{s,t,\le n}(\xi)-\Pi^{h,\widetilde{V}_{\lambda^n}}_{s,t,\le m_{s,t}}(\xi)\big|\neq 0$, then there exists $(a,b,q)\in\mathcal{P}_{s,t}$ such that $\lambda^{-(2-\gamma')n}\le \|\xi-\alpha_{a,b,q}\|\le \lambda^{-(2-\gamma')m_{s,t}+1}< 1/2 $, thus there exists $m_0\in\mathbb{Z}$ such that $\lambda^{-(2-\gamma')n}\le |\xi-\alpha_{a,b,q}+m_0|<1/2$, but then for any such $\xi$ we have
\begin{multline}
|\widetilde{V}_{\lambda^n}(\xi-\alpha_{a,b,q})|\le\sum_{m\in\mathbb{Z}}|V_{\lambda^n}(\xi-\alpha_{a,b,q}+m)|1_{[-1/2,1/2)}(\xi-\alpha_{a,b,q}+m)
\\
\le \lambda^{-n}|\xi-\alpha_{a,b,q})+m_0|^{-1/2}\lesssim \lambda^{-n}\lambda^{(1-\gamma'/2)n}\le\lambda^{-(\gamma'/2)n}\text{,}
\end{multline}
by Lemma~$\ref{frVXbounds}$. Since $h$ is $1$-bounded, we get that for all such $\xi$
\[
\big|\Pi^{h,\widetilde{V}_{\lambda^n}}_{s,t,\le n}(\xi)-\Pi^{h,\widetilde{V}_{\lambda^n}}_{s,t,\le m_{s,t}}(\xi)\big|\lesssim_{\lambda} \lambda^{-(\gamma'/2)n}\text{,}
\]
where we have used that the supports of the summands are disjoint for $m_{s,t}\gtrsim1$. For the rest of $\xi$'s there is nothing to prove. Thus we have shown that for all $n\ge m_{s,t}\gtrsim 1$ we have
\[
\|T_{\mathbb{Z}}[\Pi^{h,V_{\lambda^n}}_{s,t,\le n}]f-T_{\mathbb{Z}}[\Pi^{h,V_{\lambda^n}}_{s,t,\le m_{s,t}}]f\|_{\ell^2(\mathbb{Z})}
\lesssim_{\lambda} \lambda^{-(\gamma'/2)n}\|f\|_{\ell^2(\mathbb{Z})}\text{,}
\]
as desired. The second assertion clearly follows by noting that 
\begin{multline*}
\sup_{J\in\mathbb{N}}\sup_{I\in\mathfrak{S}(\mathbb{N}_0)}\Big\|O_{I,J}^2\Big(T_{\mathbb{Z}}\big[\Pi_{s,t,\le n}^{h,\widetilde{V}_{\lambda^n}}1_{n\ge m_{s,t}}\big]f:\,n\in\mathbb{N}_0\Big)\Big\|_{\ell^2(\mathbb{Z})}
\\
\lesssim \sup_{J\in\mathbb{N}}\sup_{I\in\mathfrak{S}(\mathbb{N}_0)}\Big\|O_{I,J}^2\Big(T_{\mathbb{Z}}\big[\Pi_{s,t,\le m_{s,t}}^{h,\widetilde{V}_{\lambda^n}}1_{n\ge m_{s,t}}\big]f:\,n\in\mathbb{N}_0\Big)\Big\|_{\ell^2(\mathbb{Z})}
\\
+\Big\|\Big(\sum_{n\ge m_{s,t}}|T_{\mathbb{Z}}[\Pi^{h,\widetilde{V}_{\lambda^n}}_{s,t,\le n}]f-T_{\mathbb{Z}}[\Pi^{h,\widetilde{V}_{\lambda^n}}_{s,t,\le m_{s,t}}]f|^2\Big)^{1/2}\Big\|_{\ell^2(\mathbb{Z})}\text{,}
\end{multline*}
and bounding the second summand as follows
\begin{multline}
\Big\|\Big(\sum_{n\ge m_{s,t}}|T_{\mathbb{Z}}[\Pi^{h,\widetilde{V}_{\lambda^n}}_{s,t,\le n}]f-T_{\mathbb{Z}}[\Pi^{h,\widetilde{V}_{\lambda^n}}_{s,t,\le m_{s,t}}]f|^2\Big)^{1/2}\Big\|_{\ell^2(\mathbb{Z})}
\\
=
\Big(\sum_{n\ge m_{s,t}}\|T_{\mathbb{Z}}[\Pi^{h,\widetilde{V}_{\lambda^n}}_{s,t,\le n}]f-T_{\mathbb{Z}}[\Pi^{h,\widetilde{V}_{\lambda^n}}_{s,t,\le m_{s,t}}]f\|_{\ell^2(\mathbb{Z})}^2\Big)^{1/2}\lesssim\Big(\sum_{n\ge m_{s,t}}\lambda^{-\gamma' n}\Big)^{1/2}\|f\|_{\ell^2(\mathbb{Z})}
\lesssim \|f\|_{\ell^2(\mathbb{Z})}\text{.}
\end{multline}
\end{proof}
\subsection{Reducing the oscillation to maximal estimates}
Applying Proposition~$\ref{freezethelen}$ with $h=1$, we see that to establish the estimate \eqref{goaleststfixed1} it suffices to show that
\begin{equation}\label{freezenst}
\sup_{J\in\mathbb{N}}\sup_{I\in\mathfrak{S}(\mathbb{N}_0)}\Big\|O_{I,J}^2\Big(T_{\mathbb{Z}}\big[\Pi_{s,t,\le m_{s,t}}^{1,\widetilde{V}_{\lambda^n;k}}1_{n\ge m_{s,t}}\big]f:\,n\in\mathbb{N}_0\Big)\Big\|_{\ell^2(\mathbb{Z})}\lesssim_{k,\lambda}(s^2+t^2+1)\|f\|_{\ell^2(\mathbb{Z})}\text{.}
\end{equation}
We will again exploit the projective nature of our operators to reduce to the maximal estimate. For this reason we fix $\psi^{(\lambda)}\in\mathcal{C}^{\infty}(\mathbb{R})$ such that
\[
1_{[-\frac{1}{4\lambda},\frac{1}{4\lambda}]}\le \psi^{(\lambda)}\le 1_{[-\frac{1}{4},\frac{1}{4}]}\text{,}
\]
and for all $n\in\mathbb{N}_0$ we define
\[
\psi^{(\lambda)}_n(\xi)\coloneqq \psi^{(\lambda)}(\lambda^{2n} \xi)\text{.}
\]
Since $\supp\psi^{(\lambda)}\subseteq [-1/2,1/2)$ and we have identified $\mathbb{T}$ with $[-1/2,1/2)$, we may also think of the function as a function on the torus. Here we unfortunately need to refine the notation $\Pi_{s,t,\le n}^{g,h}(\xi)$ since we would like to extend the property $\psi_n^{(\lambda)}\psi_m^{(\lambda)}=\psi_{\max(n,m)}^{(\lambda)}$ which clearly holds for $n\neq m$ to an analogous one for the multipliers $\Pi$. We introduce a slight refinement of the notation. For every $s,t,n\in\mathbb{N}_0$ and a pair of functions $g\colon \mathbb{Z}\times \mathbb{Z}\times \mathbb{N}$ and $h,\rho\colon \mathbb{T}\to\mathbb{C}$, we define the function $\Pi_{s,t,\le n}^{g,h,\rho}\colon \mathbb{T}\to\mathbb{C}$ by
\[
\Pi_{s,t,\le n}^{g,h,\rho}(\xi)\coloneqq \sum_{(a,b,q)\in\mathcal{P}_{s,t}}g(a,b,q)h(\xi-\alpha^{(k)}_{a,b,q})\rho(\lambda^{(2-\gamma')n}\|\xi-\alpha^{(k)}_{a,b,q}\|)\text{.}
\]
With this notation and keeping in mind that $\big(\eta(\lambda^{(2-\gamma')n}\|\xi-\alpha^{(k)}_{a,b,q}\|)\big)_{(a,b,q)\in\mathcal{P}_{s,t}}$ have mutually disjoint supports for $ m_{s,t}\gtrsim_{k,\lambda}1$, we have that
\begin{equation}\label{projforpsi}
\Pi_{s,t,\le m_{s,t}}^{1,\psi_n^{(\lambda)},\sqrt{\eta}}\Pi_{s,t,\le m_{s,t}}^{1,\psi_m^{(\lambda)},\sqrt{\eta}}=\Pi_{s,t,\le m_{s,t}}^{1,\psi_{\max(n,m)}^{(\lambda)},\eta}\text{.}
\end{equation}
Ultimately, $\eta$ and $\sqrt{\eta}$ have the same useful properties, namely they are smooth, have the same support and take values in $[0,1]$, so this will create no additional complications.  
\begin{proposition}\label{projpassagetomax}
Let $k\in\mathbb{Q}_{>0}$ with $\sqrt{k}\not\in\mathbb{Q}$ and $\lambda\in(1,2]$. Assume that there exists a positive constant $C=C(k,\lambda)$ such that for all $s,t\in\mathbb{N}_0$ with $m_{s,t}\ge C$ we have 
\begin{equation}\label{maxestmst}
\big\|\sup_{n\ge m_{s,t}}|T_{\mathbb{Z}}\big[\Pi_{s,t,\le m_{s,t}}^{1,\psi^{(\lambda)}_n,\sqrt{\eta}}\big]f|\big\|_{\ell^2(\mathbb{Z})}\le C(s^2+t^2+1)\|f\|_{\ell^2(\mathbb{Z})}\text{.}
\end{equation}
Then there exists a positive constant $C'=C'(k,\lambda)$ such that for all $s,t\in\mathbb{N}$ with $m_{s,t}\ge C'$ we have
\begin{equation}\label{osctomax}
\sup_{J\in\mathbb{N}}\sup_{I\in\mathfrak{S}(\mathbb{N}_0)}\Big\|O_{I,J}^2\Big(T_{\mathbb{Z}}\big[\Pi_{s,t,\le m_{s,t}}^{1,\widetilde{V}_{\lambda^n;k}}1_{n\ge m_{s,t}}\big]f:\,n\in\mathbb{N}_0\Big)\Big\|_{\ell^2(\mathbb{Z})}\le C'(s^2+t^2+1)\|f\|_{\ell^2(\mathbb{Z})}\text{.}
\end{equation}
\end{proposition}
\begin{proof}
Suppressing the dependences similarly to \eqref{notationsimpler}, we may estimate as follows
\begin{multline}
\sup_{J\in\mathbb{N}}\sup_{I\in\mathfrak{S}(\mathbb{N}_0)}\Big\|O_{I,J}^2\Big(T_{\mathbb{Z}}\big[\Pi_{s,t,\le m_{s,t}}^{1,\widetilde{V}_{\lambda^n}}1_{n\ge m_{s,t}}\big]f:\,n\in\mathbb{N}_0\Big)\Big\|_{\ell^2(\mathbb{Z})}
\\
\le \sup_{J\in\mathbb{N}}\sup_{I\in\mathfrak{S}(\mathbb{N}_0)}\Big\|O_{I,J}^2\Big(T_{\mathbb{Z}}\big[\Pi_{s,t,\le m_{s,t}}^{1,\psi^{(\lambda)}_n}1_{n\ge m_{s,t}}\big]f:\,n\in\mathbb{N}_0\Big)\Big\|_{\ell^2(\mathbb{Z})}
\\
+\Big\|\Big(\sum_{n\ge m_{s,t}}\big|T_{\mathbb{Z}}\big[\Pi_{s,t,\le m_{s,t}}^{1,\widetilde{V}_{\lambda^n}}-\Pi_{s,t,\le m_{s,t}}^{1,\psi^{(\lambda)}_n}\big]f\big|^2\Big)^{1/2}\Big\|_{\ell^2(\mathbb{Z})}\text{.}
\end{multline}
This reduces the task of establishing \eqref{osctomax} to proving the following two estimates
\begin{equation}\label{goal2proj}
\sup_{J\in\mathbb{N}}\sup_{I\in\mathfrak{S}(\mathbb{N}_0)}\Big\|O_{I,J}^2\Big(T_{\mathbb{Z}}\big[\Pi_{s,t,\le m_{s,t}}^{1,\psi^{(\lambda)}_n}1_{n\ge m_{s,t}}\big]f:\,n\in\mathbb{N}_0\Big)\Big\|_{\ell^2(\mathbb{Z})}\lesssim(s^2+t^2+1) \|f\|_{\ell^2(\mathbb{Z})}
\end{equation}
and
\begin{equation}\label{goal1proj}
\Big\|\Big(\sum_{n\ge m_{s,t}}\big|T_{\mathbb{Z}}\big[\Pi_{s,t,\le m_{s,t}}^{1,\widetilde{V}_{\lambda^n}}-\Pi_{s,t,\le m_{s,t}}^{1,\psi^{(\lambda)}_n}\big]f\big|^2\Big)^{1/2}\Big\|_{\ell^2(\mathbb{Z})}\lesssim\|f\|_{\ell^2(\mathbb{Z})}\text{.}
\end{equation}
For the second one we note that
\begin{multline}\label{sqnice}
\Big\|\Big(\sum_{n\ge m_{s,t}}\big|T_{\mathbb{Z}}\big[\Pi_{s,t,\le m_{s,t}}^{1,\widetilde{V}_{\lambda^n}}-\Pi_{s,t,\le m_{s,t}}^{1,\psi^{(\lambda)}_n}\big]f\big|^2\Big)^{1/2}\Big\|_{\ell^2(\mathbb{Z})}^2
=\sum_{n\ge m_{s,t}}\|T_{\mathbb{Z}}\big[\Pi_{s,t,\le m_{s,t}}^{1,\widetilde{V}_{\lambda^n}}-\Pi_{s,t,\le m_{s,t}}^{1,\psi^{(\lambda)}_n}\big]f\big\|^2_{\ell^2(\mathbb{Z})}
\\
=\sum_{n\ge m_{s,t}}\int_{-1/2}^{1/2}\Big|\Pi_{s,t,\le m_{s,t}}^{1,\widetilde{V}_{\lambda^n}}-\Pi_{s,t,\le m_{s,t}}^{1,\psi^{(\lambda)}_n}\Big|^2|\hat{f}(\xi)|^2d\xi=\sum_{n\ge m_{s,t}}\int_{-1/2}^{1/2}\Big|\Pi_{s,t,\le m_{s,t}}^{1,\widetilde{V}_{\lambda^n}-\psi^{(\lambda)}_n}\Big|^2|\hat{f}(\xi)|^2d\xi\text{.}
\end{multline}
Observe that by the disjointness of the supports for $m_{s,t}\gtrsim 1$ only the diagonal terms will remain when expanding $\Big|\Pi_{s,t,\le m_{s,t}}^{1,\widetilde{V}_{\lambda^n}-\psi^{(\lambda)}_n}\Big|^2$ yielding
\[
\Big|\Pi_{s,t,\le m_{s,t}}^{1,\widetilde{V}_{\lambda^n}-\psi^{(\lambda)}_n}\Big|^2=\sum_{(a,b,q)\in\mathcal{P}_{s,t}}|\widetilde{V}_{\lambda^n}(\xi-\alpha_{a,b,q})-\psi^{(\lambda)}_n(\xi-\alpha_{a,b,q})|^2\eta^2(\lambda^{(2-\gamma')m_{s,t}}\|\xi-\alpha_{a,b,q}\|)\text{,}
\]
and thus
\begin{multline*}
\Big\|\Big(\sum_{n\ge m_{s,t}}\big|T_{\mathbb{Z}}\big[\Pi_{s,t,\le m_{s,t}}^{1,\widetilde{V}_{\lambda^n}}-\Pi_{s,t,\le m_{s,t}}^{1,\psi^{(\lambda)}_n}\big]f\big|^2\Big)^{1/2}\Big\|^2_{\ell^2(\mathbb{Z})}
\\
=\sum_{n\ge m_{s,t}}\int_{-1/2}^{1/2}\Big(\sum_{(a,b,q)\in\mathcal{P}_{s,t}}|\widetilde{V}_{\lambda^n}(\xi-\alpha_{a,b,q})-\psi^{(\lambda)}_n(\xi-\alpha_{a,b,q})|^2\eta^2(\lambda^{(2-\gamma')m_{s,t}}\|\xi-\alpha_{a,b,q}\|)\Big)|\hat{f}(\xi)|^2d\xi
\\
=\int_{-1/2}^{1/2}\bigg(\sum_{(a,b,q)\in\mathcal{P}_{s,t}}\Big(\sum_{n\ge m_{s,t}}|\widetilde{V}_{\lambda^n}(\xi-\alpha_{a,b,q})-\psi^{(\lambda)}_n(\xi-\alpha_{a,b,q})|^2\Big)\eta^2(\lambda^{(2-\gamma')m_{s,t}}\|\xi-\alpha_{a,b,q}\|)\bigg)|\hat{f}(\xi)|^2d\xi\text{.}
\end{multline*}
Assuming momentarily that for all $\xi\in\mathbb{R}$ we have
\begin{equation}\label{verynicefact}
\sum_{n\ge m_{s,t}}|\widetilde{V}_{\lambda^n}(\xi)-\psi^{(\lambda)}_n(\xi)|^2\lesssim 1\text{,}
\end{equation}
we see that one may immediately conclude using the disjointness of the supports by noting that 
\begin{multline}\label{sqnice2}
\Big\|\Big(\sum_{n\ge m_{s,t}}\big|T_{\mathbb{Z}}\big[\Pi_{s,t,\le m_{s,t}}^{1,\widetilde{V}_{\lambda^n}}-\Pi_{s,t,\le m_{s,t}}^{1,\psi^{(\lambda)}_n}\big]f\big|^2\Big)^{1/2}\Big\|^2_{\ell^2(\mathbb{Z})}
\\
\lesssim\int_{-1/2}^{1/2}\Big(\sum_{(a,b,q)\in\mathcal{P}_{s,t}}\eta(\lambda^{(2-\gamma')m_{s,t}}\|\xi-\alpha_{a,b,q}\|)\Big)|\hat{f}(\xi)|^2d\xi\lesssim \int_{-1/2}^{1/2}|\hat{f}(\xi)|^2d\xi=\|f\|^2_{\ell^2(\mathbb{Z})}\text{.}
\end{multline}
Thus \eqref{goal1proj} will be established once we prove \eqref{verynicefact}. To see why this estimate holds, we firstly note that for $\xi\in[-1/2,1/2)$ we have 
\[
|\widetilde{V}_{\lambda^n}(\xi)|\lesssim|\xi\lambda^{2n}|^{-1/2}\quad\text{and}\quad|\widetilde{V}_{\lambda^n}(\xi)-1|\lesssim|\xi\lambda^{2n}|
\]
by \eqref{defVtilde} and \eqref{frVXbounds}. Now note that for all $\xi\in[-1/2,1/2)\setminus\{0\}$ and $n\in\mathbb{N}_0$ such that $|\xi\lambda^{2n}|\ge 1$, we have
\[
|\widetilde{V}_{\lambda^n}(\xi)-\psi^{(\lambda)}_n(\xi)|=|\widetilde{V}_{\lambda^n}(\xi)|\lesssim|\xi\lambda^{2n}|^{-1/2}\text{,}
\]
and for all $\xi\in[-1/2,1/2)\setminus\{0\}$ and $n\in\mathbb{N}_0$ such that $|\xi\lambda^{2n}|< 1$ we have by the Mean Value Theorem
\[
|\widetilde{V}_{\lambda^n}(\xi)-\psi^{(\lambda)}_n(\xi)|\le |\widetilde{V}_{\lambda^n}(\xi)-1|+|\psi(\lambda^{2n}\xi)-1|\lesssim |\xi\lambda^{2n}|+\sup_{t\in\mathbb{R}}\big|\big(\psi^{(\lambda)}\big)'(t)\big||\xi\lambda^{2n}|\lesssim|\xi\lambda^{2n}|\text{.}
\] 
The estimate \eqref{verynicefact} clearly holds for $\xi=0$, and for every $\xi\neq 0$ we have
\[
|\xi\lambda^{2n}|<1\iff\lambda^{2n}<|\xi|^{-1}\iff 2n<\log_{\lambda}(|\xi|^{-1})\iff n<\frac{\log_{\lambda}(|\xi|^{-1})}{2}\text{,}
\] 
and using our previous estimates we may proceed as follows
\begin{multline}\label{finsteparppoxsum}
\sum_{n\ge m_{s,t}}|\widetilde{V}_{\lambda^n}(\xi)-\psi^{(\lambda)}_n(\xi)|^2
\le \sum_{n<\frac{\log_{\lambda}(|\xi|^{-1})}{2}}|\widetilde{V}_{\lambda^n}(\xi)-\psi^{(\lambda)}_n(\xi)|^2+\sum_{n\ge\frac{\log_{\lambda}(|\xi|^{-1})}{2}}|V_{\lambda^n}(\xi)-\psi^{(\lambda)}_n(\xi)|^2
\\
\lesssim_{k,\lambda}\sum_{n<\frac{\log_{\lambda}(|\xi|^{-1})}{2}}|\xi\lambda^{2n}|^2+\sum_{n\ge\frac{\log_{\lambda}(|\xi|^{-1})}{2}}|\xi\lambda^{2n}|^{-1}=|\xi|^2\sum_{n<\frac{\log_{\lambda}(|\xi|^{-1})}{2}}\lambda^{4n}+|\xi|^{-1}\sum_{n\ge\frac{\log_{\lambda}(|\xi|^{-1})}{2}}\lambda^{-2n}
\\
\lesssim|\xi|^2\lambda^{2\log_{\lambda}(|\xi|^{-1})}+|\xi|^{-1}\lambda^{-\log_{\lambda}(|\xi|^{-1})}\lesssim 1\text{,}
\end{multline}
and the proof of \eqref{verynicefact} is complete, and thus the second estimate \eqref{goal1proj} is established.

We now turn our attention to \eqref{goal2proj}. The key observation here is that for $n\neq m$ we have
\begin{equation}\label{projpsinew}
\psi^{(\lambda)}_n(\xi)\psi^{(\lambda)}_m(\xi)=\psi^{(\lambda)}_{\max(n,m)}(\xi)\text{.}
\end{equation}
To see this, assume without loss of generality that $n<m$ and note that if $\xi$ is such that $\psi^{(\lambda)}_m(\xi)=0$ then \eqref{projpsinew} clearly holds. On the other hand, if we assume that $\psi^{(\lambda)}_m(\xi)\neq 0$, then $|\lambda^{2m}\xi|\le 1/4$, and thus 
\[
|\lambda^{2n}\xi|\le |\lambda^{2(m-1)}\xi|\le |\lambda^{2m}\xi|\lambda^{-1} \le \frac{1}{4\lambda}\text{,}
\]
and thus $\psi^{(\lambda)}_n(\xi)=1$ making \eqref{projpsinew} true.  By additionally taking into account the disjointness of the supports in the summands of $\Pi_{s,t,\le m_{s,t}}^{1,\psi^{(\lambda)}_n}$, for $m_{s,t}\gtrsim 1$, \eqref{projpsinew} implies that
\[
\Pi_{s,t,\le m_{s,t}}^{1,\psi^{(\lambda)}_n,\sqrt{\eta}}\Pi_{s,t,\le m_{s,t}}^{1,\psi^{(\lambda)}_m,\sqrt{\eta}}=\Pi_{s,t,\le m_{s,t}}^{1,\psi^{(\lambda)}_{\max(n,m)},\eta}=\Pi_{s,t,\le m_{s,t}}^{1,\psi^{(\lambda)}_{\max(n,m)}}\text{.}
\]
For every $J\in\mathbb{N}$ and $I\in\mathfrak{S}(\mathbb{N}_0)$ we have that
\begin{multline}\label{difficultest}
\Big\|O_{I,J}^2\Big(T_{\mathbb{Z}}\big[\Pi_{s,t,\le m_{s,t}}^{1,\psi^{(\lambda)}_n}1_{n\ge m_{s,t}}\big]f:\,n\in\mathbb{N}_0\Big)\Big\|_{\ell^2(\mathbb{Z})}
\\
=\Big\|\Big(\sum_{j=0}^{J-1}\sup_{I_j< n< I_{j+1}}\big|T_{\mathbb{Z}}\big[\Pi_{s,t,\le m_{s,t}}^{1,\psi^{(\lambda)}_n}1_{n\ge m_{s,t}}\big]f-T_{\mathbb{Z}}\big[\Pi_{s,t,\le m_{s,t}}^{1,\psi^{(\lambda)}_{I_j}}1_{I_j\ge m_{s,t}}\big]f\big|^2\Big)^{1/2}\Big\|_{\ell^2(\mathbb{Z})}
\\
\lesssim\Big\|\Big(\sum_{j=0}^{J-1}\sup_{I_j< n< I_{j+1}}\big|T_{\mathbb{Z}}\big[\Pi_{s,t,\le m_{s,t}}^{1,\psi^{(\lambda)}_n}1_{n\ge m_{s,t}}\big]f-T_{\mathbb{Z}}\big[\Pi_{s,t,\le m_{s,t}}^{1,\psi^{(\lambda)}_{I_{j+1}}}1_{I_{j+1}\ge m_{s,t}}\big]f\big|^2\Big)^{1/2}\Big\|_{\ell^2(\mathbb{Z})}
\\
+\Big\|\Big(\sum_{j=0}^{J-1}\big|T_{\mathbb{Z}}\big[\Pi_{s,t,\le m_{s,t}}^{1,\psi^{(\lambda)}_{I_{j+1}}}1_{I_{j+1}\ge m_{s,t}}\big]f-T_{\mathbb{Z}}\big[\Pi_{s,t,\le m_{s,t}}^{1,\psi^{(\lambda)}_{I_j}}1_{I_{j}\ge m_{s,t}}\big]f\big|^2\Big)^{1/2}\Big\|_{\ell^2(\mathbb{Z})}
\\
\lesssim\Big\|\Big(\sum_{j=0}^{J-1}\sup_{I_j< n< I_{j+1}}\big|T_{\mathbb{Z}}\big[\Pi_{s,t,\le m_{s,t}}^{1,\psi^{(\lambda)}_n}1_{n\ge m_{s,t}}\big]f-T_{\mathbb{Z}}\big[\Pi_{s,t,\le m_{s,t}}^{1,\psi^{(\lambda)}_{I_{j+1}}}1_{n\ge m_{s,t}}\big]f\big|^2\Big)^{1/2}\Big\|_{\ell^2(\mathbb{Z})}
\\
+\Big\|\Big(\sum_{j=0}^{J-1}\sup_{I_j< n< I_{j+1}}\big|T_{\mathbb{Z}}\big[\Pi_{s,t,\le m_{s,t}}^{1,\psi^{(\lambda)}_{I_{j+1}}}1_{n\ge m_{s,t}}\big]f-T_{\mathbb{Z}}\big[\Pi_{s,t,\le m_{s,t}}^{1,\psi^{(\lambda)}_{I_{j+1}}}1_{I_{j+1}\ge m_{s,t}}\big]f\big|^2\Big)^{1/2}\Big\|_{\ell^2(\mathbb{Z})}
\\
+\Big\|\Big(\sum_{j=0}^{J-1}\big|T_{\mathbb{Z}}\big[\Pi_{s,t,\le m_{s,t}}^{1,\psi^{(\lambda)}_{I_{j+1}}}\big]f-T_{\mathbb{Z}}\big[\Pi_{s,t,\le m_{s,t}}^{1,\psi^{(\lambda)}_{I_j}}\big]f\big|^2\Big)^{1/2}\Big\|_{\ell^2(\mathbb{Z})}+\sup_{n\in\mathbb{N}_0}\big\|T_{\mathbb{Z}}\big[\Pi_{s,t,\le m_{s,t}}^{1,\psi^{(\lambda)}_{n}}1_{n\ge m_{s,t}}\big]f\big\|_{\ell^2(\mathbb{Z})}
\\
\lesssim\Big\|\Big(\sum_{j=0}^{J-1}\sup_{\substack{I_j< n< I_{j+1}\\n\ge m_{s,t}}}\big|T_{\mathbb{Z}}\big[\Pi_{s,t,\le m_{s,t}}^{1,\psi^{(\lambda)}_n,\sqrt{\eta}}\Pi_{s,t,\le m_{s,t}}^{1,\psi^{(\lambda)}_{I_j},\sqrt{\eta}}\big]f-T_{\mathbb{Z}}\big[\Pi_{s,t,\le m_{s,t}}^{1,\psi^{(\lambda)}_n,\sqrt{\eta}}\Pi_{s,t,\le m_{s,t}}^{1,\psi^{(\lambda)}_{I_{j+1}},\sqrt{\eta}}\big]f\big|^2\Big)^{1/2}\Big\|_{\ell^2(\mathbb{Z})}
\\
+\Big\|\Big(\sum_{j=0}^{J-1}\big|T_{\mathbb{Z}}\big[\Pi_{s,t,\le m_{s,t}}^{1,\psi^{(\lambda)}_{I_{j+1}}}1_{I_j<m_{s,t}\le I_{j+1}}\big]f\big|^2\Big)^{1/2}\Big\|_{\ell^2(\mathbb{Z})}
\\
+\Big\|\Big(\sum_{j=0}^{J-1}\big|T_{\mathbb{Z}}\big[\Pi_{s,t,\le m_{s,t}}^{1,\psi^{(\lambda)}_{I_{j+1}}}\big]f-T_{\mathbb{Z}}\big[\Pi_{s,t,\le m_{s,t}}^{1,\psi^{(\lambda)}_{I_j}}\big]f\big|^2\Big)^{1/2}\Big\|_{\ell^2(\mathbb{Z})}+\sup_{n\in\mathbb{N}_0}\big\|T_{\mathbb{Z}}\big[\Pi_{s,t,\le m_{s,t}}^{1,\psi^{(\lambda)}_{n}}1_{n\ge m_{s,t}}\big]f\big\|_{\ell^2(\mathbb{Z})}\text{.}
\end{multline}
For the second and the fourth summand here we note that
\begin{multline}
\Big\|\Big(\sum_{j=0}^{J-1}\big|T_{\mathbb{Z}}\big[\Pi_{s,t,\le m_{s,t}}^{1,\psi^{(\lambda)}_{I_{j+1}}}1_{I_j<m_{s,t}\le I_{j+1}}\big]f\big|^2\Big)^{1/2}\Big\|_{\ell^2(\mathbb{Z})}\le \sup_{n\in\mathbb{N}_0}\big\|T_{\mathbb{Z}}\big[\Pi_{s,t,\le m_{s,t}}^{1,\psi^{(\lambda)}_{n}}1_{n\ge m_{s,t}}\big]f\big\|_{\ell^2(\mathbb{Z})}
\\
\le \sup_{n\ge m_{s,t}}\big\|\Pi_{s,t,\le m_{s,t}}^{1,\psi^{(\lambda)}_{n}}\big\|_{L^{\infty}(\mathbb{T})}\|f\|_{\ell^2(\mathbb{Z})}\lesssim \|f\|_{\ell^2(\mathbb{Z})}\text{,}\end{multline}
by the disjointness of the supports. For the first summand in the last bound in \eqref{difficultest} we note that
\begin{multline}
\Big\|\Big(\sum_{j=0}^{J-1}\sup_{\substack{I_j< n< I_{j+1}\\n\ge m_{s,t}}}\big|T_{\mathbb{Z}}\big[\Pi_{s,t,\le m_{s,t}}^{1,\psi^{(\lambda)}_n,\sqrt{\eta}}\Pi_{s,t,\le m_{s,t}}^{1,\psi^{(\lambda)}_{I_j},\sqrt{\eta}}\big]f-T_{\mathbb{Z}}\big[\Pi_{s,t,\le m_{s,t}}^{1,\psi^{(\lambda)}_n,\sqrt{\eta}}\Pi_{s,t,\le m_{s,t}}^{1,\psi^{(\lambda)}_{I_{j+1}},\sqrt{\eta}}\big]f\big|^2\Big)^{1/2}\Big\|_{\ell^2(\mathbb{Z})}
\\
\le\Big\|\Big(\sum_{j=0}^{J-1}\sup_{n\ge m_{s,t}}\Big|T_{\mathbb{Z}}\big[\Pi_{s,t,\le m_{s,t}}^{1,\psi^{(\lambda)}_n,\sqrt{\eta}}\big]\Big(T_{\mathbb{Z}}\big[\Pi_{s,t,\le m_{s,t}}^{1,\psi^{(\lambda)}_{I_{j+1}},\sqrt{\eta}}\big]f-T_{\mathbb{Z}}\big[\Pi_{s,t,\le m_{s,t}}^{1,\psi^{(\lambda)}_{I_j},\sqrt{\eta}}\big]f\Big)\Big|^2\Big)^{1/2}\Big\|_{\ell^2(\mathbb{Z})}
\\
=\Big(\sum_{j=0}^{J-1}\Big\|\sup_{n\ge m_{s,t}}\big|T_{\mathbb{Z}}\big[\Pi_{s,t,\le m_{s,t}}^{1,\psi^{(\lambda)}_n,\sqrt{\eta}}\big]\Big(T_{\mathbb{Z}}\big[\Pi_{s,t,\le m_{s,t}}^{1,\psi^{(\lambda)}_{I_{j+1}},\sqrt{\eta}}\big]f-T_{\mathbb{Z}}\big[\Pi_{s,t,\le m_{s,t}}^{1,\psi^{(\lambda)}_{I_j},\sqrt{\eta}}\big]f\Big)\big|\Big\|^2_{\ell^2(\mathbb{Z})}\Big)^{1/2}
\\
\lesssim(s^2+t^2+1)\Big(\sum_{j=0}^{J-1}\Big\|T_{\mathbb{Z}}\big[\Pi_{s,t,\le m_{s,t}}^{1,\psi^{(\lambda)}_{I_{j+1}},\sqrt{\eta}}\big]f-T_{\mathbb{Z}}\big[\Pi_{s,t,\le m_{s,t}}^{1,\psi^{(\lambda)}_{I_j},\sqrt{\eta}}\big]f\Big\|^2_{\ell^2(\mathbb{Z})}\Big)^{1/2}
\\
=(s^2+t^2+1)\Big\|\Big(\sum_{j=0}^{J-1}\Big|T_{\mathbb{Z}}\big[\Pi_{s,t,\le m_{s,t}}^{1,\psi^{(\lambda)}_{I_{j+1}},\sqrt{\eta}}\big]f-T_{\mathbb{Z}}\big[\Pi_{s,t,\le m_{s,t}}^{1,\psi^{(\lambda)}_{I_j},\sqrt{\eta}}\big]f\Big|^2\Big)^{1/2}\Big\|_{\ell^2(\mathbb{Z})}\text{.}
\end{multline}
To conclude for the first and the third term it suffices to prove
\[
\Big\|\Big(\sum_{j=0}^{J-1}\big|T_{\mathbb{Z}}\big[\Pi_{s,t,\le m_{s,t}}^{1,\psi^{(\lambda)}_{I_{j+1}},\rho}\big]f-T_{\mathbb{Z}}\big[\Pi_{s,t,\le m_{s,t}}^{1,\psi^{(\lambda)}_{I_j},\rho}\big]f\big|^2\Big)^{1/2}\Big\|_{\ell^2(\mathbb{Z})}\lesssim\|f\|_{\ell^2(\mathbb{Z})}\text{,}
\]
for $\rho=\eta$ and $\rho=\sqrt{\eta}$. Note that for either choice of $\rho$ we have
\begin{multline}\label{returnhere}
\Big\|\Big(\sum_{j=0}^{J-1}\big|T_{\mathbb{Z}}\big[\Pi_{s,t,\le m_{s,t}}^{1,\psi^{(\lambda)}_{I_{j+1}},\rho}\big]f-T_{\mathbb{Z}}\big[\Pi_{s,t,\le m_{s,t}}^{1,\psi^{(\lambda)}_{I_j},\rho}\big]f\big|^2\Big)^{1/2}\Big\|_{\ell^2(\mathbb{Z})}^2
\\
=\sum_{j=0}^{J-1}\big\|T_{\mathbb{Z}}\big[\Pi_{s,t,\le m_{s,t}}^{1,\psi^{(\lambda)}_{I_{j+1}},\rho}\big]f-T_{\mathbb{Z}}\big[\Pi_{s,t,\le m_{s,t}}^{1,\psi^{(\lambda)}_{I_j},\rho}\big]f\big\|_{\ell^2(\mathbb{Z})}^2
=\sum_{j=0}^{J-1}\big\|\big(\Pi_{s,t,\le m_{s,t}}^{1,\psi^{(\lambda)}_{I_{j+1}},\rho}-\Pi_{s,t,\le m_{s,t}}^{1,\psi^{(\lambda)}_{I_j},\rho}\big)\hat{f}\big\|_{L^2(\mathbb{T})}^2
\\
=\sum_{j=0}^{J-1}\int_{-1/2}^{1/2}\big|\Pi_{s,t,\le m_{s,t}}^{1,\psi^{(\lambda),\rho}_{I_{j+1}}}(\xi)-\Pi_{s,t,\le m_{s,t}}^{1,\psi^{(\lambda)}_{I_j},\rho}(\xi)\big|^2|\hat{f}(\xi)|^2d\xi
\\
=\int_{-1/2}^{1/2}\Big(\sum_{j=0}^{J-1}\big|\Pi_{s,t,\le m_{s,t}}^{1,\psi^{(\lambda)}_{I_{j+1}}-\psi^{(\lambda)}_{I_{j}},\rho}(\xi)\big|^2\Big)|\hat{f}(\xi)|^2d\xi\text{.}
\end{multline}
Thus
\begin{multline}
\Pi_{s,t,\le m_{s,t}}^{1,\psi^{(\lambda)}_{I_{j+1}}-\psi^{(\lambda)}_{I_{j}},\rho}(\xi)=\sum_{(a,b,q)\in\mathcal{P}_{s,t}}\big(\psi^{(\lambda)}_{I_{j+1}}(\xi-\alpha_{a,b,q})-\psi^{(\lambda)}_{I_{j}}(\xi-\alpha_{a,b,q})\big)\rho(\lambda^{(2-\gamma')m_{s,t}}\|\xi-\alpha_{a,b,q}\|)
\\
=\sum_{(a,b,q)\in\mathcal{P}_{s,t}}\Big(\sum_{n=I_j}^{I_{j+1}-1}\psi^{(\lambda)}_{n+1}(\xi-\alpha_{a,b,q})-\psi^{(\lambda)}_{n}(\xi-\alpha_{a,b,q})\Big)\rho(\lambda^{(2-\gamma')m_{s,t}}\|\xi-\alpha_{a,b,q}\|)\text{.}
\end{multline}
Thus, by the disjointness of the supports for $m_{s,t}\gtrsim 1$ we have
\begin{multline}
\sum_{j=0}^{J-1}\big|\Pi_{s,t,\le m_{s,t}}^{1,\psi^{(\lambda)}_{I_{j}}-\psi^{(\lambda)}_{I_{j+1}},\rho}(\xi)\big|^2
\\
=\sum_{j=0}^{J-1}\sum_{(a,b,q)\in\mathcal{P}_{s,t}}\Big|\sum_{n=I_j}^{I_{j+1}-1}\psi^{(\lambda)}_{n+1}(\xi-\alpha_{a,b,q})-\psi^{(\lambda)}_{n}(\xi-\alpha_{a,b,q})\Big|^2\rho^2(\lambda^{(2-\gamma')m_{s,t}}\|\xi-\alpha_{a,b,q}\|)
\\
\le\sum_{(a,b,q)\in\mathcal{P}_{s,t}}\sum_{j=0}^{J-1}\Big(\sum_{n=I_j}^{I_{j+1}-1}\big|\psi^{(\lambda)}_{n+1}(\xi-\alpha_{a,b,q})-\psi^{(\lambda)}_{n}(\xi-\alpha_{a,b,q})\big|\Big)^2\rho^2(\lambda^{(2-\gamma')m_{s,t}}\|\xi-\alpha_{a,b,q}\|)\text{.}
\end{multline}
We will show that for all $\xi\in\mathbb{R}$ we have
\begin{equation}\label{keyesttoend}
\sum_{n\in\mathbb{N}_0}\big|\psi^{(\lambda)}_{n+1}(\xi)-\psi^{(\lambda)}_{n}(\xi)\big|\lesssim 1\text{.}
\end{equation}
Assuming momentarily that \eqref{keyesttoend} holds we may immediately conclude as follows.
\begin{multline}
\sum_{j=0}^{J-1}\big|\Pi_{s,t,\le m_{s,t}}^{1,\psi^{(\lambda)}_{I_{j}}-\psi^{(\lambda)}_{I_{j+1}},\rho}(\xi)\big|^2\lesssim \sum_{(a,b,q)\in\mathcal{P}_{s,t}}\sum_{j=0}^{J-1}\Big(\sum_{n=I_j}^{I_{j+1}-1}\big|\psi^{(\lambda)}_{n+1}(\xi-\alpha_{a,b,q})-\psi^{(\lambda)}_{n}(\xi-\alpha_{a,b,q})\big|\Big)\cdot
\\
\cdot\Big(\sum_{n=I_j}^{I_{j+1}-1}\big|\psi^{(\lambda)}_{n+1}(\xi-\alpha_{a,b,q})-\psi^{(\lambda)}_{n}(\xi-\alpha_{a,b,q})\big|\Big)\rho^2(\lambda^{(2-\gamma')m_{s,t}}\|\xi-\alpha_{a,b,q}\|)
\\
\lesssim \sum_{(a,b,q)\in\mathcal{P}_{s,t}}\sum_{j=0}^{J-1}\Big(\sum_{n=I_j}^{I_{j+1}-1}\big|\psi^{(\lambda)}_{n+1}(\xi-\alpha_{a,b,q})-\psi^{(\lambda)}_{n}(\xi-\alpha_{a,b,q})\big|\Big)\rho^2(\lambda^{(2-\gamma')m_{s,t}}\|\xi-\alpha_{a,b,q}\|)
\\
\lesssim\sum_{(a,b,q)\in\mathcal{P}_{s,t}}\Big(\sum_{n\in\mathbb{N}_0}\big|\psi^{(\lambda)}_{n+1}(\xi-\alpha_{a,b,q})-\psi^{(\lambda)}_{n}(\xi-\alpha_{a,b,q})\big|\Big)\rho^2(\lambda^{(2-\gamma')m_{s,t}}\|\xi-\alpha_{a,b,q}\|)
\\
\lesssim\sum_{(a,b,q)\in\mathcal{P}_{s,t}}\eta(\lambda^{(2-\gamma')m_{s,t}}\|\xi-\alpha_{a,b,q}\|)\le 1\text{,}
\end{multline}
and thus returning to \eqref{returnhere} we  obtain
\begin{multline}
\Big\|\Big(\sum_{j=0}^{J-1}\big|T_{\mathbb{Z}}\big[\Pi_{s,t,\le m_{s,t}}^{1,\psi^{(\lambda)}_{I_{j+1}},\rho}\big]f-T_{\mathbb{Z}}\big[\Pi_{s,t,\le m_{s,t}}^{1,\psi^{(\lambda)}_{I_j},\rho}\big]f\big|^2\Big)^{1/2}\Big\|_{\ell^2(\mathbb{Z})}^2
\\
=\int_{-1/2}^{1/2}\Big(\sum_{j=0}^{J-1}\big|\Pi_{s,t,\le m_{s,t}}^{1,\psi^{(\lambda)}_{I_{j+1}}-\psi^{(\lambda)}_{I_{j}},\rho}(\xi)\big|^2\Big)|\hat{f}(\xi)|^2d\xi\lesssim\int_{-1/2}^{1/2}|\hat{f}(\xi)|^2d\xi
\le\|f\|^2_{\ell^2(\mathbb{Z})}\text{,}
\end{multline}
as desired. It remains to establish \eqref{keyesttoend}. We note that for all $\xi,n$ such that $|\xi|\lambda^{2(n+1)}\ge 1$ we have 
\[
\big|\psi_{n+1}^{(\lambda)}(\xi)-\psi_n^{(\lambda)}(\xi)\big|=0\text{,}
\]
by taking into account the supports of $\psi_n^{(\lambda)}$. On the other hand, for all for all $\xi,n$ such that $|\xi|\lambda^{2(n+1)}< 1$, we have
\begin{multline*}
\big|\psi_{n+1}^{(\lambda)}(\xi)-\psi_n^{(\lambda)}(\xi)\big|\le \big|\psi_{n+1}^{(\lambda)}(\xi)-1\big|+\big|\psi_{n}^{(\lambda)}(\xi)-1\big|
\\
=\big|\psi^{(\lambda)}(\lambda^{2(n+1)}\xi)-\psi^{(\lambda)}(0)\big|+\big|\psi^{(\lambda)}(\lambda^{2n}\xi)-\psi^{(\lambda)}(0)\big|
\lesssim |\lambda^{2n+2}\xi|\sup_{t\in\mathbb{R}}\big|\big(\psi^{(\lambda)}\big)'(t)\big|\lesssim\lambda^{2n+2}|\xi|\text{.}
\end{multline*}
Note that
\[
|\xi|\lambda^{2(n+1)}< 1\iff\lambda^{2(n+1)}< |\xi|^{-1}\iff 2n+2<\log_{\lambda}(|\xi|^{-1})\iff n<\frac{\log_{\lambda}(|\xi|^{-1})}{2}-1\text{.}
\]
Finally, we get
\begin{multline}
\sum_{n\in\mathbb{N}_0}\big|\psi^{(\lambda)}_{n+1}(\xi)-\psi^{(\lambda)}_{n}(\xi)\big|
=\sum_{0\le n<\frac{\log_{\lambda}(|\xi|^{-1})}{2}-1}\big|\psi^{(\lambda)}_{n+1}(\xi)-\psi^{(\lambda)}_{n}(\xi)\big|+0
\\
\lesssim\sum_{0\le n<\frac{\log_{\lambda}(|\xi|^{-1})}{2}-1}\lambda^{2n+2}|\xi|\lesssim|\xi|\sum_{0\le n<\frac{\log_{\lambda}(|\xi|^{-1})}{2}-1}\lambda^{2n+2}\lesssim1\text{,}
\end{multline}
and the proof of \eqref{osctomax} is complete. 
\end{proof}
We remind the reader that to conclude it suffices to prove Proposition~$\ref{projpro}$, and the only difficult estimate is \eqref{goaleststfixed1}. We briefly discuss the second estimate \eqref{00case} of the aforementioned proposition in the final subsection. By Proposition~$\ref{freezethelen}$ and $\ref{projpassagetomax}$, to establish the estimate \eqref{goaleststfixed1}, it suffices to prove the estimate \eqref{maxestmst}, which we do in the next subsection.
\subsection{Employing Bourgain's logarithmic lemma} We are now almost ready to  address the multi-frequency nature of the problem by passing to $L^2(\mathbb{R})$ and using Bourgain's logarithmic lemma, see Lemma~4.13 in \cite{Bourgain}. Before doing so, we make a final approximation to make Bourgain's lemma directly applicable.

Fix $\phi\in\mathcal{C}^{\infty}(\mathbb{R}\to\mathbb{C})$ such that $0\le \phi\le 1$, $\supp\mathcal{F}_{\mathbb{R}}[\phi]\subseteq [-1/2,1/2)$ and $\int\phi\neq 0$, so we may treat $\mathcal{F}_{\mathbb{R}}[\phi]$ as a multiplier on $\mathbb{T}\equiv[-1/2,1/2)$. For every $n\in\mathbb{N}_0$ we let 
\[
\chi^{(\lambda)}_{n}(\xi)\coloneqq \frac{\mathcal{F}_{\mathbb{R}}[\phi](\lambda^{2n}\xi)}{\mathcal{F}_{\mathbb{R}}[\phi](0)}\text{.}
\]
\begin{proposition}\label{projpassagetomaxnew}
Let $k\in\mathbb{Q}_{>0}$ with $\sqrt{k}\not\in\mathbb{Q}$ and $\lambda\in(1,2]$. Then there exists a positive constant $C=C(k,\lambda)$ such that for all $s,t\in\mathbb{N}_0$ with $m_{s,t}\ge C$ we have
\begin{equation}\label{finapproxosc}
\big\|\sup_{n\ge m_{s,t}}|T_{\mathbb{Z}}\big[\Pi_{s,t,\le m_{s,t}}^{1,\psi^{(\lambda)}_n,\sqrt{\eta}}\big]f|\big\|_{\ell^2(\mathbb{Z})}\le\big\|\sup_{n\ge m_{s,t}}|T_{\mathbb{Z}}\big[\Pi_{s,t,\le m_{s,t}}^{1,\chi^{(\lambda)}_n,\sqrt{\eta}}\big]f|\big\|_{\ell^2(\mathbb{Z})}+C\|f\|_{\ell^2(\mathbb{Z})}\text{.}
\end{equation}
\end{proposition}
\begin{proof}
We note that
\begin{multline}
\big\|\sup_{n\ge m_{s,t}}|T_{\mathbb{Z}}\big[\Pi_{s,t,\le m_{s,t}}^{1,\psi^{(\lambda)}_n,\sqrt{\eta}}\big]f|\big\|_{\ell^2(\mathbb{Z})}
\\
\lesssim\big\|\sup_{n\ge m_{s,t}}|T_{\mathbb{Z}}\big[\Pi_{s,t,\le m_{s,t}}^{1,\chi^{(\lambda)}_n,\sqrt{\eta}}\big]f|\big\|_{\ell^2(\mathbb{Z})}+\big\|\sup_{n\ge m_{s,t}}|T_{\mathbb{Z}}\big[\Pi_{s,t,\le m_{s,t}}^{1,\chi^{(\lambda)}_n,\sqrt{\eta}}\big]f-T_{\mathbb{Z}}\big[\Pi_{s,t,\le m_{s,t}}^{1,\psi^{(\lambda)}_n,\sqrt{\eta}}\big]f|\big\|_{\ell^2(\mathbb{Z})}
\\
\lesssim\big\|\sup_{n\ge m_{s,t}}|T_{\mathbb{Z}}\big[\Pi_{s,t,\le m_{s,t}}^{1,\chi^{(\lambda)}_n,\sqrt{\eta}}\big]f|\big\|_{\ell^2(\mathbb{Z})}+\Big\|\Big(\sum_{n\ge m_{s,t}}|T_{\mathbb{Z}}\big[\Pi_{s,t,\le m_{s,t}}^{1,\chi^{(\lambda)}_n,\sqrt{\eta}}\big]f-T_{\mathbb{Z}}\big[\Pi_{s,t,\le m_{s,t}}^{1,\psi^{(\lambda)}_n,\sqrt{\eta}}\big]f|^2\Big)^{1/2}\Big\|_{\ell^2(\mathbb{Z})}\text{,}
\end{multline}
and one may conclude now exactly as in \eqref{sqnice}, by establishing
\begin{equation}\label{verynicefact22}
\sum_{n\ge m_{s,t}}|\chi_{n}^{(\lambda)}(\xi)-\psi^{(\lambda)}_n(\xi)|^2\lesssim 1\text{,}
\end{equation}
 and arguing as in \eqref{sqnice2}, so we focus on proving the estimate above. For all $\xi,n$ such that $|\xi\lambda^{2n}|\ge 1$, we have
\[
|\chi^{(\lambda)}_n(\xi)-\psi^{(\lambda)}_n(\xi)|=0\text{,}
\]
by taking into account the supports. For all $\xi,n$ such that $|\xi\lambda^{2n}|< 1$ we have by the Mean Value Theorem
\begin{multline}
|\chi^{(\lambda)}_n(\xi)-\psi^{(\lambda)}_n(\xi)|\le |\chi^{(\lambda)}_n(\xi)-1|+|\psi^{(\lambda)}_n(\xi)-1|=|\chi^{(\lambda)}_n(\xi)-\chi^{(\lambda)}_n(0)|+|\psi^{(\lambda)}_n(\xi)-\psi^{(\lambda)}_n(0)|
\\
\lesssim\sup_{t\in\mathbb{R}}\bigg|\bigg( \frac{\mathcal{F}_{\mathbb{R}}[\phi](t)}{\mathcal{F}_{\mathbb{R}}[\phi](0)}\bigg)'(t)\bigg||\xi\lambda^{2n}|+ \sup_{t\in\mathbb{R}}\big|\big(\psi^{(\lambda)}\big)'(t)\big||\xi\lambda^{2n}|\lesssim|\xi\lambda^{2n}|\text{.}
\end{multline}
Thus is a manner identical to \eqref{finsteparppoxsum} one may establish \eqref{verynicefact22} and conclude.
\end{proof}

\begin{proposition}\label{lastestimate}
Assume $k\in\mathbb{Q}_{>0}$ with $\sqrt{k}\not\in\mathbb{Q}$ and $\lambda\in(1,2]$. Then there exists a positive constant $C=C(k,\lambda)$ such that for every $s,t\in\mathbb{N}_0$ with $m_{s,t}\ge C$ we have
\begin{equation}\label{maxestmstnew}
\big\|\sup_{n\ge m_{s,t}}|T_{\mathbb{Z}}\big[\Pi_{s,t,\le m_{s,t}}^{1,\chi^{(\lambda)}_n,\sqrt{\eta}}\big]f|\big\|_{\ell^2(\mathbb{Z})}\le C(s^2+t^2+1)\|f\|_{\ell^2(\mathbb{Z})}\text{.}
\end{equation}
\end{proposition}
\begin{proof}
We suppress the dependences similarly to \eqref{notationsimpler}. We begin by passing to convolution operators over $\mathbb{R}$. Firstly, we note that
\[
T_{\mathbb{Z}}\big[\Pi_{s,t,\le m_{s,t}}^{1,\chi^{(\lambda)}_n,\sqrt{\eta}}\big]f(x)=\mathcal{F}_{\mathbb{Z}}^{-1}\big[\Pi_{s,t,\le m_{s,t}}^{1,\chi^{(\lambda)}_n,\sqrt{\eta}}\cdot\mathcal{F}_{\mathbb{Z}}[f]\big](x)=\mathcal{F}_{\mathbb{Z}}^{-1}\big[\Pi_{s,t,\le m_{s,t}}^{1,\chi^{(\lambda)}_n,\sqrt{\eta}}\big]*_{\mathbb{Z}}f(x)\text{,}
\] 
and note that for $m_{s,t}\gtrsim 1$ we have $\big|\Pi_{s,t,\le m_{s,t}}^{1,\chi^{(\lambda)}_n,\sqrt{\eta}}\big|\le 1$. For $m_{s,t}\gtrsim 1$ we have that
\[
\Pi_{s,t,\le m_{s,t}}^{1,\chi^{(\lambda)}_n,\sqrt{\eta}}(\xi)=\frac{1}{\mathcal{F}_{\mathbb{R}}[\phi](0)}\sum_{(a,b,q)\in\mathcal{P}_{s,t}}\mathcal{F}_{\mathbb{R}}[\phi](\lambda^{2n}(\xi-\alpha_{a,b,q}))\eta(\lambda^{(2-\gamma')m_{s,t}}|\xi-\alpha_{a,b,q}|)^{1/2}\text{,}
\]
and it is supported in $[-1/2,1/2)$ when viewed as a function on $\mathbb{R}$. Thus by Lemma~4.4 in \cite{Bourgain}, to establish \eqref{maxestmstnew}, it suffices to prove that
\begin{equation}\label{maxestmstc}
\big\|\sup_{n\ge m_{s,t}}|\mathcal{F}_{\mathbb{R}}^{-1}\big[\Pi_{s,t,\le m_{s,t}}^{1,\chi^{(\lambda)}_n,\sqrt{\eta}}\big]*_{\mathbb{R}}f
|\big\|_{L^2(\mathbb{R})}\lesssim (s^2+t^2+1)\|f\|_{L^2(\mathbb{R})}\text{,}
\end{equation}
and note that
\[
\mathcal{F}_{\mathbb{R}}^{-1}\big[\Pi_{s,t,\le m_{s,t}}^{1,\chi^{(\lambda)}_n,\sqrt{\eta}}\big]*_{\mathbb{R}}f=T_{\mathbb{R}}\big[\Pi_{s,t,\le m_{s,t}}^{1,\chi^{(\lambda)}_n,\sqrt{\eta}}\big]f\text{.}
\]
We get
\begin{multline}
\mathcal{F}_{\mathbb{R}}^{-1}\big[\Pi_{s,t,\le m_{s,t}}^{1,\chi^{(\lambda)}_n,\sqrt{\eta}}\big]*_{\mathbb{R}}f(x)=\mathcal{F}_{\mathbb{R}}^{-1}\big[\Pi_{s,t,\le m_{s,t}}^{1,\chi^{(\lambda)}_n,\sqrt{\eta}}\mathcal{F}_{\mathbb{R}}[f]\big](x)
\\
=\mathcal{F}_{\mathbb{R}}^{-1}\Big[\sum_{(a,b,q)\in\mathcal{P}_{s,t}}\chi_n^{(\lambda)}(\cdot-\alpha_{a,b,q})\eta\big(\lambda^{(2-\gamma')m_{s,t}}|\cdot-\alpha_{a,b,q}|\big)^{1/2}\mathcal{F}_{\mathbb{R}}[f](\cdot)\Big](x)
\\
=\sum_{(a,b,q)\in\mathcal{P}_{s,t}}\mathcal{F}_{\mathbb{R}}^{-1}\Big[\chi_n^{(\lambda)}(\cdot-\alpha_{a,b,q})\eta\big(\lambda^{(2-\gamma')m_{s,t}}|\cdot-\alpha_{a,b,q}|\big)^{1/2}\mathcal{F}_{\mathbb{R}}[f](\cdot)\Big](x)
\\
=\sum_{(a,b,q)\in\mathcal{P}_{s,t}}\int_{\mathbb{R}}\chi_n^{(\lambda)}(\xi-\alpha_{a,b,q})\eta\big(\lambda^{(2-\gamma')m_{s,t}}|\xi-\alpha_{a,b,q}|\big)^{1/2}\mathcal{F}_{\mathbb{R}}[f](\xi)e(-x\xi)d\xi
\\
=\sum_{(a,b,q)\in\mathcal{P}_{s,t}}\int_{\mathbb{R}}\chi_n^{(\lambda)}(\zeta)\eta\big(\lambda^{(2-\gamma')m_{s,t}}|\zeta|\big)^{1/2}\mathcal{F}_{\mathbb{R}}[f](\zeta+\alpha_{a,b,q})e\big(-x(\zeta+\alpha_{a,b,q})\big)d\zeta
\\
=\sum_{(a,b,q)\in\mathcal{P}_{s,t}}\bigg(\int_{\mathbb{R}}\chi_n^{(\lambda)}(\xi)G^{(s,t)}_{\alpha_{a,b,q}}(\xi)e(-x\xi)d\xi\bigg) e(-\alpha_{a,b,q}x)\text{,}
\end{multline}
where 
\[
G^{(s,t)}_{\alpha_{a,b,q}}(\xi)\coloneqq \eta\big(\lambda^{(2-\gamma')m_{s,t}}|\xi|\big)^{1/2}\mathcal{F}_{\mathbb{R}}[f](\xi+\alpha_{a,b,q})\text{.}
\]
Thus
\begin{multline}
\mathcal{F}_{\mathbb{R}}^{-1}\big[\Pi_{s,t,\le m_{s,t}}^{1,\chi^{(\lambda)}_n,\sqrt{\eta}}\mathcal{F}_{\mathbb{R}}[f]\big](x)=\sum_{(a,b,q)\in\mathcal{P}_{s,t}}\bigg(\int_{\mathbb{R}}\chi_n^{(\lambda)}(\xi)G^{(s,t)}_{\alpha_{a,b,q}}(\xi)e(-x\xi)d\xi\bigg) e(-\alpha_{a,b,q}x)
\\
=\sum_{(a,b,q)\in\mathcal{P}_{s,t}}\mathcal{F}_{\mathbb{R}}^{-1}\big[\chi_n^{(\lambda)}G^{(s,t)}_{\alpha_{a,b,q}}\big](x)e(-\alpha_{a,b,q}x)=\sum_{(a,b,q)\in\mathcal{P}_{s,t}}\mathcal{F}_{\mathbb{R}}^{-1}\big[\chi_n^{(\lambda)}\big]*_{\mathbb{R}}\mathcal{F}^{-1}\big[G^{(s,t)}_{\alpha_{a,b,q}}\big](x)e(-\alpha_{a,b,q}x)\text{.}
\end{multline}
We note that
\begin{multline}
\mathcal{F}_{\mathbb{R}}^{-1}\big[\chi_n^{(\lambda)}\big](x)=\int_{\mathbb{R}}\chi_n^{(\lambda)}(\xi)e(-x\xi)d\xi=\frac{1}{\mathcal{F}_{\mathbb{R}}[\phi](0)}\int_{\mathbb{R}}\mathcal{F}_{\mathbb{R}}[\phi](\lambda^{2n}\xi)e(-x\xi)d\xi
\\
=\frac{1}{\mathcal{F}_{\mathbb{R}}[\phi](0)\lambda^{2n}}\int_{\mathbb{R}}\mathcal{F}_{\mathbb{R}}[\phi](\zeta)e(-x\zeta/\lambda^{2n})d\zeta=\frac{1}{\mathcal{F}_{\mathbb{R}}[\phi](0)}\cdot\frac{1}{\lambda^{2n}}\phi\Big(\frac{x}{\lambda^{2n}}\Big)\text{.}
\end{multline}
Thus
\begin{equation}\mathcal{F}_{\mathbb{R}}^{-1}\big[\Pi_{s,t,\le m_{s,t}}^{1,\chi^{(\lambda)}_n,\sqrt{\eta}}\mathcal{F}_{\mathbb{R}}[f]\big](x)=\frac{1}{\mathcal{F}_{\mathbb{R}}[\phi](0)}\sum_{(a,b,q)\in\mathcal{P}_{s,t}}\frac{1}{\lambda^{2n}}\phi\Big(\frac{\cdot}{\lambda^{2n}}\Big)*_{\mathbb{R}}\mathcal{F}^{-1}\big[G^{(s,t)}_{\alpha_{a,b,q}}\big](x)e(-\alpha_{a,b,q}x)\text{.}
\end{equation}
Finally, we have $|\mathcal{P}_{s,t}|\lesssim \lambda^{2s+t}$, and we note that distinct elements in $ \mathcal{P}_{s,t}$ have distance at least $\gtrsim \lambda^{-3s-t}$ by Lemma~$\ref{seperationverygeneral}$. We see that $m_{s,t}\ge 20\max(s,t)> 3s+t$. For $s,t\in\mathbb{N}_0$ with $m_{s,t}\gtrsim 1$, we have that all $\alpha_{a,b,q}\neq\alpha_{a',b',q'}\in\mathcal{P}_{s,t}$ are such that
\[
|\alpha_{a,b,q}-\alpha_{a',b',q'}|>\lambda^{-m_{s,t}}\text{.}
\] 
Thus by Lemma~4.13 in \cite{Bourgain} we get that
\begin{multline}\big\|\sup_{n\ge m_{s,t}}|\mathcal{F}_{\mathbb{R}}^{-1}\big[\Pi_{s,t,\le m_{s,t}}^{1,\chi^{(\lambda)}_n}\mathcal{F}_{\mathbb{R}}[f]\big]|\big\|_{L^2(\mathbb{R})}
\\
\lesssim_{\phi}\bigg\|\sup_{n\ge m_{s,t}}\Big|\sum_{(a,b,q)\in\mathcal{P}_{s,t}}\frac{1}{\lambda^{2n}}\phi\Big(\frac{\cdot}{\lambda^{2n}}\Big)*_{\mathbb{R}}\mathcal{F}^{-1}_{\mathbb{R}}\big[G^{(s,t)}_{\alpha_{a,b,q}}\big](x)e(-\alpha_{a,b,q}x)\Big|\bigg\|_{L_{dx}^2(\mathbb{R})}
\\
\lesssim\log(|\mathcal{P}_{s,t}|)^2\Big(\sum_{(a,b,q)\in\mathcal{P}_{s,t}}\Big\|\mathcal{F}^{-1}_{\mathbb{R}}\big[G^{(s,t)}_{\alpha_{a,b,q}}\big]\Big\|_{L^2(\mathbb{R})}^2\Big)^{1/2}\lesssim (s^2+t^2+1)\log^2(\lambda)\Big(\sum_{(a,b,q)\in\mathcal{P}_{s,t}}\Big\|G^{(s,t)}_{\alpha_{a,b,q}}\Big\|_{L^2(\mathbb{R})}^2\Big)^{1/2}
\\
\lesssim(s^2+t^2+1)\Big(\sum_{(a,b,q)\in\mathcal{P}_{s,t}}\int \eta\big(\lambda^{(2-\gamma')m_{s,t}}|\xi|\big)\big|\mathcal{F}_{\mathbb{R}}[f](\xi+\alpha_{a,b,q})\big|^2d\xi\Big)^{1/2}
\\
\le(s^2+t^2+1)\Big(\int\sum_{(a,b,q)\in\mathcal{P}_{s,t}}\eta\big(\lambda^{(2-\gamma')m_{s,t}}|\xi|\big)\big|\mathcal{F}_{\mathbb{R}}[f](\xi+\alpha_{a,b,q})\big|^2d\xi\Big)^{1/2}
\\
=(s^2+t^2+1)\Big(\int\sum_{(a,b,q)\in\mathcal{P}_{s,t}}\eta\big(\lambda^{(2-\gamma')m_{s,t}}|\xi-\alpha_{a,b,q}|\big)\big|\mathcal{F}_{\mathbb{R}}[f](\xi)\big|^2d\xi\Big)^{1/2}
\\
\le(s^2+t^2+1)\Big(\int\big|\mathcal{F}_{\mathbb{R}}[f](\xi)\big|^2d\xi\Big)^{1/2}=(s^2+t^2+1)\|f\|_{L^2(\mathbb{R})}\text{,}
\end{multline}
as desired and the proof is complete.
\end{proof}
This concludes the proof of the estimate \eqref{maxestmst}, which as we saw implies \eqref{goaleststfixed1}. 
\subsection{The estimate for small $s,t$.} To finish the proof of Proposition~$\ref{projpro}$, which implies Proposition~$\ref{goalosc}$, it remains to establish the estimate \eqref{00case}. This should be immediate with the techniques proposed since $C'$ is allowed to depend on $s,t$ and we only provide a sketch here. In all our considerations it was important to have $m_{s,t}\gtrsim 1$ to keep certain supports disjoint, if the constant can depend additionally on $s,t$ one may start trivially as follows.
\begin{multline}
\sup_{J\in\mathbb{N}}\sup_{I\in\mathfrak{S}(\mathbb{N}_0)}\Big\|O_{I,J}^2\Big(T_{\mathbb{Z}}\big[\Pi_{s,t,\le n}^{g,\widetilde{V}_{\lambda^n}}1_{n\ge m_{s,t}}\big]f:\,n\in\mathbb{N}_0\Big)\Big\|_{\ell^2(\mathbb{Z})}
\\
\lesssim \sum_{(a,b,q)\in\mathcal{P}_{s,t}}|g(a,b,q)|\Big\|O_{I,J}^2\Big(T_{\mathbb{Z}}\big[\widetilde{V}_{\lambda^n}(\cdot-\alpha_{a,b,q})\eta(\lambda^{(2-\gamma')n}\|\cdot-\alpha_{a,b,q}\|)1_{n\ge m_{s,t}}\big]f:\,n\in\mathbb{N}_0\Big)\Big\|_{\ell^2(\mathbb{Z})}
\\
\lesssim_{k,\lambda,s,t}\|f\|_{\ell^2(\mathbb{Z})}\text{,}
\end{multline}
where we have used $|\mathcal{P}_{s,t}|\lesssim \lambda^{2s+t}$. The last estimate above is immediate with the techniques used earlier, for example, one may argue in an identical manner we have in the previous sections with $\mathcal{P}_{s,t}$ replaced by $\{(a,b,q)\}$. Bourgain's logarithmic lemma is not needed here since the corresponding estimate in this case is a corollary of the $L^2$-boundedness of the Hardy--Littlewood maximal function.


\begin{thebibliography}{99}
\bibitem{frprimes}
E. Bahnson, L. Daskalakis, A. Dohadwala, I. Shah,
\emph{Pointwise Ergodic Theorems Along Fractional Powers of Primes.} International Mathematics Research Notices, Volume 2025, Issue 15, August 2025.

\bibitem{ncfull}
M. Boshernitzan, M. Wierdl,
\emph{Ergodic theorems along sequences and Hardy fields.} Proceedings of the National Academy of Sciences, 93 (1996), pp. 8205–8207. https://www.pnas.org/doi/abs/10.1073/pnas.93.16.8205.
		
\bibitem{bg2}
	       J. Bourgain,
            \emph{On the maximal ergodic theorem for certain subsets of the integers.} Israel J. Math. 61 (1988), pp. 39--72.
 
\bibitem{bg3}
            J. Bourgain,
            \emph{On the pointwise ergodic theorem on $L^p$ for arithmetic sets.} Israel J. Math. 61 (1988), pp. 73--84.

\bibitem{Bourgain}
J. Bourgain, 
\emph{Pointwise ergodic theorems for arithmetic sets.} Publications Mathématiques de L’Institut des Hautes Scientifiques 69, 5--41 (1989) https://doi.org/10.1007/BF02698838.

\bibitem{LDWT11}
L. Daskalakis,
\emph{Weak-Type (1,1) Inequality for Discrete Maximal Functions and Pointwise Ergodic Theorems Along Thin Arithmetic Sets.} J Fourier Anal Appl 30, 37 (2024).

\bibitem{CD}
C. Demeter,
\emph{On some maximal multipliers in $L^p$.} Rev. Mat. Iberoam. 26 (2010), no. 3, pp. 947--964.

\bibitem{refequid}
M. Drmota, R.F. Tichy, \emph{Sequences, Discrepancies and Applications.} Lecture Notes in Mathematics 1651, Springer-Verlag, Berlin, 1997.

\bibitem{JF}
J. Fornal,
\emph{Pointwise ergodic theorem along primes of the form $x^2 + ny^2$.} Preprint: arXiv:2508.15466.


\bibitem{Graf}
L. Grafakos,
\emph{Classical Fourier Analysis.} Vol. 249 of Graduate Texts in Mathematics, third edition, Springer.

\bibitem{BG}
B. J. Green, T. C. Tao,
\emph{The quantitative behaviour of polynomial orbits on nilmanifolds.} Ann. Math. 175, no. 2 (2012): 465--540.

\bibitem{IWT}
A. D. Ionescu, S. Wainger,
\emph{$L^p$ boundedness of discrete singular Radon transforms.}
J. Amer. Math. Soc. 19 (2005), no. 2, pp. 357--383.

\bibitem{Lb1}
A. Leibman,
\emph{Pointwise convergence of ergodic averages for polynomial sequences of translations on a nilmanifold.} Ergodic Theory and Dynamical Systems 25 (2005), 201--213.

\bibitem{Lb2}
A. Leibman,
\emph{Pointwise convergence of ergodic averages for polynomial actions of $\mathbb{Z}^d$ by
 translations on a nilmanifold.} Ergodic Theory and Dynamical Systems 25 (2005), 215--225.

\bibitem{Lb3}
A. Leibman,
\emph{A canonical form and the distribution of values of generalized polynomials.}  Israel Journal of Mathematics 188 (2012), 131--176.

\bibitem{LV}
J. Liouville,
\emph{Nouvelle demonstration d'un th\'eor\`eme sur les irrationnelles alg\'ebriques ins\'er\'e dans le compte rendu de la derni\`ere s\'eance.} C.R. Acad. Sci. Paris 18 (1844), 910--911.
	
\bibitem{mirek1}
M. Mirek,
\emph{Weak type (1, 1) inequalities for discrete rough maximal functions.} J. Anal. Mat. 127 (2015), 303–337.

\bibitem{MOE}
M. Mirek, T. Z. Szarek, J. Wright,
\emph{Oscillation inequalities in ergodic theory and analysis: one-parameter and multi-parameter perspectives.} Rev. Mat. Iberoam. 38 (2022), no. 7, 2249--2284.

\bibitem{primeevaluatedpolyn}
R. Nair,
\emph{On polynomials in primes and J. Bourgain's circle method approach to ergodic theorems II}
Studia Mathematica 105 (1993), 207--233.

\bibitem{NL}
	V. R. Neale,
	\emph{Bracket quadratics as asymptotic bases for the natural numbers.}  Dissertation, University of Cambridge, 2011.
\end{thebibliography}
\end{document}